\documentclass[twocolumn]{IEEEtran}

\usepackage[T1]{fontenc}
\usepackage{algorithm}
\usepackage{algpseudocode}
\usepackage{amsmath}
\usepackage{amssymb}
\usepackage{amsthm}
\usepackage{array}
\usepackage{balance} 
\usepackage{bbm}
\usepackage{bm}
\usepackage{booktabs}
\usepackage{caption}
\usepackage{cite}
\usepackage{color}
\usepackage{courier}
\usepackage{enumerate}
\usepackage{epsfig,epstopdf}
\usepackage{float}
\usepackage{graphicx}
\usepackage{multirow}
\usepackage{pseudocode}
\usepackage{subcaption}
\usepackage{soul,xcolor}
\usepackage{url}
\usepackage{verbatim}
\usepackage{yfonts,color}

\newtheorem{theorem}{Theorem}
\newtheorem{definition}[theorem]{Definition}
\newtheorem{lemma}[theorem]{Lemma}

\newtheorem{corollary}[theorem]{Corollary}
\newtheorem{assumption}[theorem]{Assumption}

\newcommand{\m}{m}
\newcommand{\q}{R}
\newcommand{\ITERNDT}{j}
\newcommand{\RANGE}{\delta}
\newcommand{\ITERR}{s}
\newcommand{\ITERRT}{s'}

\newcommand{\LIAP}{V}
\newcommand{\RTUQ}{\lambda}
\newcommand{\SC}{\mu}
\newcommand{\STEP}{\gamma}
\newcommand{\QNSB}{\eta}
\newcommand{\QNSC}{\omega}
\newcommand{\RTQ}{\sigma}
\newcommand{\secref}[1]{Sec.~\ref{#1}}
\newlength\myindent
\def\thesubsectiondis{\arabic{subsection}.} 

\begin{document}    
\title{Finite-Bit Quantization For Distributed Algorithms With Linear Convergence}
\author{Nicol\`{o} Michelusi, Gesualdo Scutari, and Chang-Shen Lee
	\thanks{N. Michelusi is with the School of Electrical, Computer and Energy Engineering, Arizona State University, AZ, USA; email: nicolo.michelusi@asu.edu.}
	\thanks{G. Scutari is with the School of Industrial Engineering, Purdue University, IN, USA; email: gscutari@purdue.edu.}
	\thanks{C.-S. Lee is with Bloomberg, New York, USA; email: ben3003neb@gmail.com.} 
	\thanks{Michelusi's work was supported in part by the National Science Foundation under grant CNS-2129015. The work of Scutari and Lee was  partially supported by the National Science Foundation under the grant  CIF 1719205; the Army Research Office under the grant  W911NF1810238; and the Office of Naval Research, under the grant N00014-21-1-2673.}}
\maketitle

\begin{abstract}
This paper studies distributed algorithms for (strongly convex) composite optimization problems  over mesh networks, subject to quantized communications.
Instead of focusing on a specific algorithmic design, a black-box model is proposed, casting linearly convergent distributed algorithms in the form of  fixed-point iterates.
The algorithmic model is equipped  with a novel random or deterministic
 \emph{Biased Compression} (BC) rule    on the quantizer design, and
  a new \emph{Adaptive encoding Non-uniform Quantizer} (ANQ)
 coupled with a communication-efficient encoding scheme,
 which  implements the BC-rule using a {\it finite} number of bits (below machine precision).  
    This fills a gap existing in most state-of-the-art
  quantization schemes, such as those based on the popular  compression rule, which rely on communication of some scalar signals
  with {negligible quantization} error (in practice quantized at the machine precision).
 A unified communication  complexity analysis is developed  for the black-box model, determining   the average number of bits required to reach a solution of the optimization problem within a target accuracy. It is shown that the proposed BC-rule preserves   {\it linear} convergence of the unquantized algorithms, and a trade-off between convergence rate and communication cost under {ANQ-based} quantization is characterized.
Numerical results validate our theoretical findings and show that distributed algorithms equipped with the proposed ANQ have more favorable communication cost than algorithms using state-of-the-art quantization rules.
\end{abstract}

\section{Introduction} \label{sec:intro}
We study distributed optimization  over a network of $m$ agents modeled as an undirected (connected) graph. We consider  mesh  networks, that is, arbitrary topologies with no central hub connected to all the other agents, where each agent can communicate with its immediate neighbors (master/worker architectures can be treated as a special case). The $m$ agents aim at solving cooperatively the optimization problem
\begin{equation}\label{eq:P}
\min_{{\bf x} \in \mathbb R^d} \quad \underbrace{\frac{1}{\m}\sum_{i=1}^\m f_i({\bf x})}_{=F(\mathbf{x})}+r(\mathbf{x}), \tag{P}  
\end{equation}
where each $f_i$ is the local cost function of agent $i$, assumed to be smooth, convex,  and known only to the agent;  $r:\mathbb{R}^d\to [-\infty, \infty]$ is a non-smooth, convex (extended-value) function  {known to all agents}, which may be used to force shared constraints or some structure on the solution (e.g., sparsity); and the  {global}   loss $F:\mathbb{R}^d\to \mathbb{R}$ (or in some cases each local loss $f_i$) is assumed to be strongly convex on the domain of $r$. This setting is fairly general and finds applications in several areas, including network information processing, telecommunications, multi-agent control, and machine learning (e.g., see 
 \cite{Molzahn2017,Nedic2018,Yang2019}).

Since the functions $f_i$ can only be accessed locally and routing local data to other agents is infeasible or highly inefficient, solving (\ref{eq:P}) calls for the design of distributed algorithms that alternate between a local computation procedure at each agent's side and some  rounds of  communication among neighboring nodes. While most existing works  focus on \emph{ad-hoc} solution methods, here we consider a \emph{general} distributed algorithmic framework,
encompassing algorithms whose dynamics are modeled by the  fixed-point iteration
\begin{align}\label{eq:fixed-point}
{\bf z}^{k+1} = \tilde{\mathcal A}\big({\bf z}^{k}\big),
\end{align}
where ${\bf z}^{k}$ is the updating variable at iteration $k$ and  $\tilde{\mathcal A}$  is a mapping that embeds the local computation and communication steps, whose fixed point  {typically}   coincides with solutions of (\ref{eq:P}).  This model encompasses several distributed algorithms over different network architectures, each one corresponding to a specific expression  of
$\mathbf z$ and $\tilde{\mathcal A}$--see \secref{Sec:framework}  for some  examples.

By assuming that $F$ is strongly convex,   we explicitly target distributed schemes in the form   (\ref{eq:fixed-point}) that converge to solutions of (\ref{eq:P}) at {\it linear rate}.
 Furthermore, since the cost of communications is often the bottleneck for distributed computing when compared with local 
 (possibly parallel) computations  (e.g., \cite{Bekkerman_book11, Lian17}), we achieve communication efficiency by embedding the  iterates (\ref{eq:fixed-point}) with {quantized} communication protocols. 
 
 Quantizing communication steps of distributed optimization algorithms have received significant  attention in recent years--see Sec.~\ref{sec:literature} for a comprehensive overview of the state-of-the-art.  {Here we only point out that existing distributed algorithms   over {\it mesh networks} are applicable only to {\it unconstrained} instances of (\ref{eq:P}) and  {\it smooth} objective functions (i.e., $r\not\equiv 0$). Furthermore, 
 even in these special instances,
 such schemes   rely on quantization rules that subsume  some  scalar signals to be encoded with   negligible  error--the typical example is the renowned {\it compression} protocol \cite{Koloskova2019}. While in practice this is achieved by quantizing at the {\it machine precision} (e.g., $32$ or $64$ bit floating-point), on the theoretical side, existing convergence analyses become elusive.
  In such cases, the assessment of their convergence is left to simulations.}

The     goal of this paper is to design a  black-box quantization mechanism for the  class of distributed algorithms (\ref{eq:fixed-point}) applicable to Problem \eqref{eq:P} over mesh networks that preserves their linear convergence while employing communications quantized with {\it finite-bit}  {(below machine precision)}.   As discussed next, this is an open problem.

\subsection{Summary of main contributions} \label{subsec:contribution}
In a nutshell, we summarize our major contributions as follows.

\noindent $\bullet$ \textbf{A black-box quantization model for (\ref{eq:fixed-point}):}
We propose a novel black-box model that introduces quantization in the communication steps  of     linearly convergent  distributed algorithms cast in the form  \eqref{eq:fixed-point}.  Our approach paves the way to a  {\it unified} design of quantization rules and analysis of their impact on the convergence rate of a gamut of distributed algorithms. This constitutes a major departure from the majority of existing studies focusing on ad-hoc algorithms and quantization rules, which in fact are   special instances  of our framework. Furthermore, our model brings for the first time quantization to distributed algorithms applicable to composite  optimization problems (i.e., \eqref{eq:P} with $r\not \equiv 0$). This is particularly relevant in machine learning applications, where empirical risk minimization problems call for (non-smooth) regularizations or constraints to control the complexity of the solution (e.g., to enforce sparsity or low-rank structure) and avoid overfitting.

\noindent $\bullet$ \textbf{Preserving linear convergence  of (\ref{eq:fixed-point})  under quantization:}  To enable quantization   below the machine precision, we provide  a novel
{\it biased compression} rule (the BC-rule)
 on the quantizer design equipping the proposed  black-box model, which allows to preserve {\it linear} convergence of the distributed algorithms while using a {\it finite} number of bits  {(below machine precision)} and  {\it without} altering their original tuning. Our condition encompasses several deterministic and random  quantization rules as special cases, new and old \cite{Kashyap2007, Aysal2008, Nedic2009, Kar2010, Li2011, Rajagopal2011, Lavaei2012, Thanou2013, Zhu2015, ElChamie2016, Li2017, Lee2017conf, Lee2018CDC, Lee2018QNEXTconf, Kajiyama2020, Magnusson2020, Lee2020}. 

\begin{table*}[t]
\footnotesize
\centering
	\begin{tabular}{| >{\centering\arraybackslash}m{1.4in}| >{\centering\arraybackslash}m{0.9in}
	| r 
	c
	l |}
    \hline
    \textbf{\bf Algorithm} &\textbf{\bf Problem, network} & \textbf{[Bits/agent/iteration]} 
    \!\!\!\!\!\!& $\!\!\!\times\!\!\!$ &\!\!\!\!\!\!
    \textbf{[Iterations to $\varepsilon$-accuracy]}  \\\hline
    \hline
    \multicolumn{5}{|c|}{{\bf State-of-the-art Quantized Distributed Algorithms}}
  \\ \hline
  Q-GD over star \cite{Magnusson2020}
    & \eqref{eq:P} with $r\equiv 0$, star & ${\mathcal O} \Big(d\log\big(1+\kappa d\big)$
    \!\!\!\!\!\!& $\!\!\!\times\!\!\!$ &\!\!\!\!\!\!
    $\kappa\log (d/\varepsilon)\Big)$
    \\\hline
    Q-NEXT \cite{Lee2018QNEXTconf,Kajiyama2020} & \eqref{eq:P} with $r\equiv 0$, mesh  & &\!\!\!\!\!\!N/A\!\!\!\!\!\!& \\ \hline
    Q-Primal-Dual \cite{Magnusson2020}
    & \eqref{eq:P} with $r\equiv 0$, mesh & 
     &\!\!\!\!\!\!N/A\!\!\!\!\!\!&
    \\\hline
    Primal-Dual (rand-$K$ compr., Option-D)~\cite{Kovalev2020}**
    & \eqref{eq:P} with $r\equiv 0$, mesh & 
    $\mathcal O\Big(d(B+\log d)$
\!\!\!\!\!\!& $\!\!\!\times\!\!\!$ &\!\!\!\!\!\!
     $\frac{\kappa}{1-\rho_2}\log(d/\varepsilon)\Big)$
     \\\hline
    Primal-Dual (Option-D)~\cite{Kovalev2020},
    LEAD  \cite{Liu2021}  (both dit-$K$ compr.)**
    & \eqref{eq:P} with $r\equiv 0$, mesh & 
    $\mathcal O\Big((d+B)$
\!\!\!\!\!\!& $\!\!\!\times\!\!\!$ &\!\!\!\!\!\!
     $\frac{\kappa}{1-\rho_2}\log(d/\varepsilon)\Big)$
    \\\hline
    \hline
    \multicolumn{5}{|c|}{{\bf ANQ-embedded Distributed Algorithms (this work)}}
  \\ \hline
    GD over star & \eqref{eq:P} with $r\equiv 0$, star
& ${\mathcal O} \Big(d\log(1+\kappa)$
\!\!\!\!\!\!& $\!\!\!\times\!\!\!$ &\!\!\!\!\!\!
$\kappa\log\big(d/\varepsilon\big) \Big)$ \\ \hline
    (Prox-)EXTRA, (Prox-)NIDS
     & \eqref{eq:P}, mesh
& ${\mathcal O} \Big(d\log\Big(\max\Big\{1+\kappa, \frac{1}{1- {\rho_2}}\Big\}\Big)$
\!\!\!\!\!\!& $\!\!\!\times\!\!\!$ &\!\!\!\!\!\!
$\max\Big\{\kappa, \frac{1}{1- {\rho_2}}\Big\}\log (d/\varepsilon)\Big)$ \\ \hline
NIDS
     & \eqref{eq:P} with $r\equiv 0$, mesh
& $\mathcal O \bigg(d\log\Big(\max\Big\{1+\kappa,\frac{1}{1-\rho_{2}} \Big\}\Big)$
\!\!\!\!\!\!& $\!\!\!\times\!\!\!$ &\!\!\!\!\!\!
$\max\Big\{\kappa,\frac{1}{1-\rho_{2}} \Big\}
\log(d/\varepsilon)
\bigg)$ \\ \hline
(Prox-)NEXT, (Prox-)DIGing
     & \eqref{eq:P}, mesh
& ${\mathcal O} \Big(d\log\Big(\max\Big\{1+\kappa, \frac{1}{(1- {\rho_2})^2}\Big\}\Big)$
\!\!\!\!\!\!& $\!\!\!\times\!\!\!$ &\!\!\!\!\!\!
$\max\Big\{\kappa, \frac{1}{(1- {\rho_2})^2}\Big\}\log (d/\varepsilon)\Big)$ \\ \hline
Primal-Dual
     & \eqref{eq:P} with $r\equiv 0$, mesh
& $\mathcal O \bigg(d\log\Big(1+\frac{\kappa}{1-\rho_2}\Big)$
\!\!\!\!\!\!& $\!\!\!\times\!\!\!$ &\!\!\!\!\!\!
$\frac{\kappa}{1-\rho_2}
\log(d/\varepsilon)
\bigg)$ \\ \hline
    \end{tabular}
    \caption{
    Comparison of the communication complexity, 
    expressed as number of bits/agent/iteration times the number of iterations to achieve $\varepsilon$-accuracy,
    of the proposed ANQ applied to state-of-the-art distributed optimization algorithms solving various instances of \eqref{eq:P} over star or mesh networks;
    $\kappa$ is the condition number of each $f_i$,
    $d$ is the dimension of $\bf x$,
    $\rho_2$ is the second largest eigenvalue of the gossip matrix ${\bf W}$ used in these algorithms  (note that $\rho_2=0$ for star-networks or fully-connected graphs). See also Corollaries \ref{coro:comm_cost_star}-\ref{coro:comm_cost_PD}. \\ **: the complexities of these schemes are adapted from \cite{Kovalev2020}; $\rho$ and $\rho_\infty$ in
    \cite{Kovalev2020} are both $\mathcal O(1/(1-\rho_2))$;
    the compression rate is chosen as $\Theta(1)$ to match the iterations to $\varepsilon$-accuracy of the ANQ-Primal-Dual, hence $K=\Theta(d)$ for rand-$K$ and
    $S=2^{K-1}-1=\Theta(\sqrt{d})$ for dit-$K$; 
    $B$ is the number of bits used to encode each scalar \emph{with negligible loss in precision} (as required by the convergence theory therein).
    \label{table:related_work} }
\end{table*} 

\noindent $\bullet$ \textbf{A novel finite-bit quantizer:} To implement the BC-rule, we also propose  a novel \emph{finite-bit} quantizer  {(below machine precision)} fulfilling the BC-rule along with a communication-efficient  bit-encoding/decoding rule which enables transmissions on   digital channels.
The resulting \emph{Adaptive encoding Non-uniform Quantizer} (ANQ) adapts the number of bits of the output (discrete representation) based upon the input signal. By doing so, it achieves a more communication-efficient design than existing quantizers that adopt a fixed number of bits
based on a predetermined fixed \cite{Kashyap2007, Aysal2008, Nedic2009, Kar2010, Rajagopal2011, Lavaei2012, Zhu2015, ElChamie2016, Lee2017conf, Lee2018CDC, Lee2020} or shrinking range \cite{Li2011, Thanou2013, Li2017, Lee2018QNEXTconf, Magnusson2020, Kajiyama2020} of the input signal,
or that rely on transmissions of scalar values at machine precision
\cite{Alistarh2017, Alistarh2018, Karimireddy2018, Stich2018, Tang2018, Koloskova2019, Zhang2019, Zheng2019, Beznosikov2020, Gorbunov2020, Kovalev2020, Stich2020, Taheri2020, Haddadpour2021, Liao2021, Liu2021, Zhang2021}.

\noindent $\bullet$ \textbf{Communication complexity:}  We derive a unified complexity analysis for
{\it any} distributed algorithm belonging to our black-box model, solving \eqref{eq:P} (possibly with $r\not \equiv 0$)
 over {\it mesh networks}. Specifically, we prove that, under suitable conditions and proper tuning (see Theorem \ref{thm:comm_cost_mesh}),  an  $\varepsilon$-solution of (\ref{eq:P}) (using a suitable optimality measure) is achieved by any of such distributed algorithms in 
 $${\mathcal O} \Big(\frac{1}{1-\RTUQ}   \log( d/\varepsilon)\Big)\quad \text{iterations},$$
  using  
$${\mathcal O} \Big(d\log\Big(1+\frac{1}{1-\RTUQ}\Big)\Big)\quad \text{bits/agent/iteration},$$
where  $\RTUQ\in (0,1)$ is the convergence rate of the distributed algorithm our quantization method is applied to, which depends on the condition number of the agents' function as well as network parameters.
Table~\ref{table:related_work} customizes the above result   to a variety of  state-of-the-art  distributed schemes subject to quantization, using the explicit expression of $\lambda$ (see Table~\ref{table:lambda} and  Corollaries~\ref{coro:comm_cost_star}-\ref{coro:comm_cost_PD}). This permits for the first time to benchmark several distributed algorithms under quantization,
{and compare them with other state-of-the-art quantization schemes (Table~\ref{table:related_work})}. Notice that the proposed methods compare favorably with existing ones and, remarkably, they achieve $\varepsilon$-accuracy with the same number of iterations (in a $\mathcal O$-sense) as their unquantized counterparts.

\noindent $\bullet$ \textbf{Numerical evaluations:}   
Finally, we extensively validate our theoretical findings on smooth and non-smooth regularized linear and logistic regression problems. Among others, our evaluations show that 1) linear convergence of all distributed algorithms  is preserved under finite-bit quantization based upon the proposed BC-rule;  2) as predicted by our analysis,   the rate approaches that of their unquantized counterpart  {(implemented at machine precision)} when a sufficient  number of bits is used; 3)  the proposed ANQ quantization outperforms, in terms of both convergence rate and communication cost, existing quantization schemes operating 
below machine precision--i.e., Q-Dual \cite{Magnusson2020} and Q-NEXT \cite{Kajiyama2020}--or at
machine precision--e.g., those relying on the conventional compression rule, such as  LEAD \cite{Liu2021} and COLD \cite{Zhang2021}.

\subsection{Related works}\label{sec:literature}
The literature on distributed algorithms is vast; here, we review relevant works employing some form of quantization
with linear convergence guarantees \cite{Tang2018, Taheri2020, Beznosikov2020, Kovalev2020, Lee2018QNEXTconf, Kajiyama2020, Magnusson2020, Liao2021},
categorized into those relying on machine precision quantization and those relying on  limited precision quantization.

\noindent{\bf 1) Machine precision quantization schemes \cite{Tang2018, Taheri2020, Beznosikov2020, Kovalev2020, Liao2021}:}
Distributed algorithms employing quantization in the agents' communications are proposed in \cite{Tang2018, Taheri2020, Beznosikov2020, Kovalev2020, Liao2021} for special instances of \eqref{eq:P} with $r\equiv 0$ (i.e., smooth and unconstrained optimization).  {In these schemes,  quantization is implemented by compressing the signal ${\bf x}\in \mathbb{R}^d$ 
 through   a (random or deterministic\footnote{We treat compression rules using deterministic mappings $\mathcal{Q}^k$ as special cases of the random ones; in this case, the expected value operator will just return the deterministic value argument.})  \emph{compression operator} $\mathbf{x}\mapsto \mathcal{Q}(\mathbf{x})$, that satisfies
  the \emph{compression rule} 
\begin{align}
\sqrt{\mathbb E[\Vert\mathcal Q({\bf x}) - {\bf x} \Vert_2^2]}\leq \QNSC \Vert {\bf x}\Vert_2 ,\quad \text{for some}\quad  \QNSC>0. \label{prob_compression_rule}
\end{align}
This rule subsumes the transmission of some scalar signals with {negligible} quantization errors,
e.g., the norm of $\mathbf x$,
requiring thus in theory infinite machine precision  (see Corollary~\ref{coro:converse} for a formal proof).   While quantizing at the  machine precision is a viable strategy in practice,  convergence guarantees of  distributed algorithms relying on  (\ref{prob_compression_rule}) become elusive--they  are established only under {negligible} quantization errors of the transmitted norm signal.
By generalizing the conventional compression rule, the proposed BC-rule overcomes this theoretical limitation and  permits to  explicitly model the communication cost of quantized communications below machine precision.}
As we will demonstrate numerically in Sec.\ref{sec:simulation}, our explicit formulation of the communication cost allows to define more communication efficient quantization schemes than state-of-the-art algorithms relying on machine precision.
 
\noindent{\bf 2) Limited precision quantization schemes \cite{Lee2018QNEXTconf, Kajiyama2020, Magnusson2020}:} While finite-rate quantization   has been extensively studied for average consensus schemes  (e.g.,  {\cite{Aysal2008, Nedic2009, Li2011, Rajagopal2011, Thanou2013, Lee2017conf, Li2017, Lee2018CDC, Lee2020}}), their extension to optimization algorithms over mesh networks is less explored \cite{Lee2018QNEXTconf, Kajiyama2020, Magnusson2020}. Specifically, in our prior work \cite{Lee2018QNEXTconf}, we equipped the NEXT algorithm \cite{Lorenzo2016, Sun2019} with a finite-bit deterministic quantization to solve \eqref{eq:P} with $r\equiv 0$; 
to preserve linear convergence,
the quantizer   shrinks   its input range linearly.  
 An  expression of the convergence rate of the scheme in \cite{Lee2018QNEXTconf} has been  later determined in \cite{Kajiyama2020} along with its scaling properties  with respect to problem, network, and quantization parameters.  

The closest paper to our work is \cite{Magnusson2020}, where the authors proposed  a finite-bit quantization mechanism preserving linear convergence of  some  algorithms
 {cast as \eqref{eq:fixed-point}.}
 Yet, there are several key differences between \cite{Magnusson2020} and our work. First, the  convergence analysis in \cite{Magnusson2020} is applicable only to
algorithms solving
smooth, unconstrained optimization problems, and thus not to  Problem~\eqref{eq:P} with $r\not\equiv  0$. Second, linear convergence under finite-bit quantization is explicitly proved in \cite{Magnusson2020}  only for schemes whose updates utilize current iterate information, namely: gradient descent (GD) over star networks and the  Primal-Dual algorithm in \cite{Uribe2021} over mesh networks. This  leaves open the question of whether   distributed algorithms using historical information--e.g., in the form of gradient tracking or dual variables--are linearly convergent under finite-bit quantization, and under which conditions;  renowned  examples include  EXTRA \cite{Shi2015EXTRA}, AugDGM \cite{Xu2018}, DIGing \cite{Nedic2017}, Harnessing \cite{Qu2018}, and NEXT \cite{Lorenzo2016, Sun2019}. Our work provides a positive answer to these open questions.   Third,  the communication complexity of the  scheme in \cite{Magnusson2020}  is
only provided for star networks, not for mesh networks, which instead is a novel contribution of this work for a wide class of distributed algorithms--see  Table~\ref{table:related_work} and {\secref{sec:commcost}}. Fourth, \cite{Magnusson2020} proposed an  {ad-hoc deterministic} quantization rule while the proposed  BC-rule encompasses several deterministic and random quantizations (including that in \cite{Magnusson2020} as a special case), possibly using a variable  number of bits (adapted to the input signal). As a result, even  when customized to the setting/algorithms in \cite{Magnusson2020}, the  BC-rule leads to more communication-efficient schemes,  both analytically (see \secref{sec:commcost}) and numerically (see \secref{sec:simulation}). 

\subsection{Organization and notation}
The remainder of this paper is organized as follows. \secref{Sec:framework} introduces the proposed black-box model, which casts distributed algorithms in the form~\eqref{eq:fixed-point}. 
\secref{Sec:framework_Q} embeds quantized  communications, introduces the proposed BC-rule, and analyzes the convergence properties. \secref{sec:pro_quant} describes the proposed quantizer, the ANQ, and studies its communication cost. \secref{sec:commcost} investigates the communication complexity, and customizes  the proposed framework and convergence guarantees to  several existing distributed algorithms, equipping them   with the ANQ rule. \secref{sec:simulation} provides some numerical results, while 
\secref{Sec:conclusion}
 draws concluding remarks. All the  proofs of our  results are  presented in the appendix.
\begin{figure}[t]
		\centering
		\includegraphics[width = 2.5 in]{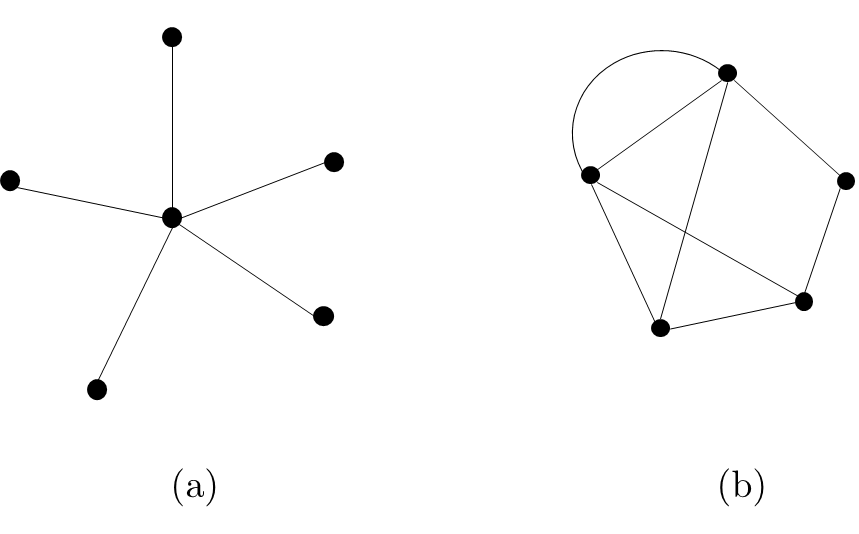}
		\caption{Examples of star network (a)  versus mesh topology (b).} \label{fig:illust_net}
\end{figure}

\textit{Notation:}
Throughout the paper, we will use the following notation. {We denote by $\mathbb Z,\mathbb R$ the set of integers and real numbers, respectively.} For any positive integer $a$,  we define $[a] \triangleq \{1,\cdots, a\}$. 
  We denote by
${\bf 0}, {\bf 1}$, and ${\bf I}$  the  vector of all zeros, the vector of all ones, and the identity matrix, respectively (of appropriate dimension).    For vectors ${\bf c}_1, \cdots, {\bf c}_{\m}$ and a set $\mathcal S \subseteq [\m]$, define ${\bf c}_\mathcal S \triangleq \{{\bf c}_i: i \in \mathcal S\}$. We use $\|\cdot\|$ to denote a norm in the Euclidean space (whose dimension will be clear from the context); when a specific norm is used, such as $\ell_2$  or $\ell_{\infty}$, we will append the associated subscript to $\|\cdot\|$. The $i$th eigenvalue of
 {a real, symmetric} 
matrix $\bf G$ is denoted by $\rho_i({\bf G})$,
ordered in non-increasing order such that $
\rho_{1}({\bf G})\geq\dots
\rho_{i}({\bf G})\geq\rho_{i+1}({\bf G})$. 
We will use a superscript to denote iteration counters of sequences generated by the algorithms, for instance, ${x}^k$ will denote the value of the $x$-sequence at iteration $k$; we will instead use $(x)^k$ for the  $k$th-power of $x$.
Finally,  asymptotic behaviors of functions are captured by the standard  big-$\mathcal O, \Theta$, and $\Omega$ notations: 1) $g(x) = \mathcal O(h(x))$ as $x \to x_0$ if and only if $\limsup_{x \to x_0} \vert g(x)/h(x) \vert\in [0, \infty)$; 2) $g(x) = \Omega(h(x))$ if and only if $h(x) = \mathcal O(g(x))$; and 3) $g(x) = \Theta(h(x))$ if and only if $g(x) = \mathcal O(h(x))= \Omega(h(x))$.

We   model  a network of $\m$ agents  as a fixed, undirected, connected graph $\mathcal {G = (V,E)}$, where $\mathcal V = [\m]$
 is the set of vertices (agents) and $\mathcal {E \subseteq V \times V}$ is the set of edges (communication links); $(i, j) \in \mathcal E$ if there is a link between agents $i$ and $j$, so that the two can send information to each other. We let $\mathcal N_i{=}\{j:(i,j)\in \mathcal E\}$ be the set of neighbors of agent $i$, and assume that $(i,i) \in \mathcal E$, i.e., $i \in \mathcal N_i$. Master/workers  architectures  will be considered as special instances--see Fig.~\ref{fig:illust_net}.

\section{A General Distributed Algorithmic Framework: Exact Communications} \label{Sec:framework}
In this section, we cast   distributed algorithms to solve \eqref{eq:P} in the form (\ref{eq:fixed-point}). As a warm-up, we begin with schemes using only current information to produce the next update (cf. \secref{sec:single-round}). 
We then generalize the model to capture distributed algorithms using historical information  via multiple rounds of communications between computation steps (cf. \secref{sec:multiple-round}).

\subsection{Warm-up: A class of distributed algorithms}\label{sec:single-round}

We cast distributed algorithms in the form (\ref{eq:fixed-point}) by incorporating computations and communications as two separate steps. We use state variable $\mathbf{z}_i$ to capture local information owned by agent $i$ (including  optimization variables)  and $\hat{\bf c}_i$ to denote the signal transmitted by agent $i$ to its neighbors.\footnote{Dimensions of these vectors are   algorithm-dependent and omitted for simplicity, and will be clear from the context.}     
Similarly to \cite{Magnusson2020}, the updates of the $ z,\hat{c}$-variables read: for agent $i\in [m]$,
\begin{align}
\begin{array}{*{20}l}
\hat{\bf c}_{i}^{k} &= \mathcal C_{i}\big({\bf z}_{i}^{k}\big), & \quad \text{(communication step)}  \\
{\bf z}_i^{k+1} &= \mathcal A_i\big({\bf z}_i^{k},  \hat{\bf c}_{{\mathcal N}_i}^{k}\big), & \quad \text{(computation step)}
\end{array}
\tag{M0} \label{eq:M0}
\end{align}
where the function $\mathbf{z}_i\mapsto  \mathcal C_i(\mathbf{z}_i)$ models the processing on the local information $\mathbf{z}_i^k$ at the current iterate,   generating the signal $\hat{\bf c}_{i}^{k}$ transmitted to agent $i$'s neighbors; the function $({\bf z}_i,  \hat{\bf c}_{{\mathcal N}_i})\mapsto \mathcal A_i({\bf z}_i,  \hat{\bf c}_{{\mathcal N}_i})$ produces the update of the  agent $i$'s state variable $\mathbf{z}_i$, based upon the  local information at iteration $k$, and the signals received by its neighbors in $\mathcal N_i$.

\textit{Some examples:}
The algorithmic model \eqref{eq:M0}    captures a variety of distributed algorithms that build updates using single rounds of communications; examples include the renewed    DGD \cite{Nedic2009grad}, NIDS \cite{Li2019}, and the Primal-Dual scheme \cite{Uribe2021}. To show a concrete example, consider   DGD, 
which aims at solving a special instance of \eqref{eq:P} with $r\equiv 0$; agents' updates read
\begin{align*}
{\bf x}_i^{k+1} = \Big(\sum_{j=1}^m w_{ij}
{\bf x}_j^{k}\Big)- \STEP \nabla f_i\big({\bf x}_i^{k}\big),\quad i\in [m],
\end{align*}
where ${\bf x}_i^{k}$ is the local copy owned by agent $i$ at iteration $k$ of the optimization variables $\bf x$, $\gamma$ is a step-size,
and $w_{ij}$'s are nonnegative weights properly chosen and compliant with  the  graph  $\mathcal{G}$ (i.e., $w_{ij} > 0$ if $(i,j) \in \mathcal E$; and  $w_{ij} =0$ otherwise). It is not difficult to check that DGD can be cast in the form \eqref{eq:M0} 
by letting \begin{align*}
\begin{cases}
    {\bf z}_i^{k} =\hat{\bf c}_i^{k}= {\bf x}_i^{k}\\\mathcal A_i({\bf z}_i^{k}, \hat{\bf c}_{{\mathcal N}_i}^{k})
=\Big(\sum_{j=1}^m w_{ij}
\hat{\bf c}_j^{k}\Big)- \STEP \nabla f_i({\bf z}_i^{k}).
\end{cases}
\end{align*}

Despite its generality, model \eqref{eq:M0} leaves out several important   distributed algorithms, specifically, the majority of schemes employing correction of the gradient direction based on past state information--these are the best performing algorithms to date. Examples include EXTRA {\cite{Shi2015EXTRA}},    DIGing {\cite{Nedic2017}} and their proximal version,  NEXT/SONATA {\cite{Lorenzo2016, Sun2019,Sun2019Math}},  and the ABC framework  \cite{Xu2020}, just to name a few. Consider, for instance, NEXT/SONATA: 
\begin{align}
\begin{cases}
{\bf x}_i^{k+1} = \sum_{j \in \mathcal N_i} w_{ij} \big({\bf x}_j^{k} - \STEP {\bf y}_j^{k}\big)\\
{\bf y}_i^{k+1} = \sum_{j \in \mathcal N_i} w_{ij} \big({\bf y}_j^{k} + \nabla f_j({\bf x}_j^{k+1}) - \nabla f_j({\bf x}_j^{k})\big). 
\end{cases}
\label{eq:NEXT}
\end{align}
Clearly, these updates do not fit model \eqref{eq:M0}: the update of the $y$-variable uses information from two iterations ($k$ and  $k+1$).  
 This  calls for a more general model, introduced  next.

 \begin{table*}[t]
	\small\centering
	\begin{tabular}{| >{\centering\arraybackslash}m{2in} | >{\centering\arraybackslash}m{0.9in} | >{\centering\arraybackslash}m{0.5in} | >{\centering\arraybackslash}m{2in}|}    
    \hline
    \textbf{\bf Algorithm} & \textbf{Problem} & \textbf{Network} & \textbf{Convergence rate $\RTUQ$}  \\ \hline
    GD over star networks \cite{Nesterov2014} & \eqref{eq:P} with $r\equiv 0$ & Star & $\frac{\kappa-1}{\kappa+1}$\\ \hline
    (Prox-)EXTRA \cite{Xu2020} & \eqref{eq:P} & Mesh &{$\max\Big\{\frac{\kappa}{\kappa+1}
, \sqrt{\rho_{2}\big({\bf W}\big)}\Big\} $}\\ \hline
        (Prox-)NIDS \cite{Xu2020} & \eqref{eq:P} & Mesh &$\max\Big\{\frac{\kappa-1}{\kappa+1}, \sqrt{\rho_2\big({\bf W}\big)}\Big\}$\\ \hline
    NIDS \cite{Li2019} & \eqref{eq:P} with $r\equiv 0$ & Mesh & $\max\Big\{\sqrt{1-\kappa^{-1}}, \sqrt{\rho_{2}({\bf W})} \Big\}$\\ \hline
    (Prox-)NEXT \cite{Xu2020} & \eqref{eq:P} & Mesh & $\max\big\{\frac{\kappa-1}{\kappa+1}, \sqrt{1{-}(1{-}\rho_2({\bf W}))^2}\big\}$\\ \hline
     (Prox-)DIGing \cite{Xu2020} & \eqref{eq:P} & Mesh & 
     {$\max\Big\{\frac{\kappa}{\kappa+1}
, \sqrt{1{-}(1{-}\rho_2({\bf W}))^2}\Big\} $}
\\ \hline
    Primal-Dual \cite{Magnusson2020, Uribe2021} & \eqref{eq:P} with $r\equiv 0$ & Mesh & $\dfrac{\frac{\rho_1({\bf L})}{\rho_{m-1}({\bf L})}-\frac{1}{\kappa}}{\frac{\rho_1({\bf L})}{\rho_{m-1}({\bf L})}+\frac{1}{\kappa}}$\\ \hline
    \end{tabular}
    \caption{Convergence rate of some representative linearly convergent distributed algorithms cast in the proposed algorithmic framework (\ref{eq:M}), under the tuning described in  Appendix~\ref{app:examples}.   $\kappa$, $L$, and $\mu$ is the condition number, the smoothness constant and the strong convexity constant of each $f_i$, respectively, and $\bf W$ is the gossip matrix used in the communication steps of the algorithms; for the Primal-Dual algorithm, $\bf L$ is the graph Laplacian, defined in Appendix~\ref{subsec:Dual}.
    \label{table:lambda}}    
\end{table*}

\subsection{Proposed general model (using historical information)}\label{sec:multiple-round} 
We generalize the algorithmic model \eqref{eq:M0} as follows:  for all $i\in [m]$,
\begin{align} 
\!\!\!\!\!\!\!\!\begin{array}{*{20}l}
\left.\begin{array}{*{20}l}
\hat{\bf c}_{i}^{k,1} = \mathcal C_{i}^1\big({\bf z}_i^{k},{\bf 0}_{{\mathcal N}_i}\big), \\
 \qquad\vdots \\
\hat{\bf c}_{i}^{k,\q} = \mathcal C_{i}^{\q}\big({\bf z}_i^{k},\hat{\bf c}_{{\mathcal N}_i}^{k,\q-1}\big),
\end{array}
\right\} \begin{array}{l}
\text{(multiple}\\
\text{communication rounds)}
\end{array}
\\
\begin{array}{*{20}l}
{\bf z}_i^{k+1} = \mathcal A_i\big(
{\bf z}_i^{k}, \hat{\bf c}_{{\mathcal N}_i}^{k,1},\cdots,\hat{\bf c}_{{\mathcal N}_i}^{k,\q} \big), \quad \text{(computation step)}
\end{array} 
\end{array} \tag{M} \label{eq:M}
\end{align}
These updates embed $R\geq 1$ rounds of local communications, via the functions $(\mathbf{z}_i, \hat{\bf c}_{\mathcal N_i}^{\ITERR-1})\mapsto  \mathcal C_i^{\ITERR}(\mathbf{z}_i, \hat{\bf c}_{\mathcal N_i}^{\ITERR-1})$; the function
$(
{\bf z}_i, \hat{\bf c}_{{\mathcal N}_i}^{1},\cdots,\hat{\bf c}_{{\mathcal N}_i}^{\q} ) \mapsto \mathcal A_i(
{\bf z}_i, \hat{\bf c}_{{\mathcal N}_i}^{1},\cdots,\hat{\bf c}_{{\mathcal N}_i}^{\q} )$ updates the local state by
  using the $\hat{\bf c}_{\mathcal N_i}$ signals
 received from the neighbors during 
 all $R$ rounds of communications,
 along with $\mathbf{z}_i$.
 Stacking agents'    state-variables $\mathbf{z}_i$, communication signals $\hat{\mathbf{c}_i}$, and mappings $\mathcal{C}_i^s$ and $\mathcal A_i$ into the respective vectors $\mathbf{z}$,   $\hat{\mathbf{c}}$, $\mathcal C^s$ and $\mathcal A$, we can rewrite \eqref{eq:M} in the compact form
\begin{align*}
\begin{array}{*{20}l}
\left.\begin{array}{*{20}l}
\hat{\bf c}^{k, 0} = {\bf 0}, \\
\hat{\bf c}^{k, \ITERR} = \mathcal C^{\ITERR}\big({\bf z}^{k},\hat{\bf c}^{k, \ITERR-1}\big),   \ITERR\in[\q],
\end{array}
\right\} \begin{array}{l}
\text{(multiple}\\
\text{communication rounds)}
\end{array}
\\
\begin{array}{*{20}l}
{\bf z}^{k+1} = \mathcal A\big(
{\bf z}^{k},\hat{\bf c}^{k,1},\cdots,\hat{\bf c}^{k,\q}\big).\hspace{1.3cm} \text{(computation step)} \\
\end{array} 
\end{array}
\end{align*}
Absorbing the communication signals $\hat{\mathbf{c}}^{k,s}$ in the mapping $\mathcal{A}$, we can finally write the above system as a fixed-point iterate on the $z$-variables only: 
\begin{align}\label{eq:map_A_tilde}
\tag{M'} 
 \mathbf{z}^{k+1}=\tilde{\mathcal A}(\mathbf{z}^k)\triangleq& \mathcal A\Big(
{\bf z}^{k},\mathcal C^1\big({\bf z}^k,{\bf 0}\big),\cdots,\\
&\mathcal C^\q\big({\bf z}^k, \mathcal C^{\q-1}\big({\bf z}^k, \cdots \mathcal C^1\big({\bf z}^k, {\bf 0}\big) \cdots\big)\big)\Big). \nonumber
\end{align}
 {Under suitable conditions, the iterates (\ref{eq:map_A_tilde})}
convergence to fixed-points $\mathbf{z}^\infty=\tilde{\mathcal A}(\mathbf{z}^\infty)$ of the mapping $\tilde{\mathcal A}$, possibly   constrained to a set $\mathcal Z\ni \mathbf{z}^\infty$. 
The convergence rate depends on the properties of  $\tilde{\mathcal A}$; here we focus on {\it linear} convergence, which can be established under the following standard condition.

\begin{assumption} \label{assump:R_conv_z}
Let $\tilde{\mathcal A}: \mathcal{Z}\to \mathcal{Z}$;  the following hold: (i)  $\tilde{\mathcal A}$
admits a fixed-point ${\bf z}^\infty$; and (ii)  
   $\tilde{\mathcal A}$ is $\RTUQ$-pseudo-contractive  on  $\mathcal Z$ w.r.t.  some norm $\Vert \bullet \Vert$, that is,   there exists  $\RTUQ \in (0, 1)$
such that 
\begin{align*}
&\Vert\tilde{\mathcal A}({\bf z})-{\bf z}^\infty\Vert \leq \RTUQ \cdot \Vert{\bf z}-{\bf z}^\infty\Vert, \quad \forall {\bf z}\in \mathcal Z. 
\end{align*}
Without loss of generality,
the norm $\|\bullet\|$  is scaled such that $\Vert\bullet\Vert_2 \leq   \Vert\bullet\|$.\footnote{This is always possible since
$\Vert \bullet \Vert$ is a norm defined on a finite-dimensional field.}
\end{assumption}
The following convergence result follows readily from Assumption \ref{assump:R_conv_z} and \cite[Ch. 3, Prop.~1.2]{Bertsekas1989}. 
\begin{theorem} \label{thm:linear_conv_orig}
Let $\tilde{\mathcal A}:\mathcal{Z} \to \mathcal{Z}$ satisfy Assumption~\ref{assump:R_conv_z}. Then: i)
the fixed point ${\bf z}^\infty$ is unique; and ii) the sequence $\{{\bf z}^{k}\}$ generated by the update  \eqref{eq:map_A_tilde} converges Q-linearly to ${\bf z}^\infty$ w.r.t. the norm $\Vert \bullet\Vert$ at rate $\RTUQ$, i.e., 
  $
  \big\Vert{\bf z}^{k+1}-{\bf z}^\infty\big\Vert \leq \RTUQ\cdot\big\Vert{\bf z}^{k}-{\bf z}^\infty\big\Vert.
$
\end{theorem}

\textit{Discussion:} 
Our model treats the underlying unquantized algorithm as a black-box with convergence rate $\RTUQ$, which depends on the optimization problem parameters--the smoothness and strong convexity constants  $L$ and $\mu$ of the agents' functions (often via the condition number $\kappa$)--and the network connectivity  {(spectral properties of $\bf W$)}.  Table~\ref{table:lambda} collects some   representative examples.

The algorithmic framework (\ref{eq:M}) encompasses a variety of distributed algorithms, while Theorem~\ref{thm:linear_conv_orig} captures their convergence properties; in addition to the schemes covered by (\ref{eq:M0}) as a special case when $R=1$, (\ref{eq:M}) can also represent EXTRA \cite{Shi2015EXTRA} and its proximal version \cite{Xu2020}, NEXT {\cite{Lorenzo2016, Sun2019,Sun2019Math}} and its proximal version \cite{Xu2020}, DIGing \cite{Nedic2017} and its proximal version \cite{Xu2020}, and prox-NIDS \cite{Xu2020}. Appendix~\ref{app:examples} provides specific expressions for the mappings $\mathcal{A}$ and $\mathcal{C}^s$ for each of the above algorithms, along with their convergence properties  under Theorem~\ref{thm:linear_conv_orig}; 
here, we  elaborate on the NEXT algorithm \eqref{eq:NEXT} as an example.  
It can be cast in the form \eqref{eq:M}  by using $R=2$ rounds of communications and letting 
\begin{align*}
&{\bf z}_i^{k} =
\left[
\begin{array}{c}
{\bf x}_i^{k} \\
{\bf y}_i^{k}      
\end{array}
\right],\quad
\hat{\bf c}_i^{k,1} = {\bf x}_i^{k} - \STEP {\bf y}_i^{k}, \\
&\hat{\bf c}_i^{k,2} = {\bf y}_i^{k} + \nabla f_i\bigg(\sum_{j\in\mathcal N_i} w_{ij}\hat{\bf c}_j^{k,1}\bigg) -  \nabla f_i({\bf x}_i^k), \text{and}\\
&
{\bf z}_i^{k+1}=
\mathcal A_i({\bf z}_i^{k}, \hat{\bf c}_{{\mathcal N}_i}^{k,1}, \hat{\bf c}_{{\mathcal N}_i}^{k,2})
=\left[
\begin{array}{c}
\sum_{j\in\mathcal N_i} w_{ij}\hat{\bf c}_j^{k,1} \\
\sum_{j\in\mathcal N_i} w_{ij}\hat{\bf c}_j^{k,2}    
\end{array}
\right]. 
\end{align*}

\section{A General Distributed Algorithmic Framework: Quantized Communications} \label{Sec:framework_Q}
In this section, we equip the distributed algorithmic framework   \eqref{eq:M} with quantized communications. The  communication channel between any two agents  is modeled as a noiseless digital channel:  only quantized signals  are received with no errors.
This means that, in each of the communication rounds,  the signals $\hat{\mathbf{c}}_j^{k,1},\ldots, \hat{\mathbf{c}}_j^{k,R}$,  $j\in \mathcal{N}_i$,   received by agent $i$ 
may no longer coincide with the intended, unquantized ones $\mathcal{C}_j^1(\mathbf{z}_j^k, \mathbf{0}_{\mathcal N_j}),\ldots ,\mathcal{C}_j^R(\mathbf{z}_j^k, \hat{\mathbf{c}}_{\mathcal N_j}^{k,R-1})$,   generated at the transmitter side of agents $j\in \mathcal{N}_i$. 
This calls for a proper encoding/decoding mechanism that  transfers, via quantized communications, the aforementioned unquantized signals at the receiver sides  with limited distortion. Here, we leverage differential encoding/decoding techniques  \cite{Li2011} coupled with a novel quantization rule. 

 We begin by recalling  the idea of quantized  differential encoding/decoding in the context of  a  point-to-point communication--the same mechanism  will be then embedded  in  the communication of the distributed multi-agent framework \eqref{eq:M}.  Consider a transmitter-receiver pair; let $\mathbf{c}^k$ be the unquantized information generated at iteration $k$, intended to be transferred to the receiver over the digital channel, and let  $\hat{\mathbf{c}}^k$ be the estimate of  $\mathbf{c}^k$, built using quantized information. The differential  encoding/decoding rule reads:  {$\hat{\bf c}^{0}= {\bf 0}$, and} for     $k\geq 1$,
 \begin{align}\label{eq:code-decod-diff}
\begin{cases}
{\bf q}^{k} &= \mathcal Q^{k}({\bf c}^{k} - \hat{\bf c}^{k-1}), \\
\hat{\bf c}^{k} &= \hat{\bf c}^{k-1} + {\bf q}^{k}, 
\end{cases}
\end{align}
where $\mathcal Q^{k}$ is the quantization operator (a map from real vectors to the set of quantized points), possibly dependent on iteration $k$. In words, at each iteration, the encoder quantizes the \emph{prediction error}  ${\bf c}^{k} - \hat{\bf c}^{k-1}$ rather than the current estimate ${\bf c}^{k}$, generating the quantized signal ${\bf q}^{k}$, which is then transmitted  {over the digital channel}.  The estimate $\hat{\bf c}^{k}$ of ${\bf c}^{k}$   is then built from  ${\bf q}^{k}$ using  a one-step prediction rule. 
Since  ${\bf q}^{k}$ is received unaltered,   $\hat{\bf c}^{k}$ is identical at the transmitter's and receiver's sides. 
Note that, when quantization errors are negligible ${\bf q}^{k}=\mathcal Q^{k}({\bf c}^{k} - \hat{\bf c}^{k-1})\approx {\bf c}^{k} - \hat{\bf c}^{k-1}$, the estimate reads  $\hat{\bf c}^{k} = \hat{\bf c}^{k-1} + {\bf q}^{k}\approx  \hat{\bf c}^{k-1} + {\bf c}^{k} - \hat{\bf c}^{k-1}={\bf c}^{k}$.

 We can now introduce our distributed algorithmic framework using quantized communications, as described in Algorithm~\ref{alg};  it embeds the  differential encoding/decoding rule \eqref{eq:code-decod-diff} in each communication round of model \eqref{eq:M}. The fixed-point based formulation of  Algorithm~\ref{alg} then reads: for $i\in [m]$, 
\begin{align*}
\begin{array}{*{20}l}
\left.\begin{array}{*{20}l}
{\bf c}_{i}^{k,1} = \mathcal C_{i}^1\big({\bf z}_i^{k},{\bf 0}_{{\mathcal N}_i}\big), \\
\hat{\bf c}_{i}^{k,1} = \hat{\bf c}_{i}^{k-1,1} + \mathcal Q_i^k\big({\bf c}_{i}^{k,1} - \hat{\bf c}_{i}^{k-1,1}\big), \\
 \qquad \vdots \\
{\bf c}_{i}^{k,\q} = \mathcal C_{i}^\q\big({\bf z}_i^{k},\hat{\bf c}_{{\mathcal N}_i}^{k,\q-1}\big), \\
\hat{\bf c}_{i}^{k,\q} = \hat{\bf c}_{i}^{k-1,\q} + \mathcal Q_i^k\big({\bf c}_{i}^{k,\q} - \hat{\bf c}_{i}^{k-1,\q}\big),
\end{array}
\right\} 
\begin{array}{l}
\text{(multiple}\\
\text{quantized}
\\
\text{communication}
\\
\text{rounds)}
\end{array}
\\
\begin{array}{*{20}l}
{\bf z}_i^{k+1} = \mathcal A_i\big(
{\bf z}_i^{k}, \hat{\bf c}_{{\mathcal N}_i}^{k,1},\cdots,\hat{\bf c}_{{\mathcal N}_i}^{k,\q} \big), \qquad \text{(computation step).}
\end{array} 
\end{array}
\end{align*}

\begin{algorithm}[t]
\caption{Distributed Algorithmic Framework with Quantized Communications} 
\label{alg} 
\begin{algorithmic} 
	\Require   
  $\hat{\bf c}^{-1,\ITERR} \triangleq {\bf 0},$ for all  $\ITERR \in [\q]$; and 
	${\bf z}^{0}\in \mathcal Z$.  	Set $k=0$;\\
  Iteration $k\to k+1$
  \State \texttt{(S.1): Multiple quantized communication rounds} 
   
   \State \quad \textbf{for } $\ITERR=1,\ldots, R$, each agent $i$:
    \begin{itemize}
      \item Computes ${\bf c}_{i}^{k,\ITERR} = \mathcal C_i^{\ITERR}({\bf z}_i^{k},\hat{\bf c}_{{\mathcal N}_i}^{k,\ITERR-1} )$  [with $\hat{\bf c}_{i}^{k,0} \triangleq {\bf 0}$];
      \item Generates  ${\bf q}_{i}^{k,\ITERR} = \mathcal Q_i^{k}({\bf c}_{i}^{k,\ITERR} - \hat{\bf c}_{i}^{k-1, \ITERR})$
      and broadcasts it to its neighbors $\ITERNDT \in \mathcal{N}_i$;
      \item  Upon receiving the signals ${\bf q}_{\ITERNDT}^{k, \ITERR}$ from its neighbors $j\in \mathcal N_i$, it reconstructs $\hat{\bf c}_{\ITERNDT}^{k, \ITERR}$ as  $$\hat{\bf c}_{\ITERNDT}^{k, \ITERR} = \hat{\bf c}_{\ITERNDT}^{k-1, \ITERR} + {\bf q}_{\ITERNDT}^{k, \ITERR}, \quad   \ITERNDT \in \mathcal N_i;$$
    \end{itemize}
   \State\quad\textbf{end} 
   
   \State \texttt{(S.2): Computation Step} \State \quad Each agent $i$ updates its own ${\bf z}_i^{k+1}$ according to  $${\bf z}_i^{k+1} = \mathcal A_i(
{\bf z}_i^{k}, \hat{\bf c}_{{\mathcal N}_i}^{k,1},\cdots, \hat{\bf c}_{{\mathcal N}_i}^{k,\q} ).$$  
\end{algorithmic}
\end{algorithm}

 Stacking agents'    state-variables $\mathbf{z}_i$, signals ${\mathbf{c}_i}$ and    $\hat{\mathbf{c}_i}$, and mappings $\mathcal{C}_i^s$,  $\mathcal A_i$, and $\mathcal{Q}^k_i$ into the respective vectors $\mathbf{z}$,  ${\mathbf{c}}$,  $\hat{\mathbf{c}}$, $\mathcal C^s$, $\mathcal A$, and $\mathcal{Q}^k$, we can rewrite the above steps in compact form as
\begin{align}\tag{Q-M}
\label{Qupdated}
\begin{array}{*{20}l}
\left.\begin{array}{*{20}l}
{\hat{\bf c}^{k,0}} = \mathbf 0,\\
{\bf c}^{k,\ITERR} = \mathcal C^\ITERR\big({\bf z}^{k},\hat{\bf c}^{k,\ITERR-1}\big), \\
\hat{\bf c}^{k,\ITERR} = \hat{\bf c}^{k-1,\ITERR} + \mathcal Q^k\big({\bf c}^{k,\ITERR} - \hat{\bf c}^{k-1,\ITERR}\big),\\
   \qquad \ITERR\in[\q],
\end{array}
\right\}
\begin{array}{l}
\text{(multiple}\\
\text{quantized}
\\
\text{communication}
\\
\text{rounds)}
\end{array}
\\
\begin{array}{*{20}l}
{\bf z}^{k+1} = \mathcal A\big(
{\bf z}^{k}, \hat{\bf c}^{k,1},\cdots,\hat{\bf c}^{k,\q} \big)
. \qquad \text{(computation step).}
\end{array} 
\end{array}
\end{align}
  
Model  (\ref{Qupdated}) paves the way to  a unified  design and  convergence analysis  of several distributed algorithms--all the schemes cast in the form (\ref{eq:M})--employing quantization in the communications, as elaborated next. 
\subsection{Convergence Analysis}
     We begin by establishing sufficient  conditions on the quantization mapping  $\mathcal{Q}^k$ and  algorithmic functions $\mathcal{A}$ and $\mathcal{C}$ in (\ref{Qupdated}) to preserve linear convergence 
  
 $\bullet$ \textbf{On the quantization mapping  $\mathcal{Q}^k$.} A first critical choice is the quantizer  $\mathcal{Q}$ (we omit the dependence on $k$ and $i$ for notation simplicity), including both random and deterministic quantization rules (the latter as a special cases of the former). For random quantization,  the    function $\mathcal{Q}(\mathbf{x})$ is a random variable for any given $\mathbf{x}\in \mathbb{R}^d$, defined on  a suitable probability space (generally dependent on $\mathbf{x}$).  We define the following novel   \emph{biased compression rule} (BC-rule), for each agent $i$.
\begin{definition}[{\bf Biased compression rule}] \label{def:quant}
Given $\mathbf{x}\in {\mathbb R}^d$, $\mathcal Q(\mathbf{x})$
 (possibly, a random variable defined on a suitable probability space) satisfies the BC-rule
with  \emph{bias}  $\QNSB \geq 0$ and \emph{compression rate}  $\QNSC \in [0,1)$ if
\begin{align}
\sqrt{\mathbb E\Big[\big\Vert \mathcal Q(\mathbf x) -  \mathbf x \big\Vert_2^2\Big]} \leq
\sqrt{d}\,\QNSB + \QNSC \Vert \mathbf x\Vert_2, \quad  \forall  \mathbf x \in {\mathbb R}^d.
 \label{eq:BCR}
\end{align}
\end{definition}
When $\mathcal Q({\bf x})$ is a deterministic map, (\ref{eq:BCR}) reduces to \begin{equation}\label{eq:BCR_determin}
\big\Vert \mathcal Q(\mathbf x) -  \mathbf x \big\Vert_2 \leq
\sqrt{d}\,\QNSB + \QNSC \Vert \mathbf x\Vert_2, \quad  \forall  \mathbf x \in {\mathbb R}^d.\end{equation}
{Roughly speaking,  the bias $\QNSB$ determines the basic spacing between {quantization points},  uniform across the entire domain.  On the other hand,  the compression term $\QNSC$ adds a non-uniform   spacing between {quantization points}:  {points} farther away from $\bf 0$ have more  separation. } 
 
The BC-rule encompasses and generalizes several existing compression and quantization rules proposed in the literature for specific algorithms, deterministic  \cite{Alistarh2018, Karimireddy2018, Stich2018, Koloskova2019, Zheng2019, Stich2020, Liao2021, Kashyap2007, Nedic2009, Li2011, Lavaei2012, Thanou2013, ElChamie2016, Li2017, Lee2018QNEXTconf, Magnusson2020, Kajiyama2020} and random    {\cite{Aysal2008, Kar2010, Rajagopal2011, Zhu2015, Alistarh2017, Lee2017conf, Lee2018CDC, Stich2018, Tang2018, Koloskova2019, Zhang2019, Beznosikov2020, Gorbunov2020,  Kovalev2020, Lee2020, Stich2020, Taheri2020, Liao2021, Haddadpour2021, Liu2021, Zhang2021 }}.  Specifically,  \textbf{(i)} the compression rules proposed  in  {\cite{Alistarh2017, Stich2018, Tang2018, Koloskova2019, Zhang2019, Beznosikov2020, Gorbunov2020, Kovalev2020, Stich2020, Taheri2020, Liao2021, Haddadpour2021, Liu2021, Zhang2021} (resp. \cite{Alistarh2018, Karimireddy2018, Stich2018, Koloskova2019, Zheng2019, Stich2020, Liao2021}) can be interpreted as  \emph{unbiased} instances of \eqref{eq:BCR} [resp. \eqref{eq:BCR_determin}]}, i.e.,  corresponding to  $\QNSB = 0$. The proof of Lemma \ref{lemma:quant} in \secref{sec:pro_quant} will show that such special instances
theoretically require an infinite number of bits to encode quantized signals, even when the input ${\bf x}$ is bounded. In practice, they are successfully implemented using   finite bits at the machine precision (e.g., to encode  some scalar quantities, such as the norm of the signal to be transmitted \cite{Taheri2020}).  However, their convergence analyses tacitly assume that errors at machine precision are negligible. Hence, these schemes 
  have no performance guarantees when implemented with limited (below machine)  precision quantizations. 
On the other hand, as  it will be seen in \secref{sec:pro_quant},
the bias term $\eta>0$ in the proposed BC-rule guarantees that
quantized signals satisfying such rule can be encoded using a finite number of bits, hence can be implemented with limited precision. \textbf{(ii)} The {quantization} rules in {\cite{Aysal2008, Kar2010, Rajagopal2011, Zhu2015, Lee2017conf, Lee2018CDC, Lee2020} (resp. \cite{Kashyap2007, Nedic2009, Li2011, Lavaei2012, Thanou2013, ElChamie2016, Li2017, Lee2018QNEXTconf, Magnusson2020, Kajiyama2020})} are special instances of  the BC-rule {\eqref{eq:BCR} [resp. \eqref{eq:BCR_determin}]},  with  $\QNSC = 0$. While they can be implemented using   a  finite number of bits {(provided that the signals to be quantized are uniformly bounded)}, they do not take advantage of  the degree of freedom offered  by the compression rate $\QNSC$ to improve communication efficiency, as shown numerically in \secref{sec:sim:cc}.

$\bullet$ \textbf{On the algorithmic mappings $\mathcal A$ and $\mathcal C^\ITERR$.} Our analysis requires additional standard conditions on the mappings $\mathcal A$ and $\mathcal C^\ITERR$ to preserve linear convergence under quantization. Roughly speaking, 
the functions  $\mathcal A$ and $\mathcal C^\ITERR$ should vary smoothly with respect to perturbations in their arguments, so that small quantization errors result  in small deviations from the trajectory of the unquantized algorithm,
as postulated next.\footnote{For the sake of notation,  the constants $L_A,L_C$ and $L_{Z}$ defined in Assumptions \ref{assump:Lipt_A} and \ref{assump:Lipt_C} are assumed to be independent of the index $\ITERR$ (communication round). Our convergence results can be readily extended to constants dependent on $s$.}

\begin{assumption} \label{assump:Lipt_A}
There exists a constant $L_A \geq 0$  such that, for every   $s\in [\q]$,  it holds 
\begin{align}
\nonumber
& \big\Vert
\mathcal A({\bf z},{\bf c}^{1},\cdots, {\bf c}^{\ITERR},  \cdots, {\bf c}^{\q})
{-}\mathcal A({\bf z},{\bf c}^{1},\cdots, \tilde{\bf c}^{\ITERR}, \cdots, {\bf c}^{\q})
\big\Vert  \\&
 \qquad\leq   L_A \Vert {\bf c}^{\ITERR}-\tilde{\bf c}^{\ITERR}\Vert_2,
\label{eq:Lipt_A}
\end{align}
for all  $ {\bf c}^{\ITERR}, \tilde{\bf c}^{\ITERR}\in \mathbb{R}^{md}$, uniformly with respect to  ${\bf z} \in \mathcal Z$, and     ${\bf c}^{1},\dots,{\bf c}^{s-1},  {\bf c}^{s+1},\ldots {\bf c}^{R}\in \mathbb R^{md}$. \end{assumption}
\begin{assumption}\label{assump:Lipt_C}
There exist constants $L_C, L_{Z}\geq 0$  such that  
 \begin{align}
 &\Vert \mathcal C^\ITERR({\bf z}, {\bf c})-\mathcal C^\ITERR({\bf z}, {\bf c}^\prime) \Vert_2 \leq L_C \Vert {\bf c}-{\bf c}^\prime \Vert_2,  \forall  {\bf c},  {\bf c}^\prime \in \mathbb{R}^{m d}, \label{eq:Lipt_C}\\ 
 &\Vert \mathcal C^\ITERR({\bf z}, {\bf c})-\mathcal C^\ITERR({\bf z}^\prime, {\bf c}) \Vert_2 \leq L_{Z} \Vert  {\bf z}-{\bf z}^\prime \Vert_2, \forall \mathbf{z},  \mathbf{z}^\prime\in \mathcal{Z}, \label{eq:Lipt_Z}
 \end{align}
 uniformly with respect to ${\bf z}  \in \mathcal Z$ and  ${\bf c}\in \mathbb{R}^{md}$, respectively. 
\end{assumption}

These assumptions are quite mild,
and satisfied by a variety of existing distributed algorithms, as shown
in   Appendix \ref{app:examples}. We are now ready to  introduce our  main convergence result.

\begin{theorem} \label{thm:conv}
Let $\{{\bf z}^{k}\}$ be the sequence generated by Algorithm~\ref{alg} under Assumptions {\ref{assump:R_conv_z}, \ref{assump:Lipt_A}, and \ref{assump:Lipt_C}}, with $\mathcal Q^{k}$ satisfying the BC-rule \eqref{eq:BCR} with
 {bias $\QNSB=\QNSB^0 \cdot (\RTQ)^k$ and compression rate $\QNSC \in [0,\bar{\QNSC}(\RTQ))$}, for some $\RTQ \in (\RTUQ, 1)$ and $\QNSB^0 > 0$ , where $\bar{\QNSC}(\RTQ)$ is defined as 
\begin{align}
\bar{\QNSC}(\RTQ)
\triangleq 
\frac{\RTQ}{\q}\cdot
\frac{\RTQ-\RTUQ}{\RTQ-\RTUQ+2 L_A L_{Z}[\q\max\{1,(2L_C)^{\q-1}\}]^2}.\label{omega_bar}
\end{align}
Then,  
$$\sqrt{\mathbb E[\Vert {\bf z}^{k} - {\bf z}^\infty \Vert_2^2]} \leq {\LIAP}_0 \cdot (\RTQ)^k,\quad k=0,1,\ldots,$$ where ${\LIAP}_0$ is a positive constant, whose expression is given in  \eqref{eq:V0}, Appendix~\ref{pf:thm:conv}.
\end{theorem}

\begin{IEEEproof}
See Appendix \ref{pf:thm:conv}.
\end{IEEEproof}
Note that, when the deterministic instance of the BC-rule is used [see (\ref{eq:BCR_determin})], the convergence rate reads $\Vert {\bf z}^{k} - {\bf z}^\infty \Vert_2 \leq {\LIAP}_0 \cdot (\RTQ)^k,$ for all  $k=0,1,\ldots$ .

Theorem~\ref{thm:conv} shows that linear convergence is achievable when communications are quantized in distributed optimization, provided that  the   {bias}  {$\QNSB$} and  {compression rate} $\omega$ of the  BC-rule are chosen suitably.  The shrinking  requirement on the bias {$\QNSB$} (linear at rate $\RTQ$)
 is \emph{not} restrictive, since it is consistent with the contraction dynamics of the iterates: the range of inputs to the quantizer vanishes in a similar fashion, a fact that guarantees that quantized values can be encoded using a uniformly bounded number of bits (see Theorem~\ref{thm:comm_cost_linear}, Sec.~\ref{sec:commcost}).
Similarly, the bound on $\omega$ guarantees that quantization errors along the iterates do not accumulate excessively. 

  Theorem~\ref{thm:conv} certifies linear convergence in terms of number of iterations.
  However, the algorithm that uses quantized communications converges slower (with rate $\RTQ$) than its unquantized counterpart (rate $\RTUQ$),
  revealing a tension  between the  amount of  quantization/compression of the transmitted signals (measured by $\QNSB$ and  $\omega$)    and the resulting linear convergence rate $\RTQ$: as we will see in the forthcoming sections,
  this tension entails a trade-off between faster convergence (closer to that of the unquantized algorithm) and communication cost, which we aim to characterize.
  Building on this result,  we will also investigate the communication complexity of the schemes (\ref{Qupdated})--the total number of bits needed to reach an $\varepsilon$-solution of problem (\ref{eq:P}). Since this analysis depends on the specific quantizer design,  the next section introduces a novel quantizer  that efficiently implements  the  BC-rule, and a communication-efficient bit-encoding/decoding scheme.    When embedded in  (\ref{Qupdated}), the proposed quantization leads to linearly convergent   distributed algorithms whose communication complexity compares favorably with that of existing ad-hoc schemes (\secref{sec:commcost}).

\section{Non-Uniform Quantizer with Adaptive Encoding/Decoding}
\label{sec:pro_quant} 
As discussed in \secref{Sec:framework_Q}, the BC-rule encompasses a variety of quantizer designs. 
In this section, we propose a  quantizer   that fulfills the BC-rule with minimum number of quantization   points (\secref{sec:quant_design}). The quantizer is then coupled with a  {communication-}efficient  bit-encoding/decoding rule which enables transmission on the digital channel (\secref{sec:encoding}). We refer to the proposed quantizer coupled with the encoding/decoding scheme as  \emph{Adaptive encoding Non-uniform Quantizer}  (ANQ). 

\subsection{Quantizer design} \label{sec:quant_design}
Since  no information is assumed on the distribution of
the input signal, a  natural approach is to  quantize each vector signal component-wise.  
We design such a scalar quantizer $\mathcal Q: [-\RANGE, \RANGE]\to \mathbb Q$ under the BC-rule by minimizing the number of  {quantization points} $|\mathbb Q|$ for a fixed input range $[-\delta,\delta]$. Equivalently, we seek  
$\mathcal Q$  
 that maximizes $\RANGE$ under the BC-rule, for a given number   $N=|\mathbb Q|$ of  {quantization points}.  These designs are provided in Lemmas~\ref{lemma:quant} and~\ref{lemma:quant_prob}
for  the deterministic and probabilistic cases,
respectively. For convenience, we focus on the case of  $N$ odd; the case  of  $N$ even is provided in    Appendices \ref{pf:lemma:quant} and \ref{pf:lemma_quant_prob}.
We point out that the restriction of the input of $\mathcal{Q}$ to $[-\delta,\delta]$,   as opposed to the unconstrained domain  in the   BC-rule  (Definition~\ref{def:quant}), is instrumental in formulating  the quantizer's design as  an optimization problem. As  shown in Lemmas~\ref{lemma:quant} and~\ref{lemma:quant_prob}, the resulting optimized quantization points are independent of the specific choice of $\delta$; hence, the proposed    quantizer can be used (component-wise) for input signals in $\mathbb{R}^d$.
\begin{lemma}[Deterministic Quantizer] \label{lemma:quant}
Let $\mathcal Q:[-\RANGE, \RANGE]\to\mathbb Q$.
The maximum range $\delta$ that can be quantized using  $|\mathbb Q|=N$ (odd) points under the BC-rule \eqref{eq:BCR_determin} with bias $\QNSB\geq 0$ and compression rate $\QNSC \in [0,1)$ is
\begin{align}
{\delta(\QNSB,\QNSC,N)}=\frac{q_{(N-1)/2}+q_{(N+1)/2}}{2}, \label{eq:delta_det}
\end{align}
with quantization points 
\begin{align}
\label{qell}
&q_{\ell}=-q_{-\ell}=
\dfrac{\QNSB}{\QNSC}\left[
\Big(\frac{1+\QNSC}{1-\QNSC}\Big)^\ell-1\right],
\quad  \ell\geq 0.
\end{align}
The resulting optimal quantization rule reads: $x\mapsto \mathcal{Q}(x)=q_{\ell(x)}$, with 
\begin{equation} \label{eq:lstar}\ell(x)
=
\mathrm{sign}(x) \cdot
\bigg\lceil
\frac{\ln(1-\QNSC)+\ln(1+\frac{\QNSC}{\QNSB}|x|)
}{\ln(1+\QNSC)-\ln(1-\QNSC)}\bigg\rceil.
\end{equation} 
\end{lemma}

\begin{IEEEproof}
See Appendix \ref{pf:lemma:quant}.
\end{IEEEproof}  

\begin{lemma}[Probabilistic Quantizer] \label{lemma:quant_prob}
For any given   ${x}\in [-\delta,\delta]$, let $\mathcal Q({x})\in \mathbb{Q}$ 
be a random variable  defined on a suitable probability space.   
The maximum range $\delta$ that can be quantized using  $|\mathbb Q|=N$ (odd) points under the BC-rule \eqref{eq:BCR} with $\mathbb E[\mathcal Q({x})] = {x}$, bias $\QNSB\geq 0$ and compression rate $\QNSC \in [0,1)$ is
\begin{equation}
\delta(\QNSB, \QNSC, N) = 
q_{(N-1)/2},\label{eq:delta_prob}
\end{equation}
with quantization points
\begin{align}
\label{qellprob}
q_{\ell}=-q_{-\ell}
=
\frac{\eta}{\omega}
\Big[\Big(\sqrt{1+(\omega)^2}+\omega\Big)^{2\ell}-1\Big]
,\quad  \ell\geq 0.
\end{align}
 The resulting optimal quantization rule reads: $x\mapsto \mathcal{Q}(x)=q_{\ell(x)}$, 
 with   
\begin{align}
&\ell(x) 
=
\begin{cases}
\ell-1, & \textrm{ w.p. }\, \frac{q_{\ell} - x}{q_{\ell} - q_{\ell-1}}; \\
\ell, & \text{ w.p. } \,\frac{x - q_{\ell-1}}{q_{\ell} - q_{\ell-1}},
\end{cases} \nonumber \quad \text{and} \\&     \ell=\mathrm{sign}(x)
\bigg\lceil\frac{\ln(1+\frac{\omega}{\eta}|x|)}
    {2\ln\Big(\sqrt{1+(\omega)^2}+\omega\Big)}\bigg\rceil
.
\label{eq:lstar_prob}
\end{align}
\end{lemma}
\begin{IEEEproof}
See Appendix \ref{pf:lemma_quant_prob}.
\end{IEEEproof}

From these lemmas, one infers that quantization points under the BC-rule should be  non-uniformly spaced--hence the  name ANQ.
Furthermore, the deterministic quantizer maps an input  $x$ to the nearest   $q_\ell$,
whereas the probabilistic one
maps $x$ to one of the two nearest  {quantization points},  selected randomly such that $\mathbb E[\mathcal Q(x)] = x$.

Note that, when specialized to the conventional compression rule that uses
$\QNSB=0$, the two lemmas above yield $\delta(0, \QNSC, N)=0$  for any finite $N$, implying that infinite quantization points (hence number of bits) are required to encode signals. The next corollary formalizes this  negative result.

\begin{corollary}[Converse] \label{coro:converse}
The (unbiased, $\QNSB=0$) compression rule \eqref{prob_compression_rule}
cannot be satisfied
using a finite number of  {quantization points}  to quantize signals
within a given  range $[-\RANGE, \RANGE]^d$, for any finite $\delta>0$.
Therefore, the   compression rules in \cite{Alistarh2017, Alistarh2018, Karimireddy2018, Stich2018, Tang2018, Koloskova2019, Zhang2019, Zheng2019, Beznosikov2020, Kovalev2020, Taheri2020, Liao2021, Haddadpour2021} 
{theoretically require an infinite number of bits to encode quantized signals.}
\end{corollary}
\begin{proof}
See Appendix \ref{app:pf_converse}.
\end{proof}

Note that  the index $\ell(x)$ is a sufficient information to reconstruct the quantization point $q_{\ell(x)}$ at the receiver.
In the next section, we present a communication-efficient finite bit-encoding/decoding scheme to transmit  $\ell(x)$ over the digital channel. 

\subsection{Adaptive encoding scheme} \label{sec:encoding}
 It remains to design an   encoding/decoding  scheme mapping  the index $\ell(x)$ into a finite-bit  representation, to be transmitted over the digital channel. To do so, we adopt an adaptive number of bits, based upon the value of  $\ell(x)$, as detailed next.  

 We assume that a constellation $\mathbb S=[S]\cup\{0\}$
 of $S+1$
symbols is used, with $S\geq 2$ (this might be
obtained as $\mathbb S\equiv (\tilde{\mathbb S})^{\mathrm{w}}$,
by concatenating sequences of $\mathrm{w}$  symbols from a smaller constellation $\tilde{\mathbb S}$).
We use the symbol $0$ to indicate the end
 of  an information sequence, and the remaining $S$ symbols in the set $[S]$ to  encode the value of $\ell(x)$.
With  $\tilde{\mathcal L}_{-1}\equiv\emptyset$, let 
$$\tilde{\mathcal L}_b\equiv\bigg\{
-\bigg\lceil\frac{S^{b+1}-1}{2(S-1)}\bigg\rceil
+1,\dots,
\bigg\lfloor\frac{S^{b+1}-1}{2(S-1)}\bigg\rfloor
\bigg\},\quad  b=0,1,\ldots,$$
and 
\begin{align}
\label{Lb}
{\mathcal L}_b=\tilde{\mathcal L}_b\setminus\tilde{\mathcal L}_{b-1},\quad   b=0,1,\ldots.
\end{align}
It is not difficult to see  that $\{{\mathcal L}_b\,:\,b= 0,1,\ldots \}$ creates a partition of $\mathbb Z$ and $|{\mathcal L}_b|=(S)^b$.
Therefore, a natural way to
encode $\ell(x)\in{\mathcal L}_b$ is to use a unique sequence of  $b$ symbols from $[S]$,
followed by $0$ to mark the end of the information  {sequence}.
$\ell(x)$ is then encoded as  
$[s_1,\dots,s_{b},0]\in(\mathbb S)^{b+1}$.

Upon receiving this sequence, the receiver can detect the start and end of the information symbols, and decode the associated $\ell(x)$ by inverting the symbol-mapping. The communication cost to transmit the index $\ell(x)\in\mathcal L_b$ is thus $b+1$ (symbols),
which leads to the following upper bound on the communication cost incurred by each agent $i$ to quantize and encode a $d$-dimensional vector $\bf x$. Again, we focus on the case when $N$ is odd; the other case is provided in the proof in Appendix \ref{app:pf_C_ell_ub}.
\begin{lemma} \label{lemma:C_ell_ub}
The number of bits $ C({\bf x})$ required   by the ANQ with
 bias  $\QNSB\geq 0$ and compression rate $\QNSC\geq 0$  and   constellation of $S+1$ symbols to quantize and encode an 
input signal
${\bf x}\in\mathbb R^d$
is upper bounded by

 {\bf (i) Deterministic quantizer:}
\begin{align}
\label{eq:C_ell_ub}
& C({\bf x})\leq 
 3d\log_2(S+1)\\&
{+}d\log_2(S{+}1)\log_S\bigg(2{+}\frac{\ln(1-\QNSC){+}\ln\big(1{+}\frac{\QNSC\Vert \mathbf x\Vert_2}{\sqrt{d}\QNSB}\big)}{\ln(1+\QNSC) - \ln(1-\QNSC)}\bigg) \quad \mathrm{bits};   \nonumber
\end{align}

 {\bf (ii) Probabilistic quantizer with  {$\mathbb E[\mathcal Q(\bf x)] =\bf  x$:}}
 \begin{align}
 \label{eq:C_ell_ub_prob}
& C({\bf x})
\leq
3d\log_2(S+1)\\
&{+}d\log_2(S{+}1)\log_S\bigg(2{+}\frac{\ln\Big(1{+}\frac{\omega\Vert\mathbf x\Vert_2}{\sqrt{d}\eta}\Big)}{2\ln\Big(\sqrt{1{+}(\omega)^2}{+}\omega\Big)}\bigg)
\quad \mathrm{bits}, a.s..     \nonumber
\end{align}
\end{lemma}

\begin{proof}
See Appendix \ref{app:pf_C_ell_ub}.
\end{proof}
Compared with  existing deterministic quantizers  \cite{Magnusson2020, Kajiyama2020} that are special cases of the BC-rule (with  $\QNSC = 0$), the proposed  ANQ adapts the number of bits to the input signal--less bits for smaller input signals (mapped to smaller $\ell$)  and more bits for larger ones (mapped to larger $\ell$)--rather than   using a fixed number of bits  determined by the worst-case input signal \cite{Magnusson2020, Kajiyama2020}. This leads to more communication-efficient  schemes,   as will be certified by  Theorem~\ref{thm:comm_cost_mesh}.

\section{Communication Complexity of (\ref{Qupdated})\\ based on  the ANQ Rule} \label{sec:commcost}
We now study the communication complexity of the distributed schemes falling within the framework  (\ref{Qupdated})  and using the ANQ to quantize    communications. Our results complement Theorem~\ref{thm:conv} and are of two types: (i)  first, 
we characterize the trade-off between the number of bits/agent
required by the ANQ at each iteration and the linear convergence rate
 (Theorem \ref{thm:comm_cost_linear});  (ii) then, we  investigate the communication complexity, characterizing the total number of bits/agent transmitted to achieve an $\varepsilon$-solution of \eqref{eq:P} (Theorem~\ref{thm:comm_cost_mesh}). 
Since these results are applicable to
any distributed algorithm within the setting of (\ref{Qupdated}),
we finally customize (ii) to some specific instances (\secref{sec:examples_ANQ}).

 Throughout this section,  all the  results stated in terms of $\mathcal{O}$-notation are meant  asymptotically  when $\m, d \to \infty$. Also, the  following additional mild assumption is postulated, which is  satisfied by a variety of  existing algorithms, see Appendix~\ref{app:examples}.
\begin{assumption} \label{assump:commcost}
The constants $L_A, L_C,L_Z,\q$ and the initial conditions $\Vert{\mathcal C}^{\ITERR}({\bf z}^{0}, {\bf 0})\Vert_2$ and  $\Vert {\bf z}^0 - {\bf z}^\infty\Vert$ satisfy
\begin{align*}
   & L_A\cdot L_Z = \mathcal O(1),\  L_C = \mathcal O(1),\ \q = \mathcal O(1),\\&
   \Vert {\bf z}^0 - {\bf z}^\infty\Vert=\mathcal O(\sqrt{\m d})
   \\   &\text{and} \quad \Vert{\mathcal C}^{\ITERR}({\bf z}^{0}, {\bf 0})\Vert_2=\mathcal O(L_Z\sqrt{\m d}),\ \forall s\in [R].
\end{align*}
\end{assumption}

Our first result on the number of bits transmitted at each iteration to sustain linear convergence is summarized next. 

\begin{theorem} \label{thm:comm_cost_linear}
{Instate}  the setting of  Theorem \ref{thm:conv}, under the additional Assumption~\ref{assump:commcost}. Furthermore,  suppose that the deterministic ANQ (or probabilistic ANQ with $\mathbb E[\mathcal Q({\bf x})] = {\bf x}$) is used to quantize all the communications in  (\ref{Qupdated}), with 
$\QNSB^0 = \Theta(L_Z(\RTQ-\RTUQ))$
and
$\QNSC$ such that $1 -\QNSC/\bar\QNSC(\RTQ)=\Omega(1)$. 
Then,  linear convergence $\sqrt{\mathbb E[\Vert {\bf z}^{k} - {\bf z}^\infty \Vert_2^2]}=\mathcal O(\sqrt{md}\cdot (\RTQ)^k)$,
$k=0,1,\ldots$,  is achieved with an
 average number of bits/agent at every iteration $k$ of order
 \begin{equation}\label{eq:bits-per-it}
\mathcal O \left(d\cdot\log \Big(1+\frac{1}{\RTQ(\RTQ-\RTUQ)}\Big)\right).  
 \end{equation}  
\end{theorem}
\begin{IEEEproof}
See Appendix \ref{pf:thm:comm_cost_linear}.
\end{IEEEproof}
The following comments are in order. 

\textbf{(i)} As expected,
 {the faster the quantized algorithm (smaller $\RTQ$), the larger
the communication cost; in particular,}
when $\RTQ\to\RTUQ$, the number of bits required to sustain   linear convergence at rate $\sigma$ grows  indefinitely. In other words, an  infinite number of bits is required if a quantized distributed scheme (\ref{Qupdated}) wants to match  the convergence rate  of its unquantized counterpart. 

\textbf{(ii)} It is interesting to contrast the communication efficiency (bits transmitted per iteration) of the  proposed model (\ref{Qupdated})  equipped with the ANQ with that of existing schemes  {applicable to special instances of (\ref{eq:P}) or network topologies. Specifically, \cite{Kajiyama2020, Magnusson2020} study (\ref{eq:P}) with $r\not\equiv 0$ over mesh networks;   the scheme therein convergences linearly (under suitable tuning/assumptions) while using  $$\mathcal O \left(d\log \Big(1+\frac{\sqrt{\m d}}{\RTQ(\RTQ - \RTUQ)}\Big)\right)\quad \text{bits/agent/iteration.}$$    The algorithm in  \cite{Magnusson2020} is also applicable to  (\ref{eq:P}), with $r\not\equiv 0$, over star-networks; it uses    $$\mathcal O \left(d\log \Big(1+\frac{\sqrt{d}}{\RTQ(\RTQ - \RTUQ)}\Big)\right)\quad \text{bits/agent/iteration}. $$ } Both  are less favorable than \eqref{eq:bits-per-it}, due to the fact that the ANQ adapts the number of bits to the input signal rather than adopting a constant number of bits for any input signal. 

\textbf{(iii)} 
Theorem \ref{thm:comm_cost_linear} reveals a tension    between convergence rate  (the closer  $\RTQ$  to $\RTUQ$, the faster the algorithm) and  number of transmitted bits per iteration (the larger $\RTQ$, the smaller the cost). In the following, we provide a favorable choice of  $\sigma$ 
that  resolves this tension by characterizing the communication complexity, i.e., the total number of  bits transmitted per agent to achieve a target $\varepsilon$-accuracy.
\begin{theorem} \label{thm:comm_cost_mesh}
 {Instate}  the setting of  Theorem \ref{thm:comm_cost_linear}, with $\RTQ$ chosen so that
 $$\frac{(1-\RTUQ)^2}{(1-\RTQ)(\RTQ-\RTUQ)}=\mathcal O(1).$$
Then, the  {average} number of bits transmitted per agent to achieve    $(1/\m)\mathbb E[\Vert {\bf z}^{k} - {\bf z}^\infty\Vert_2^2] \leq \varepsilon$ scales as
\begin{equation}
\mathcal O \bigg(d\cdot\log\Big(1+\frac{1}{1-\RTUQ}\Big
)
\frac{1}{1-\RTUQ} 
\log(d/\varepsilon)
\bigg) \quad  {\mathrm{bits}}/{\mathrm{agent}}, \label{comm_cost_gen}
\end{equation}
 {achieved in
\begin{equation}
\mathcal O \bigg(
\frac{1}{1-\RTUQ}\cdot 
\log(d/\varepsilon)
\bigg) \quad  \mathrm{iterations}
\end{equation}
and with 
\begin{equation}
\mathcal O \bigg(d\cdot\log\Big(1+\frac{1}{1-\RTUQ}\Big)
\bigg) \quad  \mathrm{bits/agent/iteration}.
\end{equation}}
\end{theorem}
\begin{IEEEproof}
See Appendix \ref{pf:thm:comm_cost_mesh}.
\end{IEEEproof}
The following comments are in order.

\textbf{(i)} Intuitively,   Theorem~\ref{thm:comm_cost_mesh} provides a range of values of $\RTQ$   to balance the tension
between convergence rate and communication cost per iteration, resulting from too small or too large values of  $\RTQ$. The condition of the theorem can be satisfied, e.g., by  choosing $\RTQ=(1+\RTUQ)/2$. 

\textbf{(ii)} 
The term $d\cdot \log (1+ 1/(1-\RTUQ))$ in \eqref{comm_cost_gen} represents the number of bits/iteration/agent under the additional restriction on $\sigma$, as postulated  by Theorem~\ref{thm:comm_cost_mesh}: the faster the unquantized algorithm (i.e., the smaller $\RTUQ$), the less bits are required.

\textbf{(iii)} 
The second term $(1-\RTUQ)^{-1} \log(d/\varepsilon)$ represents the total number of iterations required to achieve $\varepsilon$ accuracy;
 {remarkably, these are the same (in a $\mathcal O$-sense) as the unquantized algorithm}, with   $\log(d)$ capturing the gap of the initial point from the  {fixed point}.  As expected,  the number of iterations increases as the   unquantized algorithm slows down ($\RTUQ$ increases),  the dimension $d$ increases, and/or  the target error $\varepsilon$ decreases. 

Nest, we  customize  Theorem \ref{thm:comm_cost_mesh} to some distributed algorithms  within (\ref{Qupdated}).

\subsection{Special cases of (\ref{Qupdated}) using the ANQ rule} \label{sec:examples_ANQ}

\textbf{1) GD over star-networks:} Our first case study is the GD algorithm solving (\ref{eq:P})  with $r\equiv 0$ over star networks. The unquantized scheme reads 
\begin{equation}
{\bf x}^{k+1} = {\bf x}^{k} - \frac{\STEP}{\m}\sum_{i=1}^\m \nabla f_i({\bf x}^{k}), \label{eq:par_GD}
\end{equation}
with  $\STEP \in (0, 2/L)$.
This is an instance of the algorithmic framework  \eqref{eq:M}; therefore, we can employ quantization   using   (\ref{Qupdated}). When the ANQ is employed,  a direct application of  Theorem \ref{thm:comm_cost_mesh} 
leads to the following communication complexity for the quantized GD ({we termed the algorithm ANQ-GD-star}).

\begin{corollary}[ANQ-GD-star, see Appendix \ref{AppD1}] \label{coro:comm_cost_star}
 {Given \eqref{eq:P} over a star-network, where   $r\equiv 0$ and  $F$ is $L$-smooth and $\mu$-strongly convex, thus with   condition number $\kappa=L/\mu$, and unique minimizer $\mathbf{x}^\star$, consider the ANQ-GD-star with stepsize  $\STEP = 2/(\mu+L)$ and tuning for the ANQ as in   Theorem \ref{thm:comm_cost_mesh}. Then,  the average number of bits/agent for   $\Vert {\bf x}^k - {\bf x}^*\Vert_2^2 \leq \varepsilon$     is of order \begin{equation}\label{eq:cor-GD-quantized}
{\mathcal O} \Big(d\cdot\log(1+\kappa)\kappa\log\big(d/\varepsilon\big) \Big),
\end{equation}
achieved in
\begin{equation}
{\mathcal O} \Big(\kappa\log\big(d/\varepsilon\big) \Big)\quad
\mathrm{ iterations}
\end{equation}
and with 
\begin{equation}
{\mathcal O} \Big(d\cdot\log(1+\kappa)\Big)\quad 
\mathrm{bits/agent/iteration}.
\end{equation}
} \end{corollary}
 This behavior compares favorably with   $\mathcal O(d\kappa\log(d(1+\kappa))\log (d/\varepsilon))$ bits/agent obtained  in \cite{Magnusson2020}
using  a deterministic quantizer, with fixed number of bits among agents and iterations.
It also matches $\mathcal O(d\kappa\log(1+\kappa)\log (d/\varepsilon))$ bits/agent obtained  in \cite{Alistarh2021tight} using the same deterministic quantizer as in \cite{Magnusson2020}.
However, \cite{Alistarh2021tight} uses a central coordinator to optimize  the number of bits used by each agent at every iteration.
Note that the lower complexity bound provided in
\cite{Davies2021} for  $\kappa=1$
reads: $\Omega\big(d \log (d/\varepsilon)\big)$ bits/agent, confirming the tightness of our result.

\textbf{2) Distributed algorithms employing gradient correction:} Our second example deals with distributed algorithms solving \eqref{eq:P} (possibly, with $r\not \equiv 0$) over mesh networks.  With a slight abuse of notation, below we denote  by $\kappa$, $L$ and $\mu$ the condition number, the smoothness constant and the strong convexity constant  of each $f_i$, respectively. We consider the most popular schemes,  employing gradient correction in the optimization direction--see Appendix~\ref{app:examples} for a description of these algorithms. We denote by  $\rho_2 \triangleq \rho_2({\bf W})$     the second largest eigenvalue of the gossip matrix ${\bf W}$ used in these algorithms  (note that $\rho_2=0$ for star-networks or fully-connected graphs).   The communication complexity of these algorithms when using the  ANQ is summarized next--they follow readily from Theorem~\ref{thm:comm_cost_mesh} and the convergence results of the unquantized algorithms in \cite{Xu2020}.

\begin{corollary}[ANQ-(Prox-)EXTRA,   ANQ-(Prox-)NIDS, and ANQ-NIDS over
mesh networks, see Appendices \ref{subsec:ABC_EXTRA}, \ref{subsec:ABC_NIDS}, and \ref{subsec:NIDS}]\label{coro:comm_cost_EXTRA_et_al} 
 Consider Problem \eqref{eq:P} over mesh networks, and the ANQ-(Prox-)EXTRA and     ANQ-{(Prox-)}NIDS  algorithms with stepsize  $\STEP = 2/(L+\mu)$ and the ANQ-NIDS algorithm with stepsize  $\STEP = 1/L$, where the  ANQ is  tuned as in Theorem~\ref{thm:comm_cost_mesh}. 
Then, the average number of bits/agent for $\Vert {\bf x}^k - {\bf x}^*\Vert_2^2 \leq \varepsilon$ is of the order of
\begin{align*}
\mathcal O\left(d\max\Big\{\kappa, \frac{1}{1{-}{\rho_2}}\Big\}\log\Big(\max\Big\{1{+}\kappa, \frac{1}{1{-}{\rho_2}}\Big\}\Big)\log (d/\varepsilon)\right),
\end{align*}
achieved in
\begin{align*}
{\mathcal O} \Big(\max\Big\{\kappa, \frac{1}{1- {\rho_2}}\Big\}\log (d/\varepsilon)\Big)\quad 
\mathrm{ iterations}
\end{align*}
and with 
\begin{align*}
{\mathcal O} \Big(d\log\Big(\max\Big\{1+\kappa, \frac{1}{1- {\rho_2}}\Big\}\Big)\Big)\quad 
\mathrm{ bits/agent/iteration}.
\end{align*}
\end{corollary}

\begin{corollary}[ANQ-(Prox-)NEXT and ANQ-(Prox-)DIGing over
mesh networks, see Appendices \ref{subsec:ABC_NEXT}-\ref{subsec:ABC_DIGing}]\label{coro:comm_cost_NEXT_et_al}   Consider Problem \eqref{eq:P} over mesh networks as described above, and the ANQ-(Prox-)NEXT and ANQ-(Prox-)DIGing algorithms with stepsize  $\STEP = 2/(L+\mu)$, where the ANQ is  tuned as in Theorem~\ref{thm:comm_cost_mesh}. Then, the average number of bits/agent for $\Vert {\bf x}^k - {\bf x}^*\Vert_2^2 \leq \varepsilon$ is of the order of
\begin{align*}
\mathcal O\!\left(\!d\max\Big\{\kappa,\frac{1}{(1{-}{\rho_2})^2}\Big\}\!\log\!\!\Big(\!\max\!\Big\{\!1{+}\kappa,\frac{1}{(1{-}{\rho_2)^2}}\Big\}\!\Big)\!\log\!(d/\varepsilon)\!\!\right)\!,
\end{align*}
achieved in
\begin{align*}
{\mathcal O} \Big(\max\Big\{\kappa, \frac{1}{(1- {\rho_2})^2}\Big\}\log (d/\varepsilon)\Big)\quad 
\mathrm{ iterations}
\end{align*}
and with 
\begin{align*}
{\mathcal O} \Big(d\log\Big(\max\Big\{1+\kappa, \frac{1}{(1- {\rho_2})^2}\Big\}\Big)\Big)\ 
\mathrm{ bits/agent/iteration}.
\end{align*}
\end{corollary}

\begin{corollary}[ANQ-Primal-Dual over
mesh networks, see Appendix \ref{subsec:Dual}]\label{coro:comm_cost_PD}   {Consider Problem \eqref{eq:P} with $r\equiv 0$ over mesh networks as described above, and the ANQ-Primal-Dual algorithm with stepsize $\STEP = \frac{2L\mu}{\mu\rho_{m-1}({\bf L}) + L\rho_{1}({\bf L})}$, where the ANQ is tuned as   in Theorem~\ref{thm:comm_cost_mesh}. Then, the average number of bits/agent for $\Vert {\bf x}^k - {\bf x}^*\Vert_2^2 \leq \varepsilon$ is of the order of
\footnote{Note that $\frac{\rho_1({\bf L})}{\rho_{m-1}({\bf L})}=\mathcal O((1-\rho_2)^{-1})$.}
\begin{align*}
\mathcal O\left(d\, \frac{\kappa}{1-\rho_2}\log\Big(1+\frac{\kappa}{1-\rho_2}\Big)\log(d/\varepsilon)
\right),
\end{align*}
achieved in
\begin{align*}
\mathcal O\bigg(
\frac{\kappa}{1-\rho_2}
\cdot 
\log(d/\varepsilon)
\bigg) \quad  \mathrm{iterations}
\end{align*}
and with 
\begin{equation}
\label{commcost1}
\mathcal O \bigg(d\cdot\log\Big(1+\frac{\kappa}{1-\rho_2}\Big)
\bigg) \quad  \mathrm{bits/agent/iteration}.
\end{equation}}
\end{corollary}

 {It is interesting to compare the ANQ-Primal-Dual   (Corollary~\ref{coro:comm_cost_PD}) with   \cite{Kovalev2020},
which applies quantization to a Primal-Dual scheme using   the widely adopted  compression rule.  The number of iterations required for the Option-D scheme (the best performing one) in \cite{Kovalev2020}   to achieve $\varepsilon$-accuracy,
using the rand-$K$ or dit-$K$ compression methods (see \cite[Example 4]{Kovalev2020}),
is\footnote{The asymptotic result in \cite[Theorem 5]{Kovalev2020} maps to our setting with the following modifications:
$\rho$ (the ratio between the largest and the second smallest eigenvalues of $\mathbf L$) and $\rho_\infty$ (the ratio of the largest weight and the second smallest eigenvalue of $\mathbf L$) defined therein are both of order $\mathcal O(1/(1-\rho_2))$;
when adopting the setting on the initial error from our Assumption \ref{assump:commcost} and the problem setting of \eqref{eq:P}, the parameter $m$ therein becomes 1 and the numerator in the $\log$ term becomes $d$.}
\begin{align}
\mathcal O\Big(\frac{d}{K}\frac{\kappa}{1-\rho_2}\log(d/\varepsilon)\Big), \label{eq:niter_optiond}
\end{align}
Therefore,
to achieve the same number of iterations as ANQ-Primal-Dual (in a $\mathcal O$-sense),
rand-$K$ in \cite{Kovalev2020} requires $K=\Theta(d)$,
whereas dit-$K$ in \cite{Kovalev2020} requires $S=2^{K-1}-1=\Theta(\sqrt{d})$,
with a communication cost of (in bits/agent/iteration)
\begin{align}\label{commcost2}
\begin{cases}
\mathcal O\Big(d+B\Big) &\text{(dit-$K$)},\\
\mathcal O\Big(d(B+\log d)\Big)&\text{(rand-$K$)},
\end{cases}
\end{align}
where $B$ is the number of bits used to encode each scalar \emph{with negligible loss in precision.}

Comparing the communication costs
\eqref{commcost1} and \eqref{commcost2}, the following remarks are in order:  
\textbf{(i)}  {If we neglect the $B$-dependence, our scheme has the same scaling behavior as dit-$K$
and better scaling behavior than rand-$K$, as the problem dimension $d$ increases.}
\textbf{(ii)} It is not clear how the number of bits $B$  should be chosen to encode scalars
in \cite{Kovalev2020} as a function of the system parameters, in a $\mathcal O$-sense, in order to 
\emph{make loss in precision truly negligible} (a condition required by the convergence analysis therein). On the other hand,
 our communication cost analysis, which focuses on
quantization below machine precision,
provides an explicit answer to this question.}

\section{Numerical Results} \label{sec:simulation}
In this section, we validate numerically  our theoretical findings and compare different distributed algorithms using quantization. We consider two instances of \eqref{eq:P}: a linear regression and a logistic regression problem,   both with $F$ strongly convex. 
 The communication network is modeled as an undirected graph of $\m=20$ agents,  generated by the Erdos-Renyi model with edge activation probability of $0.6$.
We measure performance of the algorithms using the mean square error  ${\rm MSE}^k$ and the network communication cost ${\rm C_{cm}}(\varepsilon)$ incurred over the network to reach ${\rm MSE}^k \leq \varepsilon$, defined as
\begin{align}
{\rm MSE}^k{\triangleq}\frac{\sum_{i=1}^\m \Vert {\bf x}_i^{k} - {\bf x}^*\Vert_2^2}{\m\Vert{\bf x}^*\Vert_2^2},\ 
{\rm C_{cm}}(\varepsilon){\triangleq}\sum_{k = 0}^{k_\varepsilon} \sum_{\ITERR=1}^\q \sum_{i=1}^\m b_{i}^{k,\ITERR},  \label{eq:sim_metric}
\end{align}
 where $k_\varepsilon \triangleq \min_{k \geq 0} {\rm MSE}^k \leq \varepsilon$ is the number of iterations required to achieve $\varepsilon$-accuracy, $\mathbf x^*$ is the optimal solution of \eqref{eq:P},
 and $b_{i}^{k, \ITERR}$ is the number of bits used by the  quantizer to encode the  $\ITERR$th transmitted signal by agent $i$ at iteration $k$.
\subsection{Linear regression problem} \label{sec:sim:ls}
\noindent \textbf{Problem setting:} Consider the following linear regression problem  over mesh networks: 
\begin{align}
f_i({\bf x}){=}\frac{1}{2}\Vert {\bf U}_i {\bf x} - {\bf v}_i\Vert_2^2 + \frac{0.01}{2}\Vert {\bf x}\Vert_2^2\ \text{and}\ 
r({\bf x}){=}\alpha \Vert {\bf x}\Vert_1, \label{eq:problem_ls}
\end{align}
where ${\bf U}_i \in \mathbb R^{20 \times 40}$   and ${\bf v}_i \in \mathbb R^{20 \times 1}$ are the feature vector and
 observation measurements, respectively, accessible only by agent $i$. 
The matrix $\mathbf U_i$ is generated   independently across agents, according to the model in \cite{Agarwal2010}, namely: $[\mathbf U_i]_{:,1}\sim\mathcal N(\mathbf 0,\frac{1}{1-\beta^2}\mathbf I)$ (first column), and for the other  columns $q>1$,
 $[{\bf U}_{i}]_q|[{\bf U}_{i}]_{q-1}\sim
\mathcal N(\beta[{\bf U}_{i}]_{q-1},\mathbf I)$, where $\beta=0.3$.
In this way, each row of ${\bf U}_i$ is a Gaussian random vector with zero mean and covariance depending on $\beta$: larger $\beta$ generates more ill-conditioned covariance matrices.
Then, letting ${\bf x}_0 \in \mathbb R^{40}$ be the ground truth vector, generated as a sparse vector with $70\%$ zero entries, and  i.i.d. nonzero entries  drawn from  $\mathcal N(0,1)$, we generate  ${\bf v}_i$
as ${\bf v}|({\bf U},{\bf x}_0)\sim\mathcal N({\bf U}{\bf x}_0, 0.04{\bf I})$.
We use $\mu, L$ for   the strong convexity and smoothness parameters of each $f_i$, respectively.

 We test several distributed algorithms considering either smooth   ($\alpha = 0$) or non-smooth ($\alpha>0$) instances of the least square problem \eqref{eq:problem_ls}. In fact, most of the existing quantization schemes are applicable only to smooth optimization problems. 
The free parameters of these algorithms are optimized based on the recommendations in the original papers, while optimizing numerical convergence with the smallest number of bits, unless otherwise stated; the weight matrix $\tilde{\bf W}$ used to  mix the received signals is constructed according to the Metropolis-Hastings rule \cite{Xiao2004}; the number of bits transmitted by each scheme as reported in the figures is per agent, per dimension, per iteration.
For each quantized algorithm, we choose $\RTQ = 0.99 \cdot \RTUQ + 0.01$, where  $\RTUQ = ({\rm MSE}^{100}/{\rm MSE}^{50})^{0.01}$
is the numerical estimate of the  convergence rate of its unquantized counterpart  {(i.e., implemented at machine precision)}. In the simulations, all algorithms except those using LPQ are evaluated with 1 realization since they are deterministic algorithms. Those using the probabilistic quantizer LPQ, i.e., LEAD \cite{Liu2021} and COLD \cite{Zhang2021}, are averaged over 10 realizations, with fixed ${\bf U}_i, {\bf V}_i, {\bf x}^0$, and network topology.

\begin{figure*}
     \centering
     \hfill
     \begin{subfigure}[b]{0.45\textwidth}
         \includegraphics[width = .9\linewidth,trim={0 50 40 80},clip]{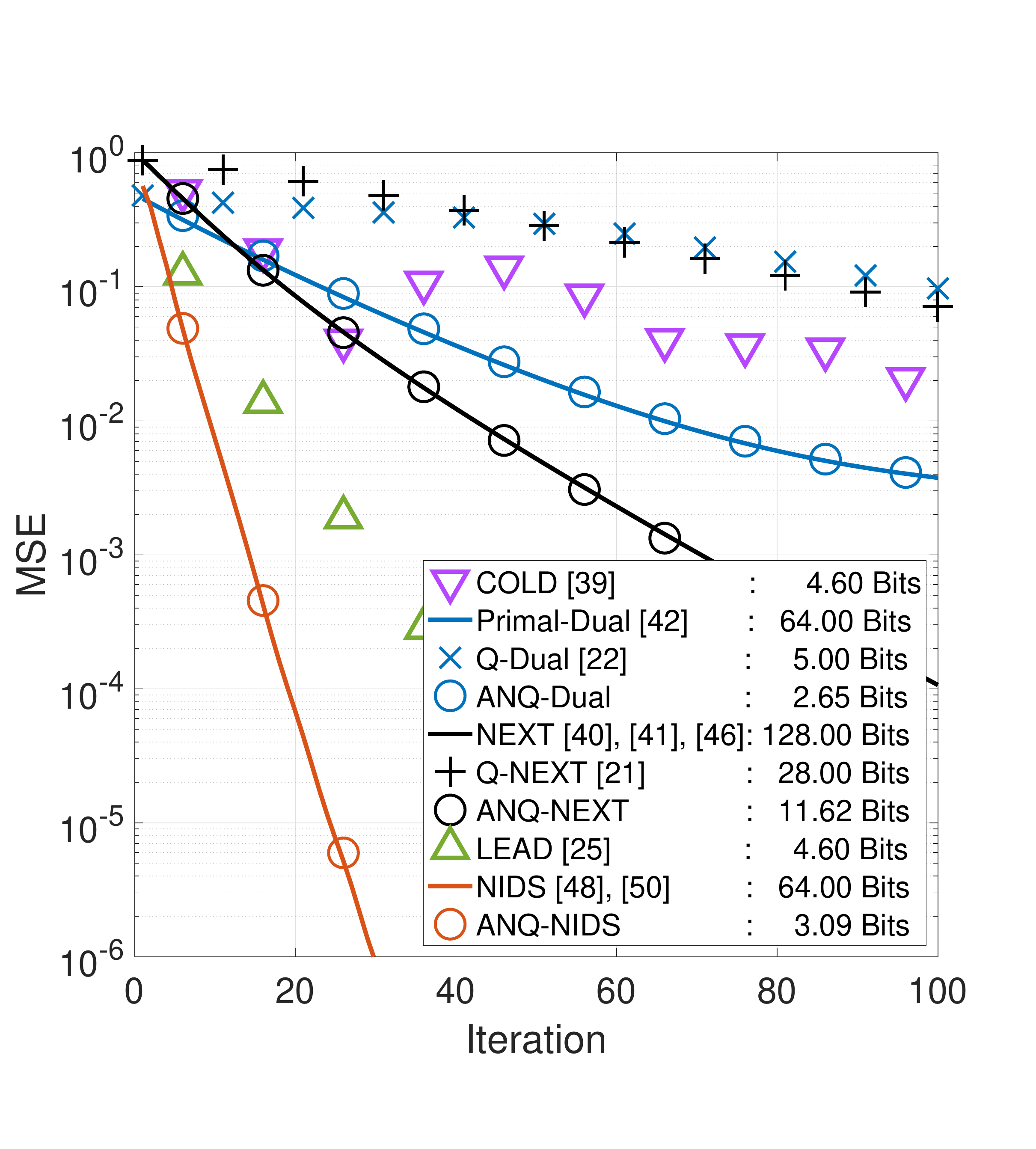}	
	    \caption{Smooth linear regression
	    } \label{fig:LS:smooth}
     \end{subfigure}
     \hfill
     \begin{subfigure}[b]{0.45\textwidth}
        \includegraphics[width = .9\linewidth,trim={0 50 40 80},clip]{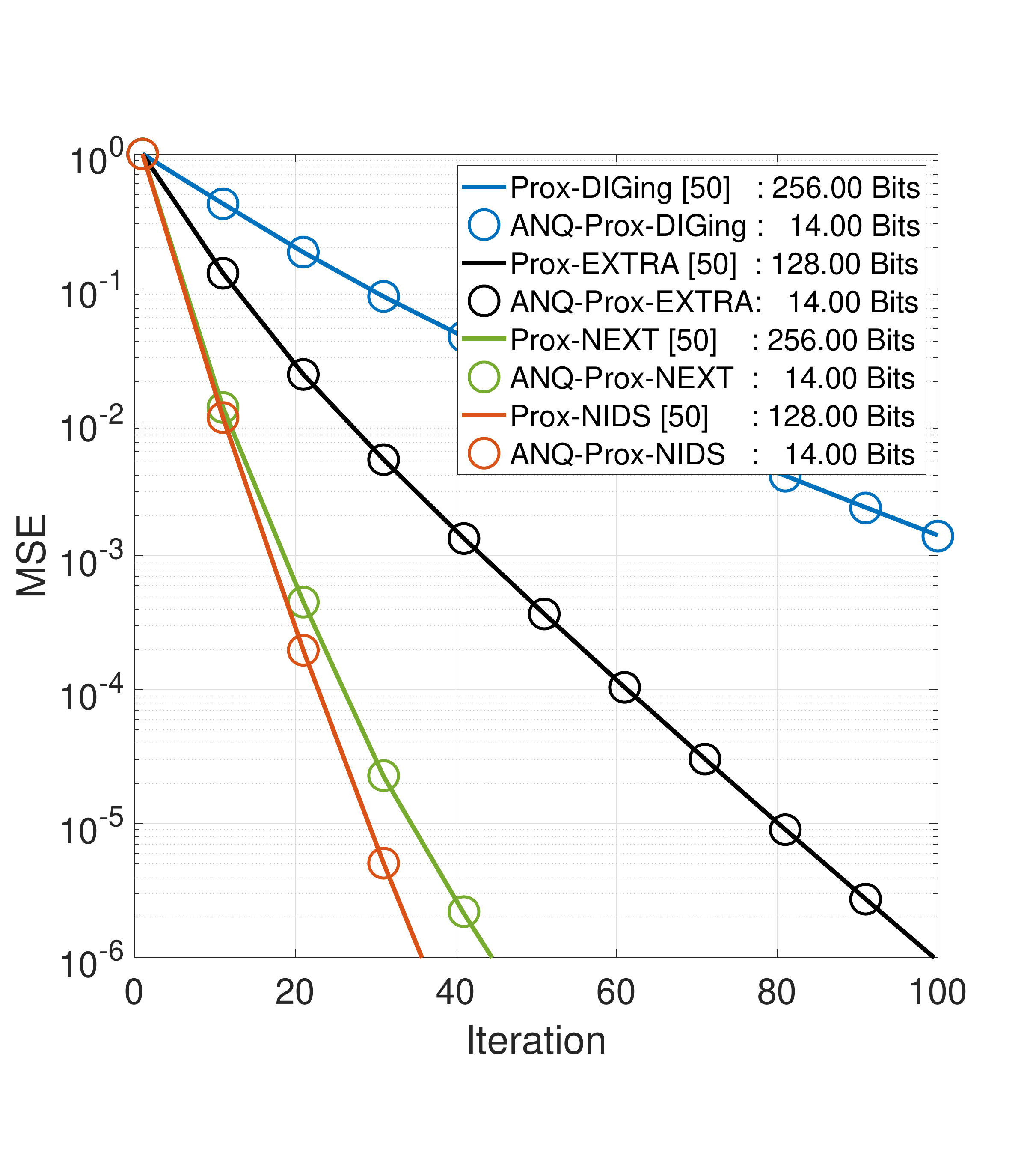}
	    \caption{Non-smooth linear regression} \label{fig:LS:nonsmooth}
     \end{subfigure}
     \hfill
     \caption{Linear regression problem  \eqref{eq:problem_ls}: MSE versus iterations.
    Solid curves and markers refer to  algorithms implemented using machine precision and quantized communications, respectively. In the legend, for each scheme, we report the number of bits transmitted, per agent, per dimension, per iteration.}
\end{figure*}

\noindent \textbf{Smooth linear regression (Fig. \ref{fig:LS:smooth}):} We begin by considering the smooth linear regression problem. We consider the following  
{benchmark schemes, using  machine precision (64-bit representation for each scalar):}
\begin{enumerate}
     \item[1)] \texttt{Primal-Dual} \cite{Uribe2021} with step-size $\STEP = 2L\mu/(\mu\rho_{m-1}({\bf L}) + L\rho_1({\bf L}))$ (\cite[Proposition 2]{Magnusson2020}), where ${\bf L}$ is the graph Laplacian matrix associated with the graph.
   \item[2)]   \texttt{NEXT} \cite{Lorenzo2016, Qu2018, Sun2019} with step-size  $\STEP = 0.0029$, manually tuned for fastest practical  convergence.
 \item[3)]  \texttt{NIDS} \cite{Xu2020, Li2019} with step-size  $\STEP = \frac{2}{L+\mu}$  and mixing matrix ${\bf W}=[(1+\nu){\mathbf I}+(1-\nu)\tilde{\mathbf W}]/2$,  with $\nu = 0.001$.
\end{enumerate}
In addition, we consider the following algorithms, that implement the above benchmark schemes using quantized communications:
\begin{enumerate}
	\item[4)] \texttt{Q-Dual} \cite{Magnusson2020} and \texttt{Q-NEXT} \cite{Kajiyama2020};
    \item[5)] \texttt{ANQ-Dual}: this is the Primal-Dual algorithm   \cite{Uribe2021} equipped with the proposed  {deterministic} ANQ  (see Appendix~\ref{subsec:Dual}), with $\QNSB^0 = 0.01$ and $\QNSC = \bar{\QNSC}/2$ [recall that $\bar{\QNSC}$    is defined in \eqref{omega_bar}];
    \item[6)] \texttt{ANQ-NEXT}: this is
    the NEXT algorithm \cite{Lorenzo2016, Qu2018, Sun2019} quantized using the deterministic ANQ with  $\QNSB^0 = 0.029$ and $\QNSC = \bar{\QNSC}/2$;
    \item[7)] \texttt{ANQ-NIDS}: this is an instance of the NIDS algorithm \cite{Xu2020, Li2019} equipped with the  {deterministic} ANQ  (see Appendix \ref{subsec:NIDS}) with parameters $\QNSB^0 = 0.1$ and $\QNSC = \bar{\QNSC}/2$
    {\item[8)] \texttt{LEAD} \cite{Liu2021} and \texttt{COLD} \cite{Zhang2021},
    both implemented using the low-precision quantization (LPQ) 
    \cite{Taheri2020}:
    to transmit a signal $\mathbf x$,
the amplitude $\Vert {\bf x}\Vert_2$ is encoded at machine precision (64 bits), and 3 bits are adopted to encode each normalized element $x_i/\Vert {\bf x}\Vert_2$.\label{hltext1}}
\end{enumerate}

In Fig.~\ref{fig:LS:smooth},  we plot  the MSE  versus {iteration} index $k$. 
 {Remarkably, all algorithms, when equipped with the proposed ANQ, incur a negligible loss of convergence speed with respect to their machine precision counterpart.}
 Comparing  ANQ with the compression-based distributed algorithms LEAD \cite{Liu2021} and COLD \cite{Zhang2021}, we infer that the proposed ANQ-NIDS is faster while using   less bits.
  Comparing
 {ANQ with the state-of-the-art quantized algorithms,} we notice that ANQ is more communication-efficient than
 {Q-NEXT and Q-Dual, which instead use}
 deterministic uniform quantizers with shrinking range:  
 ANQ-NEXT (11.62 bits) and  ANQ-Dual (2.65 bits) use less bits per iteration  than   Q-NEXT   (28 bits) and Q-Dual (5 bits), respectively, while  converging faster. Note that, with the parameters chosen as recommended in Theorem \ref{thm:conv}, all ANQ-based algorithms shown in the figure have convergence guarantees,
 while Q-Dual and Q-NEXT do not  in the simulated setting.
 In fact, Q-Dual and Q-NEXT use 5 and 28 bits, respectively,
 which fall below the minimum number of bits that guarantee linear convergence, calculated to be 13 from \cite[Theorem 1]{Magnusson2020} and 78  from \cite[Theorem 4]{Kajiyama2020}, respectively.

\noindent\textbf{Non-smooth linear regression (Fig. \ref{fig:LS:nonsmooth}):}  We now move to the non-smooth instance of \eqref{eq:problem_ls}, with $\alpha = 10^{-4}$. To our knowledge, there is no existing quantized algorithms solving such instance of \eqref{eq:problem_ls}. Hence,
 we tested the following ANQ-based quantized algorithms:   
\begin{enumerate}
  \item[1)] \texttt{ANQ-Prox-EXTRA}: this is an instance of the Prox-EXTRA algorithm \cite{Xu2020} equipped with the  {deterministic} ANQ  (see Appendix \ref{subsec:ABC_EXTRA}) with parameters $\QNSB^0 = 6.67\times 10^{-4}$ and $\QNSC = \bar{\QNSC}/2$;
  \item[2)] \texttt{ANQ-Prox-NEXT}: this is   the Prox-NEXT algorithm \cite{Xu2020} equipped with the {deterministic} ANQ  (see Appendix \ref{subsec:ABC_NEXT}) with parameters $\QNSB^0 = 2.34\times 10^{-3}$ and $\QNSC = \bar{\QNSC}/2$;
  \item[3)] \texttt{ANQ-Prox-DIGing}: this is   the Prox-DIGing algorithm \cite{Xu2020} equipped with the  {deterministic} ANQ  (see Appendix \ref{subsec:ABC_DIGing}) with parameters $\QNSB^0 = 3.05 \times 10^{-3}$ and $\QNSC = \bar{\QNSC}/2$;
  \item[4)] \texttt{ANQ-Prox-NIDS}: this is   the Prox-NIDS algorithm in \cite{Xu2020} equipped with the  {deterministic} ANQ  (see Appendix \ref{subsec:ABC_NIDS}) with parameters $\QNSB^0 = 7.7 \times 10^{-4}$ and $\QNSC = \bar{\QNSC}/2$.
\end{enumerate}

\begin{figure*}
     \centering
     \hfill
     \begin{subfigure}[b]{0.45\textwidth}
         \includegraphics[width = .9\linewidth,trim={0 50 40 80},clip]{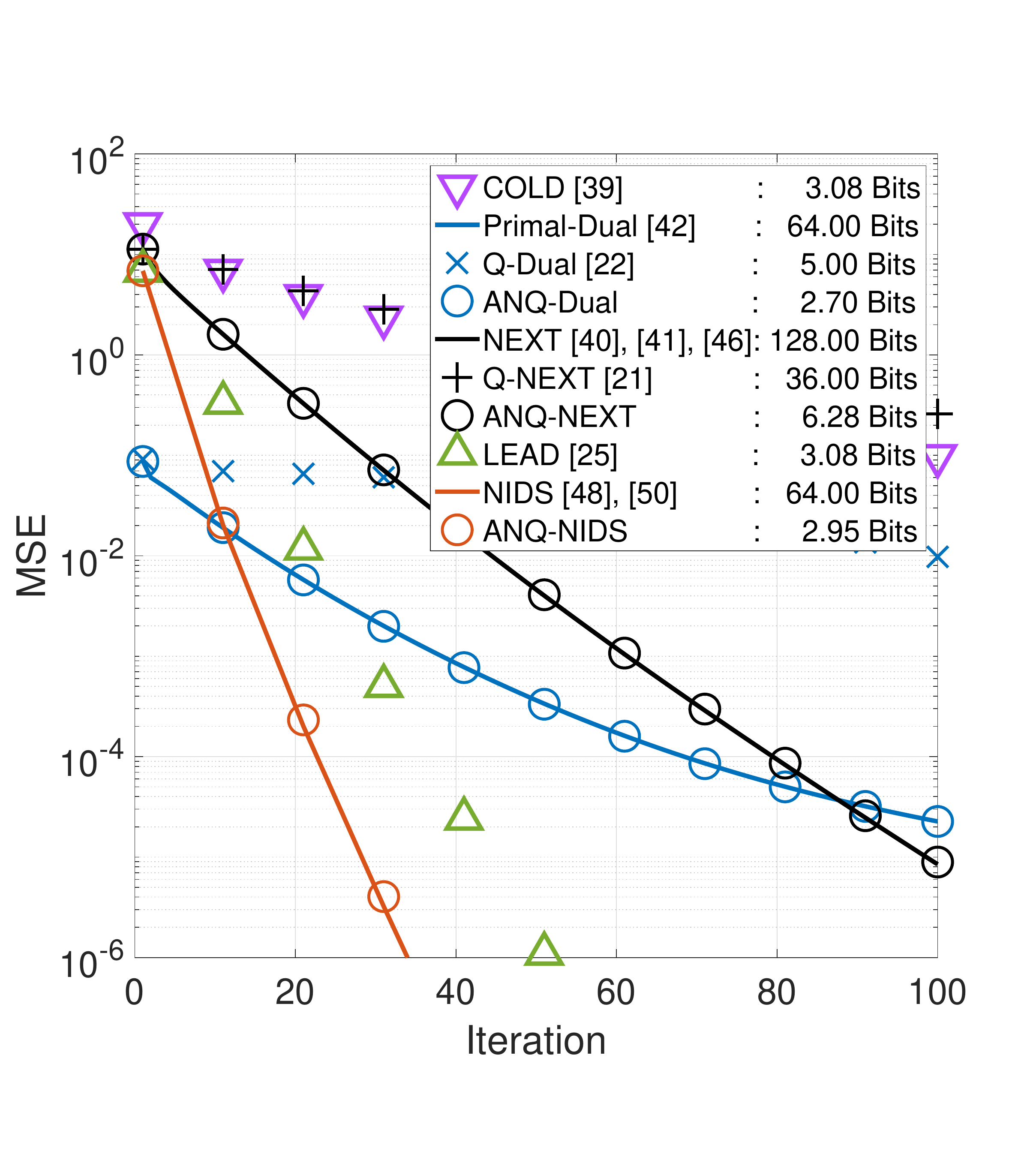}	
	    \caption{Smooth logistic regression
	    } \label{fig:logit:smooth}
     \end{subfigure}
     \hfill
     \begin{subfigure}[b]{0.45\textwidth}
        \includegraphics[width = .9\linewidth,trim={0 50 40 80},clip]{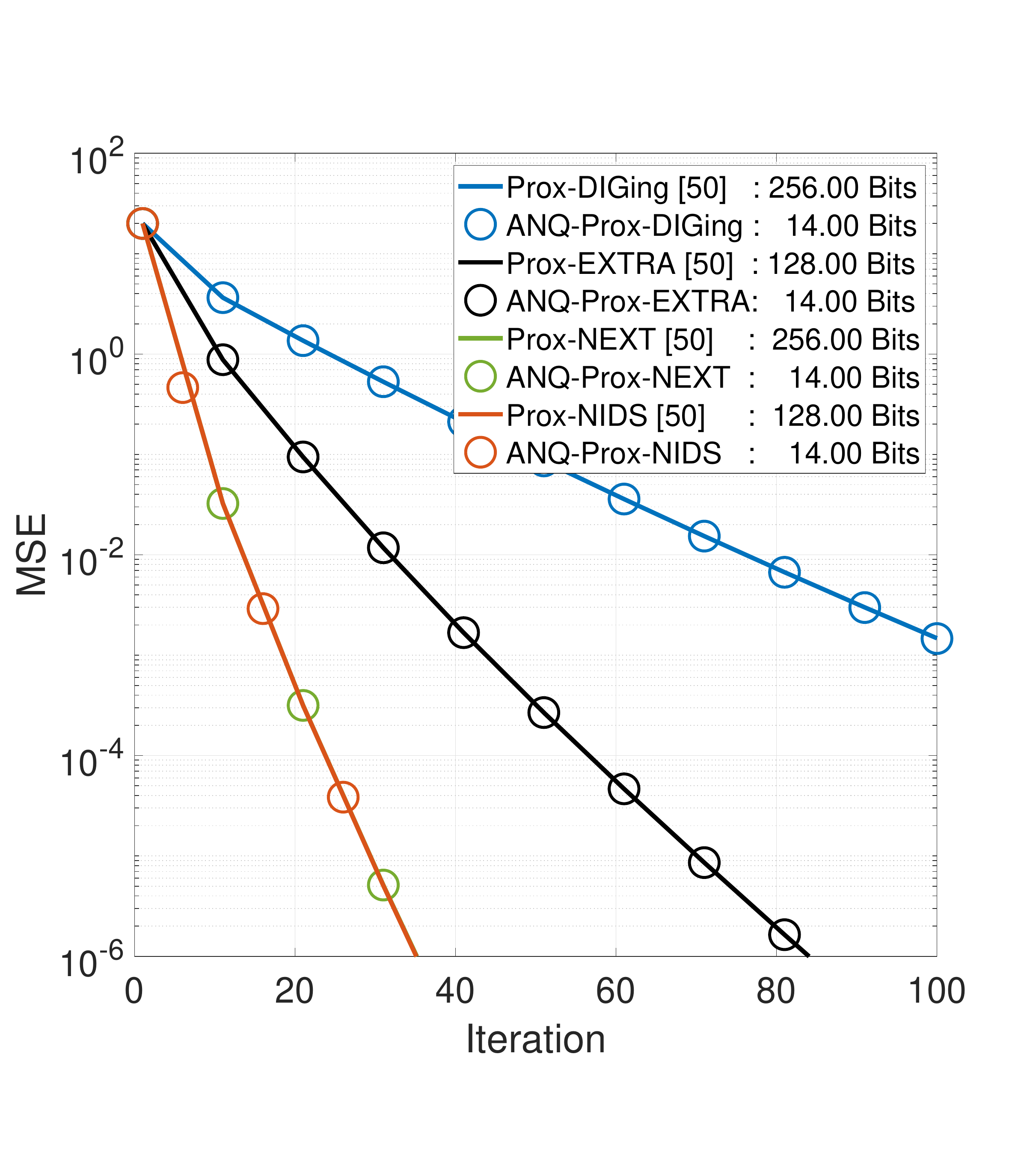}
	    \caption{Non-smooth logistic regression } \label{fig:logit:nonsmooth}
     \end{subfigure}
     \hfill
     \caption{Logistic regression problem  \eqref{eq:problem_logit}: MSE versus iterations.
    Solid curves and markers refer to  algorithms implemented using machine precision and quantized communications, respectively. In the legend, for each scheme, we report the number of bits transmitted, per agent, per dimension, per iteration.}
\end{figure*}

As benchmark, we also included their unquantized counterparts, {implemented at machine precision}; in all these schemes, we used the weight matrix  ${\bf W}=[(1+\nu){\mathbf I}+(1-\nu)\tilde{\mathbf W}]/2$ with $\nu = 0.001$; the step-size is chosen according to  \cite{Xu2020}, namely:   $\STEP = \frac{2\rho_{m}({\bf W})}{L + \mu \rho_{m}({\bf W})}$ for Prox-EXTRA, $\STEP = \frac{2\rho_{m}({\bf W}^2)}{L + \mu \rho_{m}({\bf W}^2)}$ for Prox-DIGing, and $\STEP = \frac{2}{L + \mu}$ for Prox-NEXT and Prox-NIDS.

Fig. \ref{fig:LS:nonsmooth} plots the MSE achieved by all the algorithms versus the iteration index.  As predicted,  all four quantized schemes converge linearly.
 Remarkably, all of the ANQ-equipped algorithms incur a negligible loss of convergence speed with respect to their machine precision counterparts, while using a 
 fraction of the communication budget -- only $14$ bits per agent/dimension/iteration.

\subsection{Logistic regression}  \label{sec:sim:logit}
We now consider the distributed logistic regression problem using  the MNIST dataset \cite{Lecun1998}. This is an instance of \eqref{eq:P} with
\begin{align}\nonumber
&f_i({\bf x}) = \frac{0.01}{2} \Vert {\bf x}\Vert_2^2 + \frac{1}{3000}\sum_{p=1}^{3000}\ln\Big(1+\exp\big(-v_{i, p} {\bf u}_{i, p}^\top{\bf x}\big)\Big),\\&  \quad \text{and} \quad r({\bf x}) =  \alpha \Vert {\bf x}\Vert_1, \label{eq:problem_logit}
\end{align}
where ${\bf u}_{i, p} \in \mathbb R^{784 \times 1}$ and $v_{i, p} \in \{-1, 1\}$ are the feature vector and
labels, respectively, only accessible by agent $i$. Here we implement the one-vs.-all scheme, i.e., the goal is to distinguish the data of label '0' from others. To generate ${\bf u}_{i,p}$, we first flatten each picture of size $28 \times 28$ in MNIST into a real feature vector of length $28 \times 28 = 784$,
and then normalize it to unit $l_2$ norm. We then allocate equal number of feature vectors and labels to each agent. In the simulations, all algorithms except those using LPQ are evaluated with 1 realization since they are deterministic algorithms. Those using the probabilistic quantizer LPQ, i.e., LEAD \cite{Liu2021} and COLD \cite{Zhang2021}, are averaged over 10 realizations, with fixed feature/label allocations and network topology. 
 
 \noindent \textbf{Smooth logistic regression (Fig. \ref{fig:logit:smooth}):} We begin by  considering the smooth logistic regression problem \eqref{eq:problem_logit}, with $\alpha = 0$. We tested the same algorithms (with the same tuning)  as described in \secref{sec:sim:ls} for the smooth linear regression problem.
In Fig.~\ref{fig:logit:smooth},  we plot  the MSE  versus {iteration} index $k$. Consistently with the results in  Fig.~\ref{fig:LS:smooth}, we notice the following facts.  ANQ-NIDS achieves the fastest convergence, followed by ANQ-NEXT, ANQ-Dual, Q-Dual   and Q-NEXT. Comparing our quantization method with existing ones  on the same unquantized algorithm, we notice that the proposed ANQ is more communication-efficient than Q-NEXT and Q-Dual:
 ANQ-NEXT (6.28 bits) and  ANQ-Dual (2.7 bits) use less bits per iteration  than   Q-NEXT  (36 bits) and Q-Dual (5 bits), respectively, while at the same time converging faster. Comparing with the compression-based algorithms LEAD and COLD, it is shown that ANQ-NIDS achieves  better convergence rate with less bits.

 \noindent\textbf{Non-smooth logistic regression (Fig. \ref{fig:logit:nonsmooth}):}  We now consider the non-smooth instance of the logistic regression problem    \eqref{eq:problem_logit}, with $\alpha = 10^{-4}$. We tested the same algorithms (with the same tuning)  as described in \secref{sec:sim:ls} for the non-smooth linear regression problem.
Fig. \ref{fig:logit:nonsmooth} plots the MSE achieved by all the algorithms versus   iterations $k$. The results confirm the trends already commented in   Fig. \ref{fig:LS:nonsmooth}.

\begin{figure*}
     \centering
     \hfill
     \begin{subfigure}[b]{0.45\textwidth}
         \includegraphics[width =.9\linewidth,trim={0 40 40 70},clip]{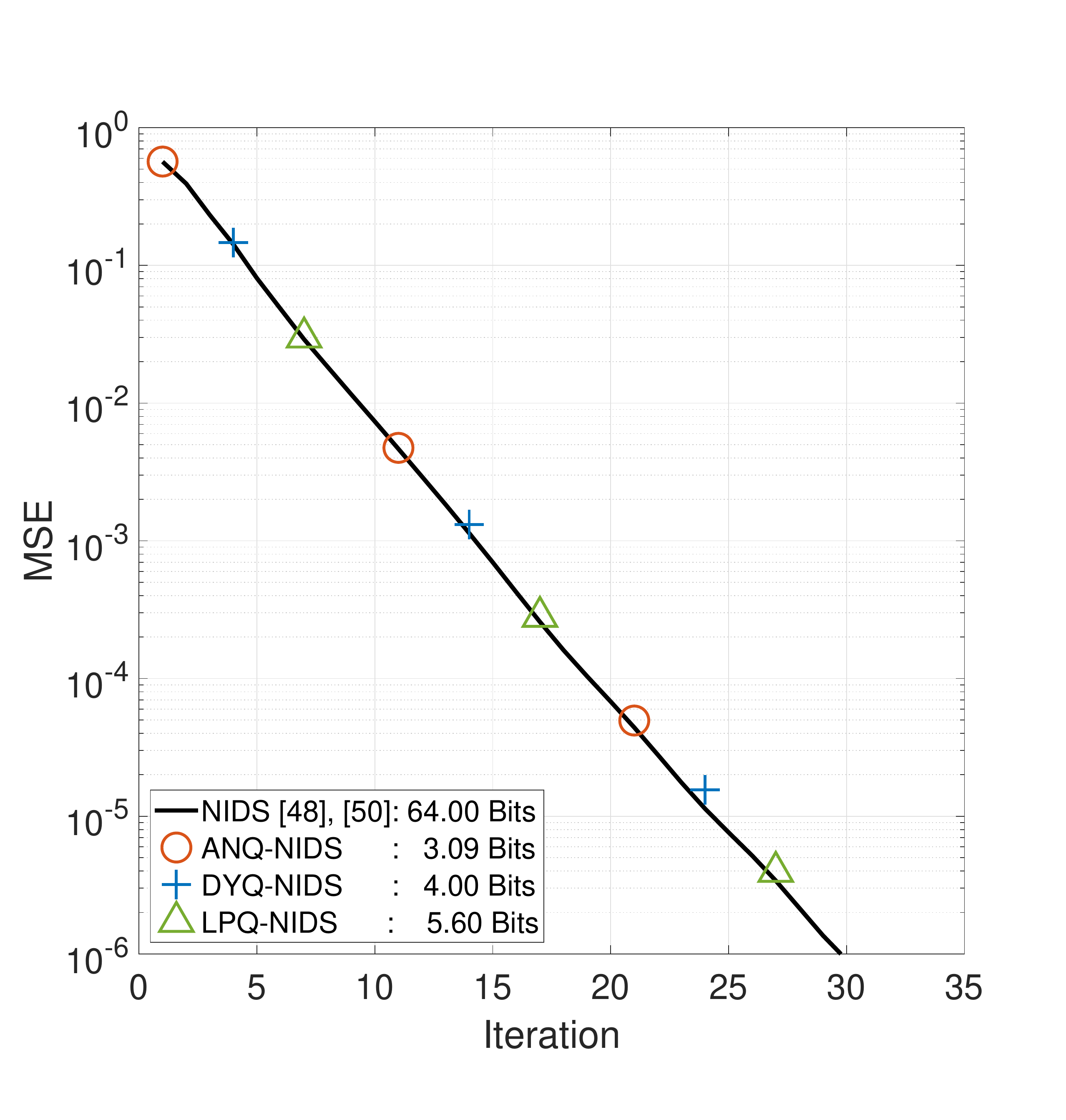}
	    \caption{Smooth linear regression  \eqref{eq:problem_ls}
	    } \label{fig:quant:LS}
     \end{subfigure}
     \hfill
     \begin{subfigure}[b]{0.45\textwidth}
        \includegraphics[width =.9\linewidth,trim={0 40 40 70},clip]{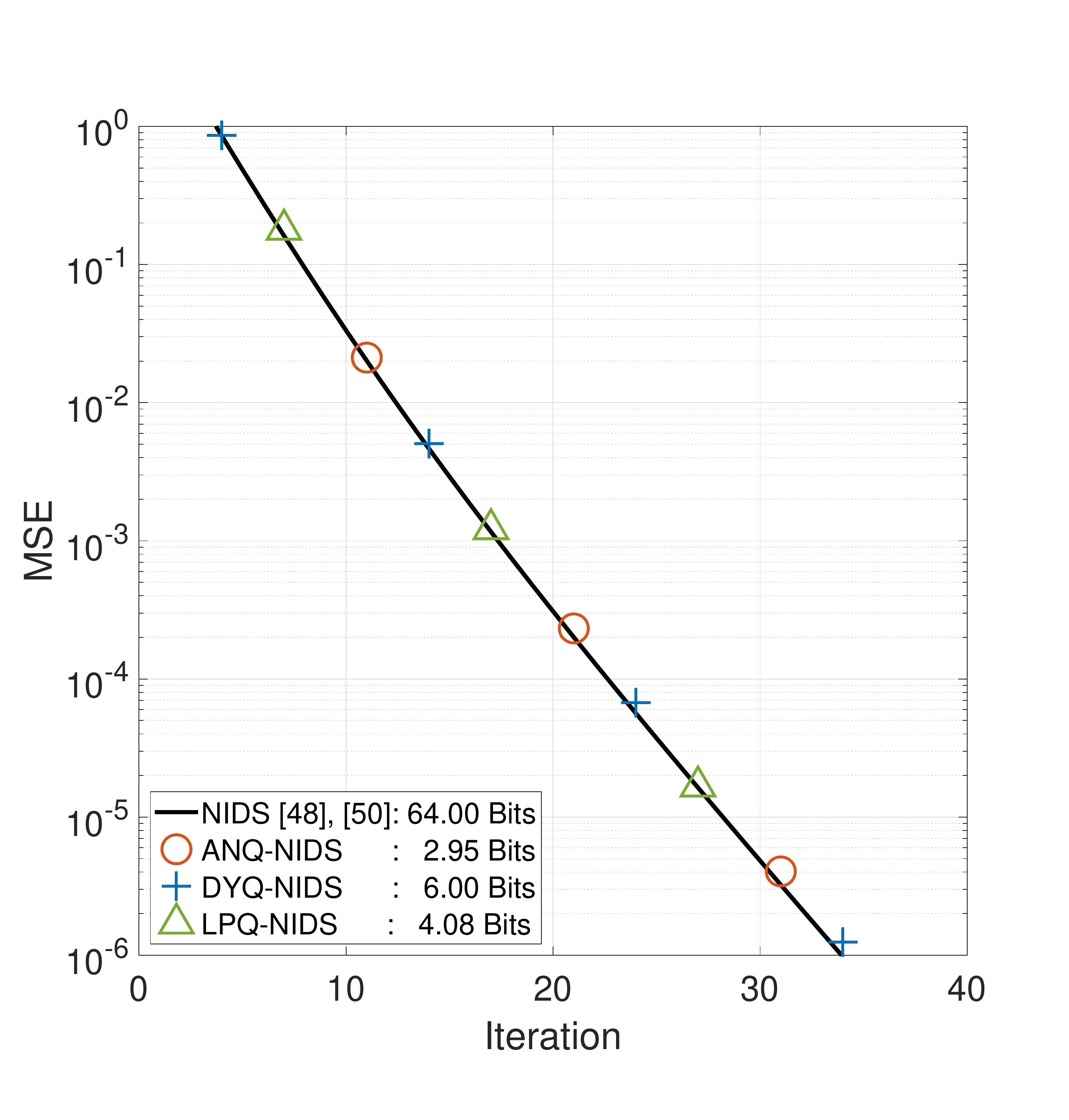}
	    \caption{Smooth logistic regression  \eqref{eq:problem_logit}} \label{fig:quant:logit}
     \end{subfigure}
     \hfill
     \caption{MSE versus iterations
      for different quantization rules applied to NIDS \cite{Li2019, Xu2020}.}
\end{figure*}

\subsection{Comparison of quantization rules} \label{subsec:sim:quant}
From the above results, it is clear that ANQ-NIDS outperforms existing algorithms using dynamic quantization (DYQ)--including Q-Dual and Q-NEXT--as well as those employing  low-precision quantization (LPQ)--such as  LEAD and COLD.
This advantage may be due to the \emph{black-box} nature of ANQ:
it can be applied to a variety of  distributed  algorithms, including those  known in the literature to be the 
fastest ones. This is a significant advantage over ad-hoc quantization schemes, which  are limited in their applicability to the specific algorithms   they are designed for.
Therefore, an interesting question is whether  the performance superiority comes solely from the underlying (unquantized) algorithm, i.e., NIDS, or also from the quantization rule, i.e., ANQ. To answer this question, we compare the above three quantization rules, DYQ, LPQ, and ANQ on the same distributed algorithm, NIDS \cite{Li2019, Xu2020}.

Fig. \ref{fig:quant:LS} and Fig. \ref{fig:quant:logit} show the MSE versus iterations for ANQ-NIDS, LPQ-NIDS, and DYQ-NIDS solving the smooth linear regression and smooth logistic regression problems \eqref{eq:problem_ls} and \eqref{eq:problem_logit} ($r\equiv0$), respectively.
The parameters for DYQ-NIDS and LPQ-NIDS are selected to closely match the performance of NIDS with machine precision while using the smallest number of bits, whereas  those for ANQ-NIDS are selected according to our  analysis.
Both figures show that ANQ-NIDS consistently requires less bits than the other quantization schemes. 
More precisely, ANQ-NIDS uses 25\% less bits than DYQ-NIDS and 44\% less then LPQ-NIDS in the linear regression case, and 50\% less bits than DYQ-NIDS and 27\% less then LPQ-NIDS in the logistic regression problem.

\subsection{Communication cost}  \label{sec:sim:cc}
We now study the effect of the dimension $d$ on the communication cost for different algorithms solving the non-smooth linear regression problem  \eqref{eq:problem_ls}, with  $\alpha=10^{-4}$. Note that the rate of a machine precision algorithm $\RTUQ$ depends on both the weight matrix and the condition number $\kappa = L/\SC$, which  depends itself on  $d$. We chose the coefficient of the $l_2$ regularizer so as to make  $\kappa$ and thus $\RTUQ$ remain fixed across different $d$. The rest of the settings are the same as in Fig. \ref{fig:LS:nonsmooth}. Fig. \ref{fig:comm_cost_scale}    plots the network communication cost ${\rm C_{cm}}(\varepsilon)$ versus   $d$ as defined in \eqref{eq:sim_metric}, required to reach a target MSE-accuracy $\varepsilon = 10^{-8}$. We observe that ${\rm C_{cm}}$ scales roughly linearly with respect to the dimension for all algorithms, which is consistent with Theorem \ref{thm:comm_cost_mesh}.

\begin{figure}[h]
	\centering
	\includegraphics[width = 0.81\linewidth,trim={30 20 40 40},clip]{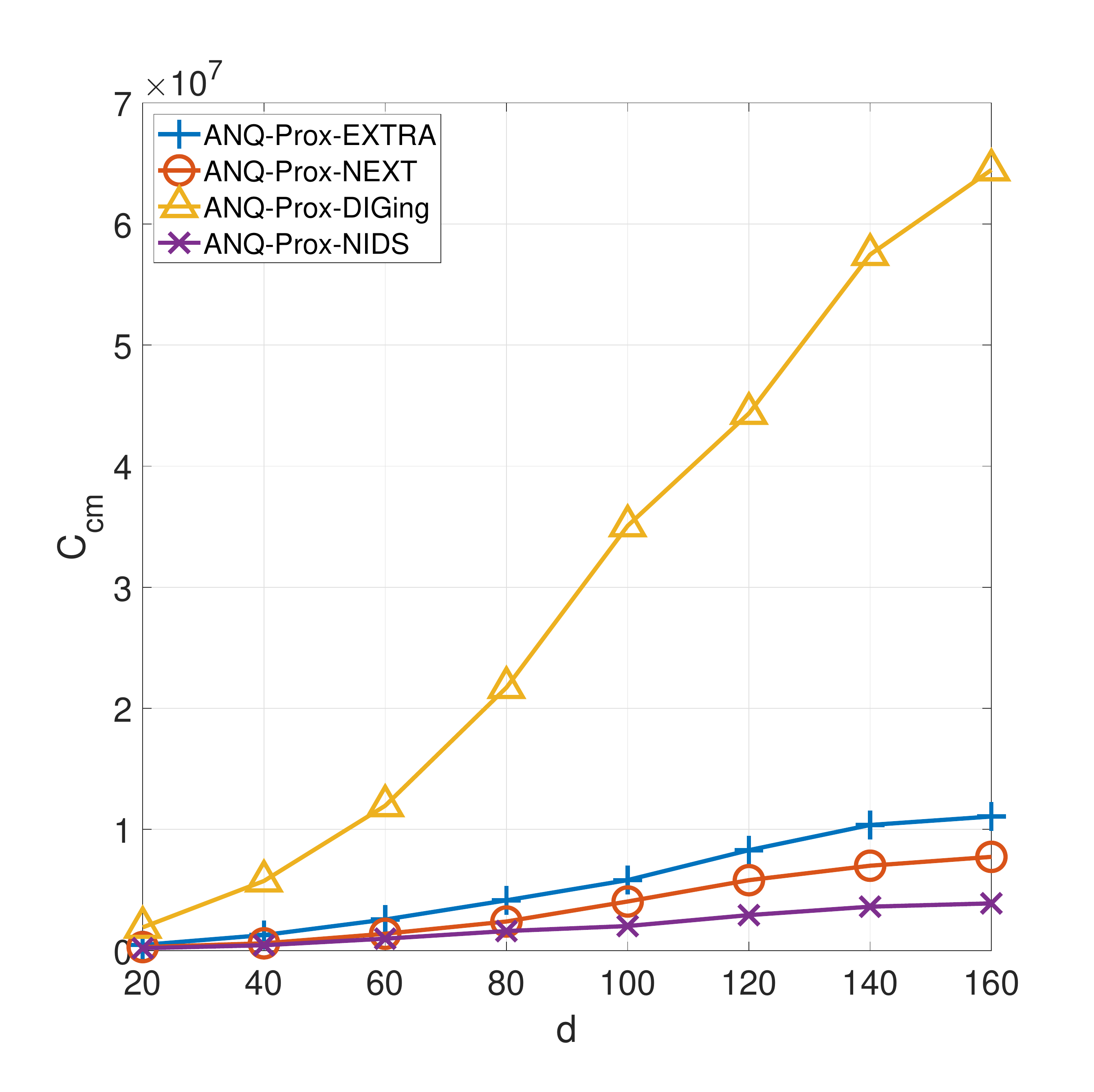}	
	\caption{Network communication cost evaluation on	non-smooth linear regression problem with $\alpha = 10^{-4}$ versus $d$, where $\RTUQ$ remains fixed for all $d$ in each algorithm.}
 \label{fig:comm_cost_scale}
\end{figure}

Finally, we investigate numerically the effect of $\RTQ$ and $\QNSC$ on the network communication cost as defined in  \eqref{eq:sim_metric}, for a target MSE-accuracy $\varepsilon = 10^{-14}$. We consider the ANQ-NIDS algorithm with $\eta^0 = 0.001$, solving the smooth linear regression problem  \eqref{eq:problem_ls}, with $\alpha = 0$. Fig. \ref{fig:commcost_sigma} plots the network communication cost versus $\RTQ$ with $\QNSC = \bar{\omega}/2$. Note that this figure justifies the discussion in \secref{sec:commcost} that $\RTQ$ should be chosen away from $\RTUQ$ and 1 in order to save on communication cost.
Fig. \ref{fig:commcost_sigma} plots the communication cost \eqref{eq:sim_metric} versus $\QNSC$ with $\RTQ = 0.99 \times \RTUQ + 0.01$.
It can be seen that, by optimizing the compression rate $\QNSC$, a saving of  15\% in communication cost can be obtained over a quantization scheme that employs no compression ($\QNSC=0$).
This observation numerically supports our  BC-rule, which generalizes the deterministic/probabilistic quantizers that have no compression term. 

\begin{figure*}
     \centering
     \hfill
     \begin{subfigure}[b]{0.45\textwidth}
        \includegraphics[width = .9\linewidth,trim={0 60 50 70},clip]{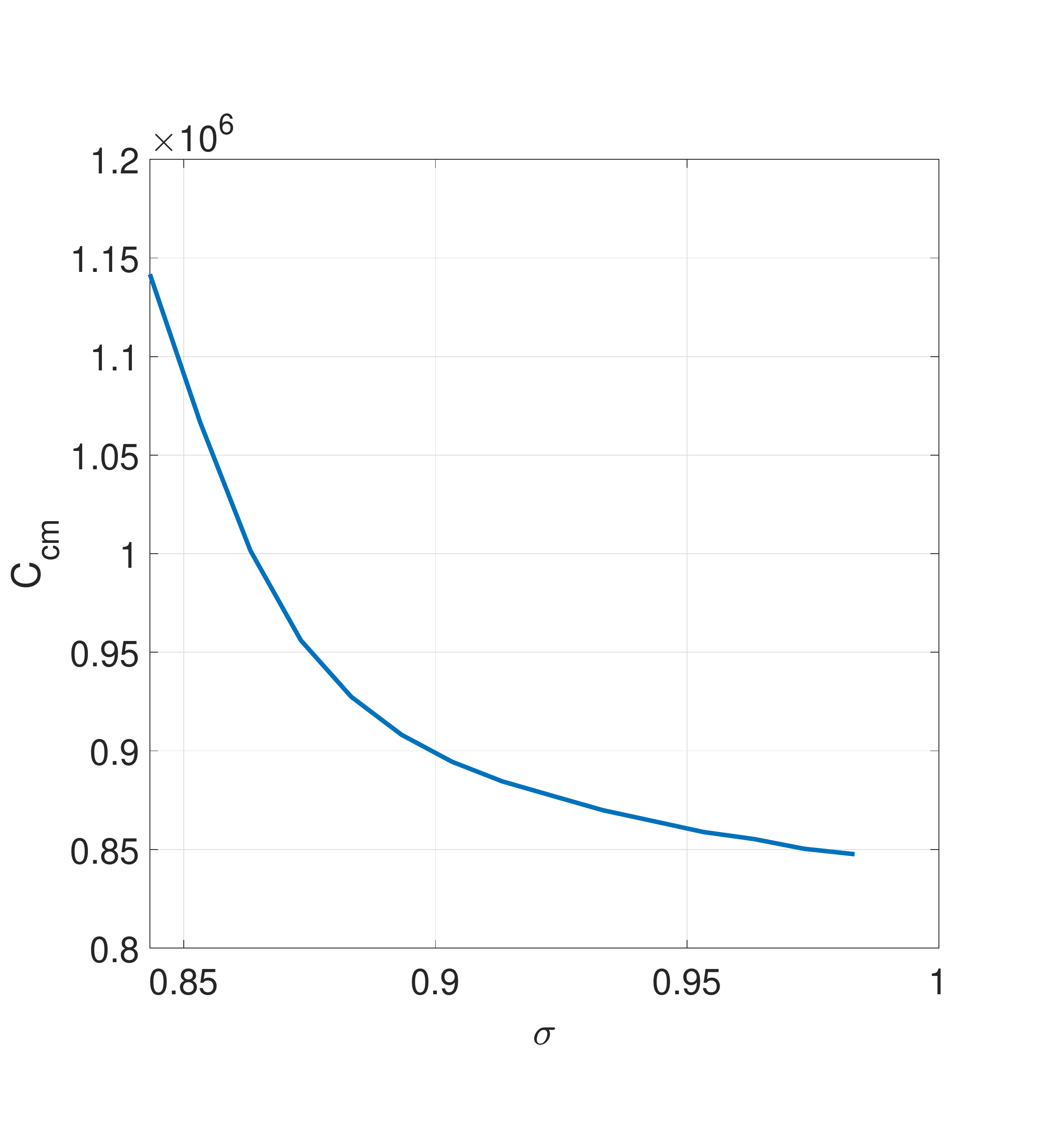}
	    \caption{Effect of $\RTQ\in(\lambda,1)$ with $\QNSC = \bar{\QNSC}/2$} \label{fig:commcost_sigma}
     \end{subfigure}
     \hfill
     \begin{subfigure}[b]{0.45\textwidth}
        \includegraphics[width = .9\linewidth,trim={0 60 50 70},clip]{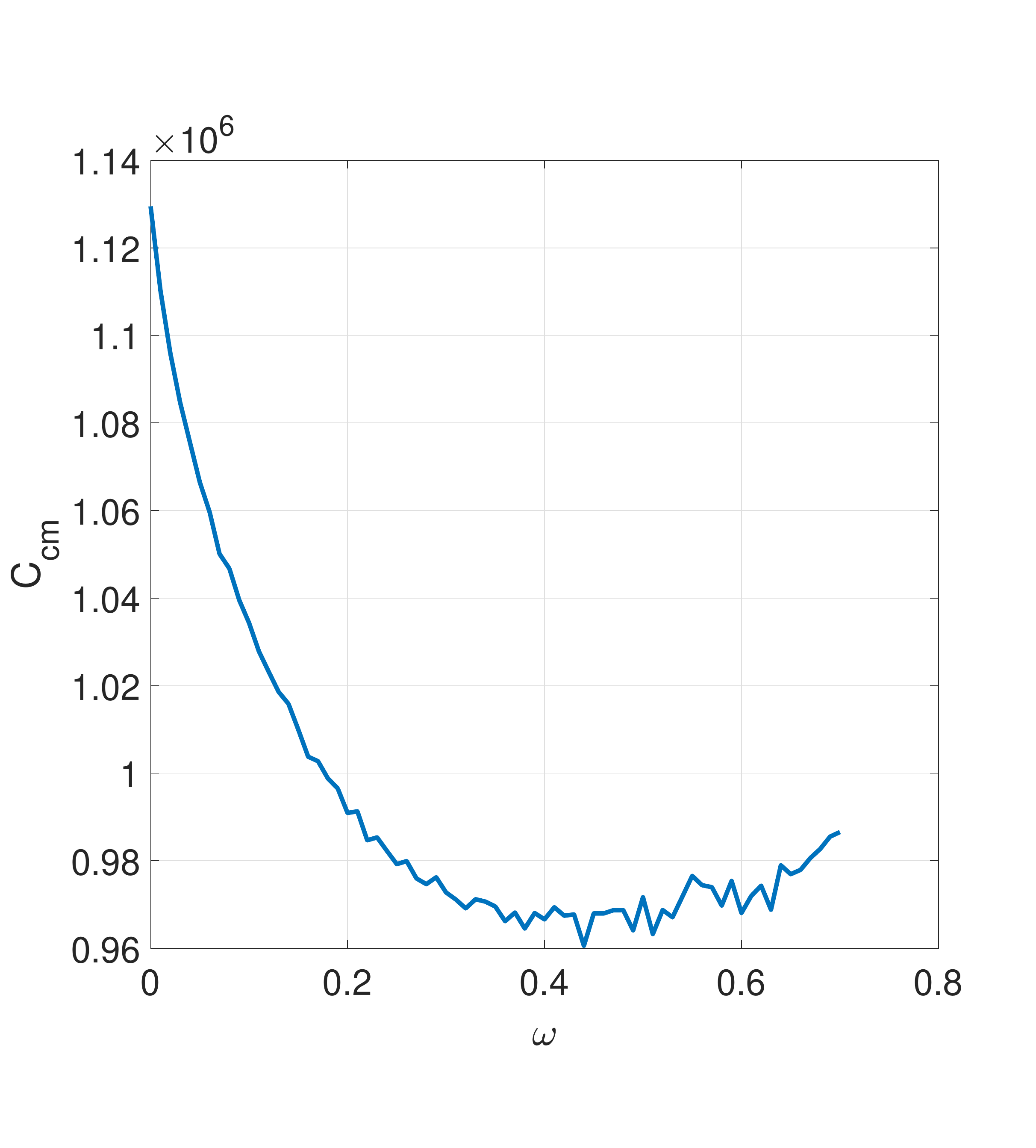}
	    \caption{Effect of $\QNSC$ with $\RTQ = 0.99 \times \RTUQ + 0.01$} \label{fig:commcost_omega}
     \end{subfigure}
     \hfill
     \caption{Network communication cost evaluation of ANQ-NIDS on smooth linear regression problem.}
\end{figure*}

\section{Conclusions} \label{Sec:conclusion}
In this paper, we propose a black-box model and unified convergence analysis for a general class of linearly convergent algorithms subject to quantized communications. The new black-box model encompasses 
composite optimization problems and distributed algorithms using historical information (e.g., EXTRA \cite{Shi2015EXTRA} and NEXT \cite{Lorenzo2016}), thus generalizing existing algorithmic frameworks, which are not applicable to these settings. 
To enable quantization  (below machine precision), we propose a novel \emph{biased compression (BC-)rule} that preserves linear convergence of distributed algorithms while using  a \textit{finite}  number of bits in each communication. As special instance of the BC-rule, we also proposed a  new random or deterministic quantizer, the ANQ,  {coupled with a communication-efficient encoding scheme}. We analyzed the communication cost of a gamut of distributed algorithms equipped with the  ANQ (in a unified fashion), showing favorable performance analytically and numerically, both in terms of convergence rate and communication cost, with respect to state-of-the-art quantization rules (including uniform and compression-based ones) and ad-hoc distributed algorithms.  

\bibliographystyle{IEEEtran}
\bibliography{IEEEabrv,ref_Q_opt}	

\appendices

\section{
 Linear Converge under quantization}
\label{pf:thm:conv}
\def\thesubsectiondis{\arabic{subsection}.} 
In this appendix, we prove   Theorem \ref{thm:conv}. We begin by introducing some   preliminary results, whose proofs are deferred to Appendix \ref{app:pf_conv_aux}. Throughout this section, we make the blanket assumption that the  conditions    in Theorem~\ref{thm:conv}  are satisfied. In particular,  $\RTQ\in (\RTUQ,1)$ and $\QNSC\in[0,\bar{\QNSC}(\RTQ))$, with 
$\bar{\QNSC}(\RTQ)$ defined in \eqref{omega_bar}. Due to the possibly random nature of the quantizer,  $\{{\bf z}^k, {\bf c}^{k,s}, \hat{\bf c}^{k,s}\}_{k \geq 0, s \in [R]}$ is a stochastic process defined on a proper probability space; we denote by $\mathcal F^{k,s}$   the $\sigma$-algebra generated by
$\{{\bf z}^{k'}, {\bf c}^{k',s'}, \hat{\bf c}^{k',s'}\}_{k'< k, s' \in [\q]} \cup \{{\bf z}^{k}, {\bf c}^{k,s'}, \hat{\bf c}^{k,s'-1}\}_{s' \leq s}$ ($\hat{\bf c}^{k,s}$ excluded).

\subsection{Preliminaries} The idea of the proof is to show by induction that both the optimization error $\Vert{\bf z}^{k} - {\bf z}^\infty\Vert$ and the input to the quantizer, $\Vert{\bf c}^{k, \ITERR} - \hat{\bf c}^{k-1, \ITERR}\Vert_2$, are   linearly convergent (in expectation)  at rate $\RTQ$, i.e.,
\begin{align}
\sqrt{\mathbb E[\big\Vert {\bf z}^{k} - {\bf z}^\infty \big\Vert^2]}  &\leq {\LIAP}_0 \cdot (\RTQ)^k,  \label{error_induction} \\
\sqrt{\mathbb E[\big\Vert \mathcal {\bf c}^{k, \ITERR} - \hat{\bf c}^{k-1, \ITERR}\big\Vert_2^2] }
&\leq F^{\ITERR}\cdot(\RTQ)^k, \quad \forall \ITERR \in [\q] \cup \{0\}, \label{qin_dist_induction}
\end{align}
where $F^0=0$, and $V_0$,  $F^{\ITERR}$, $s\in [R]$ 
satisfy
\begin{align}
&{\LIAP}_0 \geq 
\max\Big\{\Vert {\bf z}^{0} - {\bf z}^\infty\Vert,
\frac{
\sqrt{\m d}R\QNSB^0 +\QNSC\mathbf F^\top{\bf 1}
 }{\RTQ-\RTUQ}\tilde L_A\Big\}, \label{V_tilde_cond} \\
&F^{\ITERR}
  \geq
\max\Big\{L_{Z} c^* + L_C\sqrt{\m d}\QNSB^0+L_C(1+\QNSC) F^{\ITERR-1}, \label{Fm_cond}
\\&
\quad  \frac{\sqrt{\m d}\QNSB^0(1+L_C \RTQ )
  +L_C \RTQ (1+\QNSC) F^{\ITERR-1} + L_{Z}(1+\RTQ)  {\LIAP}_0
}{\RTQ -\QNSC}\Big\},  \nonumber
\end{align}
$\forall \ITERR \in [\q],$ and we have defined ${\bf F} \triangleq (F^s)_{s\in[\q]}$,
\begin{align}
c^* \triangleq \frac{1}{L_Z} \max_{s=[\q]}\Vert \mathcal C^\ITERR({\bf z}^0, {\bf 0})\Vert_2,\ \ 
\tilde L_A \triangleq L_A\sum_{s=0}^{\q-1}(L_C)^{s}.
 \label{eq:cstar}
\end{align}

The existence of such $V_0$ and $F^{\ITERR}$, $s\in [R]$,  is proved in the following lemma.
\begin{lemma} \label{lemma:V_F}
Let  $\QNSC\in[0,\bar{\QNSC}(\RTQ))$.
Then, \eqref{V_tilde_cond} and \eqref{Fm_cond} are satisfied by
\begin{align}\nonumber
&{\LIAP}_0=\max\Big\{c^*, \Vert {\bf z}^{0} - {\bf z}^\infty\Vert\Big\}\\&
\qquad+\frac{L_A \psi\sqrt{md}\q^2\QNSB^0 }{\RTQ-\RTUQ}\cdot
\frac{
1+\frac{\q\QNSC}{\RTQ}[(1+L_C\RTQ )\psi-1]
}{1-\QNSC/\bar\QNSC}
, \label{eq:V0}
\\
\nonumber&F^{\ITERR}
=
\frac{\sqrt{md}\QNSB^0 (1+L_C \RTQ)+2 L_{Z}{\LIAP}_0(\RTQ, \QNSC, \QNSB^0) }{\RTQ-\QNSC}\\&\qquad
\times
\sum_{\ITERRT=0}^{\ITERR-1}\Big(\frac{2  L_C}{1 -\QNSC/\RTQ}\Big)^{\ITERRT}
  ,\quad  \ITERR \in [\q], \label{Fm}
\end{align}
where $\psi\triangleq\max\{1,(2L_C)^{\q-1}\}$.
\end{lemma}

Since the effect of $\Vert{\bf c}^{k, \ITERR} - \hat{\bf c}^{k-1, \ITERR}\Vert_2$ on $\Vert{\bf z}^{k} - {\bf z}^\infty\Vert$ is through quantization, we need the following bound on the quantization error (its proof follows readily by the Cauchy–Schwarz inequality).
\begin{lemma} \label{lemma:BCR_net}
Let the quantizer $\mathcal Q_i$ used by agent $i\in [m]$  satisfy the BC-rule  \eqref{eq:BCR} with bias $\QNSB\geq 0$ and compression rate $\QNSC\in [0,\bar{\QNSC}(\sigma))$.   Then the following holds for the stack $\mathcal Q\triangleq [\mathcal Q_1, \ldots, Q_m]^\top$:
\begin{align}
\sqrt{\mathbb E\big[\big\Vert \mathcal Q({\bf x}) {-}  {\bf x} \big\Vert_2^2\big|\mathbf x\big]}
\leq
\sqrt{\m d}\,\QNSB  +\QNSC\Vert  {\bf x}\Vert_2,
 \label{eq:BCR_net}
\end{align}
$ \forall  {\bf x}=[{\bf x}_1^\top,\ldots, {\bf x}_m^\top]^\top \in \mathbb R^{md}.$
\end{lemma}
A direct application of Lemma \ref{lemma:BCR_net} leads to the following bound on the quantizer's input, which we use recurrently in the proofs: 
\begin{align}
\nonumber
&\sqrt{\mathbb E\big[\big\Vert
\hat{\bf c}^{k, \ITERR}-{\bf c}^{k, \ITERR}
\big\Vert_2^2 | \mathcal F^{k,s}\big]}
\\&\nonumber
\stackrel{\eqref{Qupdated}}{=} 
\sqrt{\mathbb E\big[\big\Vert
\mathcal Q^{k}({\bf c}^{k, \ITERR} - \hat{\bf c}^{k-1, \ITERR})
-({\bf c}^{k, \ITERR}-\hat{\bf c}^{k-1, \ITERR})
\big\Vert_2^2| \mathcal F^{k,s}\big]}
\\&
\stackrel{\eqref{eq:BCR_net}}{\leq}
\sqrt{\m d}\,\QNSB^0 \cdot (\RTQ)^k
+\QNSC\big\Vert{\bf c}^{k, \ITERR} - \hat{\bf c}^{k-1, \ITERR}\big\Vert_2,\quad a.s.,
\label{Qerror}
\end{align}
for all $\ITERR\in[\q]$ and $k=0,1,\ldots $, where we used $\QNSB^k=\QNSB^0 \cdot (\RTQ)^k$.

The following lemma bounds the distortion introduced by quantization in one iteration of (\ref{Qupdated}).
\begin{lemma} \label{lemma:agg_Lipt}
There holds: for all $k=0,1,\ldots, $
\begin{align*}
&\sqrt{\mathbb E[\Vert{\bf z}^{k+1} - {\bf z}^\infty\Vert^2]}
\leq
\lambda \sqrt{\mathbb E[\big\Vert{\bf z}^{k}-{\bf z}^\infty\big\Vert^2]}
\\&
\quad{+}\tilde L_A\sqrt{\m d}\q\QNSB^{0}\cdot (\RTQ)^k
{+}\tilde L_A\QNSC\sum_{\ITERR=1}^{\q }\sqrt{\mathbb E[
\Vert{\bf c}^{k, \ITERR} - \hat{\bf c}^{k-1, \ITERR}\Vert_2^2]},
\end{align*}
$a.s.,$ where $\tilde L_A$ is defined in \eqref{eq:cstar}.
\end{lemma}

We conclude this section of preliminaries with the following useful result.
\begin{lemma} \label{lemma:rvs}
Let $\{X_t:t\in[T]\}\subset\mathbb R$ be a collection of random variables. Then,
$$
\sqrt{\mathbb E\bigg[\Big(\sum_{t=1}^T X_t\Big)^2\bigg]}
\leq
\sum_{t=1}^T\sqrt{\mathbb E[X_t^2]}.
$$
\end{lemma}
\begin{proof}
It can be proved by developing the square within the expectation on the left hand side  expression, and by using
$\mathbb E[X_tX_u]\leq \sqrt{\mathbb E[X_t^2]}\sqrt{\mathbb E[X_u^2]}$.
\end{proof}

\subsection{Proof of Theorem \ref{thm:conv}}
 We prove     \eqref{error_induction} and \eqref{qin_dist_induction} by induction. 
Let ${\LIAP}_0$ and $F^{\ITERR}$,  $\ITERR\in [\q]$, satisfy \eqref{V_tilde_cond} and \eqref{Fm_cond}. 
 Since $\Vert {\bf z}^{0} - {\bf z}^\infty\Vert\leq {\LIAP}_0$ (see \eqref{V_tilde_cond}) and $\mathcal {\bf c}^{0,0}=\hat{\bf c}^{-1,0}=\mathbf 0$, \eqref{error_induction} holds for $k=0$ and  \eqref{qin_dist_induction} holds trivially for $k=0$ and $\ITERR = 0$. We  now use induction to prove that \eqref{qin_dist_induction} holds for $k=0$ and $\ITERR\in[\q]$. Assume that \eqref{qin_dist_induction} holds for $k=0$ and $\ITERR < \q$. Then, 
it follows that
\begin{align*}
&\big\Vert  {\bf c}^{0,\ITERR+1} - \hat{\bf c}^{-1,\ITERR+1}\big\Vert_2
= \big\Vert  {\bf c}^{0,\ITERR+1}\big\Vert_2
\\
&\stackrel{(a)}{\leq}
 \big\Vert \mathcal C^{\ITERR+1}\big( {\bf z}^{0}, {\bf 0} \big)\big\Vert_2
 + \big\Vert
 \mathcal C^{\ITERR+1}\big( {\bf z}^{0}, \hat{\bf c}^{0,\ITERR} \big)
 -\mathcal C^{\ITERR+1}\big({\bf z}^{0}, {\bf c}^{0,\ITERR}\big)
 \big\Vert_2 \\&\qquad
 +\big\Vert\mathcal C^{\ITERR+1}\big({\bf z}^{0}, {\bf c}^{0,\ITERR} \big)
 -\mathcal C^{\ITERR+1}\big( {\bf z}^{0}, {\bf 0}\big)\big\Vert_2 \\
&\stackrel{\eqref{eq:Lipt_C},\eqref{eq:cstar}}{\leq} 
 L_{Z} c^* 
 +L_C\big\Vert\hat{\bf c}^{0,\ITERR}-{\bf c}^{0,\ITERR}\big\Vert_2
 + L_C\big\Vert {\bf c}^{0,\ITERR} - \hat{\bf c}^{-1,\ITERR}\big\Vert_2,\ a.s.,
\end{align*}
where in (a) we used the triangle inequality and ${\bf c}^{0,\ITERR+1}=\mathcal C^{\ITERR+1}\big({\bf z}^{0}, \hat{\bf c}^{0,\ITERR}\big)$.
Taking the conditional expectation on both sides and using Lemma \ref{lemma:rvs} yield
\begin{align*}
&\sqrt{\mathbb E\big[\big\Vert  {\bf c}^{0,\ITERR+1} - \hat{\bf c}^{-1,\ITERR+1}\big\Vert_2^2 | \mathcal F^{0,s}\big]}
\\&
\leq L_{Z} c^* 
 {+}L_C\sqrt{\mathbb E\big[\big\Vert\hat{\bf c}^{0,\ITERR}-{\bf c}^{0,\ITERR}\big\Vert_2^2| \mathcal F^{0,s}\big]}
 {+} L_C\big\Vert {\bf c}^{0,\ITERR} - \hat{\bf c}^{-1,\ITERR}\big\Vert_2
 \\
 &\stackrel{\eqref{Qerror}}{\leq} L_{Z} c^* 
+L_C\sqrt{\m d} \QNSB^0
+L_C(1+\QNSC)
\big\Vert {\bf c}^{0,\ITERR} - \hat{\bf c}^{-1,\ITERR}\big\Vert_2,\quad a.s..
\end{align*}
Taking the unconditional expectation on both sides
and using again Lemma \ref{lemma:rvs} yield
\begin{align*}
&\sqrt{\mathbb E\big[\big\Vert  {\bf c}^{0,\ITERR+1} - \hat{\bf c}^{-1,\ITERR+1}\big\Vert_2^2\big]}
\\&\leq 
L_{Z} c^* 
+L_C\sqrt{\m d} \QNSB^0
+L_C(1+\QNSC)
\sqrt{\mathbb E[\big\Vert {\bf c}^{0,\ITERR} - \hat{\bf c}^{-1,\ITERR}\big\Vert_2^2]}
\\
&\stackrel{\eqref{qin_dist_induction}}{\leq} L_{Z} c^* + L_C\sqrt{\m d} \QNSB^0
+L_C(1+\QNSC) F^{\ITERR}
\stackrel{\eqref{Fm_cond}}{\leq} F^{\ITERR+1},
\end{align*}
which completes the induction proof of \eqref{qin_dist_induction} for $k=0$ and $\ITERR\in [\q]$.

 Now, let us assume that \eqref{error_induction} and \eqref{qin_dist_induction} hold for a generic $k=0.1,\ldots$;   we prove  that they hold at $k+1$.
We begin with \eqref{error_induction}. Invoking Lemmas \ref{lemma:agg_Lipt}, \ref{lemma:rvs}, and using the induction hypotheses
\eqref{error_induction} and \eqref{qin_dist_induction} at $k$, yield
\begin{align*}
&\sqrt{\mathbb E[\Vert {\bf z}^{k+1} - {\bf z}^\infty\Vert^2]}   
\leq \RTUQ {\LIAP}_0\cdot(\RTQ)^k
\\&
 + \tilde L_A\big(\sqrt{md} R\QNSB^0 +\QNSC \mathbf  F^\top{\bf 1}\big)  \cdot (\RTQ)^k
 \leq {\LIAP}_0 \cdot (\RTQ)^{k+1},
\end{align*}
where the last inequality  follows from the definition of $V_0$ in \eqref{V_tilde_cond},
which concludes   the induction argument for \eqref{error_induction}.

We now prove that \eqref{qin_dist_induction} holds for $k+1$, by induction over $\ITERR\in [R]$. First, note that  \eqref{qin_dist_induction} holds trivially for $k+1$ and $\ITERR=0$, since 
$\mathcal {\bf c}^{k+1,0}=\hat{\bf c}^{k,0}=\mathbf 0$.
Now, assume that \eqref{qin_dist_induction} holds   at iteration $k+1$  for $\ITERR < \q$.
Then,
\begin{align*}
&\big\Vert {\bf c}^{k+1,\ITERR+1} - \hat{\bf c}^{k, \ITERR+1} \big\Vert_2
\\&
\stackrel{\eqref{Qupdated}}{=}\big\Vert 
\mathcal C^{\ITERR+1}\big(\mathbf z^{k+1}, \hat{\bf c}^{k+1, \ITERR}\big) 
-\mathcal C^{\ITERR+1}\big(\mathbf z^{k+1}, {\bf c}^{k+1, \ITERR}\big) \\&
\quad{+}\mathcal C^{\ITERR+1}\big(\mathbf z^{k+1}, {\bf c}^{k+1, \ITERR}\big) 
{-}\mathcal C^{\ITERR+1}\big(\mathbf z^{k}, \hat{\bf c}^{k, \ITERR}\big){+}
{\bf c}^{k, \ITERR+1}{-}\hat{\bf c}^{k, \ITERR+1}
\big\Vert_2
 \nonumber \\
&\stackrel{(a)}{\leq} \big\Vert \mathcal C^{\ITERR+1}\big(\mathbf z^{k+1},\hat{\bf c}^{k+1, \ITERR}\big) - \mathcal C^{\ITERR+1}\big(\mathbf z^{k+1},{\bf c}^{k+1,\ITERR}\big) \big\Vert_2 \\&
 \quad+ \big\Vert \mathcal C^{\ITERR+1}\big(\mathbf z^{k+1},{\bf c}^{k+1,\ITERR}\big) - \mathcal C^{\ITERR+1}\big(\mathbf z^{k},\hat{\bf c}^{k, \ITERR}\big) \big\Vert_2
 \\&
 \quad+ \big\Vert
{\bf c}^{k, \ITERR+1}-\hat{\bf c}^{k, \ITERR+1}
\big\Vert_2
\stackrel{\eqref{eq:Lipt_C}, \eqref{eq:Lipt_Z}}{\leq} L_C \big\Vert \hat{\bf c}^{k+1, \ITERR} - {\bf c}^{k+1,\ITERR} \big\Vert_2
\\&
\quad + L_C \big\Vert {\bf c}^{k+1,\ITERR} - \hat{\bf c}^{k, \ITERR} \big\Vert_2
 + L_{Z} \big\Vert {\bf z}^{k+1} - {\bf z}^{\infty} \big\Vert_2
 \\&
 \quad+ L_{Z} \big\Vert {\bf z}^{k} - {\bf z}^{\infty} \big\Vert_2
 + \big\Vert
\hat{\bf c}^{k, \ITERR+1}-{\bf c}^{k, \ITERR+1}
\big\Vert_2,\ a.s..
\end{align*}

Then, taking the  expectation conditional on $\mathcal F^{k+1, \ITERR}$ and invoking Lemma \ref{lemma:rvs} and \eqref{Qerror} to bound $\sqrt{\mathbb E\big[\big\Vert \hat{\bf c}^{k+1, \ITERR}-{\bf c}^{k+1, \ITERR}
\big\Vert_2^2 | \mathcal F^{k+1,s}\big]}$, yield
\begin{align*}
&\sqrt{\mathbb E[\big\Vert {\bf c}^{k+1,\ITERR+1} - \hat{\bf c}^{k, \ITERR+1} \big\Vert_2^2|\mathcal F^{k+1, \ITERR}]}
\leq L_C \sqrt{\m d}\QNSB^0\cdot(\RTQ)^{k+1}\\&\quad + L_C (1+\QNSC) \big\Vert {\bf c}^{k+1,\ITERR} - \hat{\bf c}^{k, \ITERR} \big\Vert_2
 + L_{Z} \big\Vert {\bf z}^{k+1} - {\bf z}^{\infty} \big\Vert_2
  \\&
\quad + L_{Z} \big\Vert {\bf z}^{k} - {\bf z}^{\infty} \big\Vert_2
 + \big\Vert
\hat{\bf c}^{k, \ITERR+1}-{\bf c}^{k, \ITERR+1}
\big\Vert_2,\quad a.s..
\end{align*}
Now, taking the expectation conditional on $\mathcal F^{k, \ITERR+1} \subseteq\mathcal F^{k+1, \ITERR}$, invoking Lemma \ref{lemma:rvs}, and \eqref{Qerror} to bound $\sqrt{\mathbb E\big[\big\Vert
\hat{\bf c}^{k, \ITERR+1}-{\bf c}^{k, \ITERR+1}
\big\Vert_2^2 | \mathcal F^{k,s+1}\big]}$, yield
\begin{align*}
&\sqrt{\mathbb E[\big\Vert {\bf c}^{k+1,\ITERR+1} - \hat{\bf c}^{k, \ITERR+1} \big\Vert_2^2|\mathcal F^{k,\ITERR+1}]}
{\leq}\sqrt{\m d}\QNSB^0 (1+L_C \RTQ){\cdot}(\RTQ)^k\\&
\quad+ L_C (1+\QNSC) \sqrt{\mathbb E[\Vert {\bf c}^{k+1,\ITERR} - \hat{\bf c}^{k,\ITERR} \Vert_2 ^2|\mathcal F^{k,\ITERR+1}]}
 \\
&\quad+ L_{Z} \sqrt{\mathbb E[\big\Vert {\bf z}^{k+1} - {\bf z}^{\infty} \big\Vert_2^2|\mathcal F^{k,\ITERR+1}]}
\\&
\quad+L_{Z} \sqrt{\mathbb E[\big\Vert {\bf z}^{k}{-}{\bf z}^{\infty} \big\Vert_2^2|\mathcal F^{k,\ITERR+1}]}
{+}\QNSC\big\Vert {\bf c}^{k, \ITERR+1}{-}\hat{\bf c}^{k-1, \ITERR+1} \big\Vert_2,
\end{align*}
$a.s.$.
Taking the unconditional expectation and invoking Lemma \ref{lemma:rvs} again yield
\begin{align*}
&\sqrt{\mathbb E[\big\Vert {\bf c}^{k+1,\ITERR+1} - \hat{\bf c}^{k, \ITERR+1} \big\Vert_2^2]}\leq \sqrt{\m d}\QNSB^0 (1+L_C \RTQ) \cdot (\RTQ)^k
\\&
\quad+L_C(1{+}\QNSC)\sqrt{\mathbb E[\big\Vert{\bf c}^{k+1, \ITERR}{-}\hat{\bf c}^{k, \ITERR}\big\Vert_2^2]}
{+}L_{Z} \sqrt{\mathbb E[\big\Vert {\bf z}^{k+1}{-}{\bf z}^{\infty} \big\Vert_2^2]}
\\&
\quad+ L_{Z} \sqrt{\mathbb E[\big\Vert {\bf z}^{k} - {\bf z}^{\infty} \big\Vert_2^2]}
+\QNSC\sqrt{\mathbb E[\big\Vert{\bf c}^{k, \ITERR+1} - \hat{\bf c}^{k-1,\ITERR+1}\big\Vert_2^2]}
\nonumber \\
&\stackrel{(a)}{\leq}
\sqrt{md}\QNSB^0(1+L_C \RTQ)\cdot(\RTQ)^k
+L_C(1+\QNSC)F^{\ITERR} \cdot (\RTQ)^{k+1}\\&
\quad+\QNSC F^{\ITERR+1} \cdot (\RTQ)^k
+ L_{Z} (1+\RTQ)  {\LIAP}_0 \cdot (\RTQ)^k
  \stackrel{\eqref{Fm_cond}}{\leq}
  F^{\ITERR+1}\cdot (\RTQ)^{k+1}, 
\end{align*}
where in $(a)$ we used  the induction hypotheses
\eqref{qin_dist_induction} (applied to the second and last terms) and \eqref{error_induction} (applied to the third and fourth terms).
This proves the induction for \eqref{qin_dist_induction}, and the theorem.

\subsection{Proof of auxiliary lemmas for Theorem \ref{thm:conv}} \label{app:pf_conv_aux}
\subsubsection{Proof of Lemma \ref{lemma:V_F}}
It is not difficult to check that  conditions \eqref{V_tilde_cond} and \eqref{Fm_cond} can be satisfied by choosing
\begin{align*}
&{\LIAP}_0 \geq 
\max\Big\{c^*,\Vert {\bf z}^{0} - {\bf z}^\infty\Vert,
\frac{
\sqrt{\m d}R\QNSB^0 +\QNSC\mathbf F^\top{\bf 1}
 }{\RTQ-\RTUQ}\tilde L_A\Big\}, \\
&F^{\ITERR}
  \geq
  \frac{\sqrt{\m d}\QNSB^0(1{+}L_C \RTQ)
  {+}L_C\RTQ (1{+}\QNSC) F^{\ITERR-1}{+}L_{Z}(1{+}\RTQ)  {\LIAP}_0
}{\RTQ -\QNSC},
\end{align*}
$\forall \ITERR \in [\q],$
where $c^*,\tilde L_A$ are defined in \eqref{eq:cstar}. 
Moreover, since $\QNSC<\RTQ < 1$, it is sufficient to choose
\begin{align*}
&F^{\ITERR}
  =
  \frac{1}{\RTQ}
  \frac{\sqrt{md}\QNSB^0 (1{+}L_C \RTQ)
   +2L_C\RTQ F^{\ITERR-1}{+}2  L_{Z} {\LIAP}_0
}{1 -\QNSC/\RTQ},\ \forall \ITERR \in [\q]. 
\end{align*}
Solving this expression recursively yields \eqref{Fm}.

We now prove \eqref{eq:V0}. We begin noting  that $F^{\ITERR}$ is a non-decreasing function of $\ITERR$, hence
$F^{\ITERR} \leq F^\q$. Moreover, $F^\q$ is an affine function of ${\LIAP}_0$. Using
the facts that
$(\frac{2L_C}{1 -\QNSC/\RTQ})^\ITERR
\leq \psi (1 -\QNSC/\RTQ)^{-\ITERR}$, where $\psi\triangleq \max\{1,(2L_C)^{\q-1}\}$,
and
$$
(1 -\QNSC/\RTQ)^{-\ITERR}
\leq(1 -\QNSC/\RTQ)
(1 -\q\QNSC/\RTQ)^{-1},\ \forall s\in [R-1]\cup \{0\},$$ and $\QNSC<\bar\QNSC(\RTQ)<\RTQ/\q$, we can  upper bound $F^{\ITERR}$  as
\begin{align}
&
F^{\ITERR} \leq F^\q
\leq 
 a_1
+{\LIAP}_0 a_2\triangleq\bar F^\q, \label{eq:FR_bar}
\end{align}
where 
\begin{align}
a_1 &\triangleq
\psi\sqrt{md}\QNSB^0(1+L_C \RTQ)
 \cdot \frac{\q/\RTQ}{1 -\q\QNSC/\RTQ}
, \label{a1_ori}\\
a_2 &\triangleq
2 L_{Z}\psi \cdot \frac{\q/\RTQ}{1 -\q\QNSC/\RTQ}.
\label{a2_ori}
\end{align}
Furthermore, since
$\mathbf F^\top{\bf 1}\leq \q F^{\q}\leq \q(a_1+V_0a_2)$ and $\tilde{L}_A=L_A\sum_{s=0}^{\q-1}(L_C)^{s} \leq L_A \psi\q $, to satisfy \eqref{V_tilde_cond}, it is sufficient to choose ${\LIAP}_0$ as
\begin{align*}
&{\LIAP}_0{\geq}
\max\Big\{c^*, \Vert {\bf z}^{0} - {\bf z}^\infty\Vert
,
 L_A \psi \q^2\frac{
\sqrt{md}\QNSB^0  +\QNSC(a_1{+}V_0a_2)
 }{\RTQ-\RTUQ}\Big\}. 
 \end{align*}
Using  $x\geq \max\{c,a+b x\} \Leftrightarrow x\geq \max\{c,a/(1-b)\}$, under   $b<1$, the above condition is equivalent to
 \begin{align}
&{\LIAP}_0\geq\max\Big\{c^*, \Vert {\bf z}^{0} - {\bf z}^\infty\Vert,\frac{L_A \psi \q^2(\sqrt{md}\QNSB^0{+}\QNSC a_1)}{\RTQ-\RTUQ- L_A \psi\q^2\QNSC a_2 } \Big\}, \label{eq:V_tilde_ab}\end{align}
as long as 
$ L_A \psi \q^2\QNSC a_2/(\RTQ-\RTUQ)<1$. Solving with respect to $\omega$ (note that $a_2$ is a function of $\omega$), this condition is equivalent to $\QNSC\in[0,\bar{\QNSC}(\RTQ))$ with $\bar{\QNSC}(\RTQ)$ given by \eqref{omega_bar}, hence it holds by assumption.
Substituting   the values of $a_1,a_2$
in \eqref{eq:V_tilde_ab} and  using  $\max\{a,b\}\leq a+b$ (for $a,b\geq 0$), yields \eqref{eq:V0}.\hfill $\square$.

\subsubsection{Proof of Lemma \ref{lemma:agg_Lipt}}
At iteration $k$, let $\zeta_\ITERR $ be defined as
\begin{align*}
\zeta^\ITERR=
\mathcal A\left(\mathbf z^{k},\hat{\bf c}^{k,1},\dots,\hat{\bf c}^{k, \ITERR},\tilde{\bf c}_{\ITERR}^{\ITERR+1},\dots,\tilde{\bf c}_{\ITERR}^{\q}\right), 
\end{align*}
where
\begin{align*}
\tilde{\bf c}_{\ITERR}^{\ITERR} = \hat{\bf c}^{k, \ITERR} \quad \text{and} \quad
\tilde{\bf c}_{\ITERR}^{\ell+1} \triangleq \mathcal C^{\ell+1}\left({\bf z}^{k},\tilde{\bf c}_{\ITERR}^{\ell} \right), \forall \ell\geq \ITERR. 
\end{align*}
In other words, $\tilde{\mathbf c}_{s}^{\ell}, \zeta^\ITERR $ are the communication signals at round $\ell$ and the updated computation state, respectively, obtained by applying the unquantized communication mapping after round $s$ and the quantized one before round $s$. Clearly, ${\bf z}^{k+1}=\mathcal A ({\bf z}^{k}, \hat{\bf c}^{k,1},\cdots, \hat{\bf c}^{k,\q} )=\zeta^R$ and $\tilde{\mathcal A}({\bf z}^{k})=\zeta^0$ (unquantized update of the computation state).
It then follows that
\begin{align*}
\mathbf z^{k+1}=\zeta^\q
=
\tilde{\mathcal A}({\bf z}^{k})
+\sum_{\ITERR=1}^{\q}\left(\zeta^{\ITERR}
-\zeta^{\ITERR-1}\right),\quad a.s.. 
\end{align*}
Invoking the triangle inequality yields
\begin{align}
\big\Vert{\bf z}^{k+1}{-}{\bf z}^\infty \big\Vert  &\leq 
\big\Vert\tilde{\mathcal A}({\bf z}^{k}){-}{\bf z}^\infty\big\Vert
{+}\sum_{\ITERR=1}^{\q }\Vert\zeta^{\ITERR}
{-}\zeta^{\ITERR-1}\Vert,\ a.s.. \label{eq:agg_Lipt_1}
\end{align}
We now study the second term. 
From the Lipschitz continuity of $\mathcal A$ (Assumption \ref{assump:Lipt_A}) and the definition of $\zeta^{\ITERR}$, it holds that
$$
\Vert\zeta^{\ITERR}
-\zeta^{\ITERR-1}\Vert
\leq
 L_A\sum_{\ell=\ITERR}^{\q}
\Vert\tilde{\bf c}_{\ITERR}^{\ell}
-\tilde{\bf c}_{\ITERR-1}^{\ell}\Vert_2,\quad a.s..
$$
Furthermore,
$$\Vert\tilde{\bf c}_{\ITERR}^{\ITERR}
-\tilde{\bf c}_{\ITERR-1}^{\ITERR}\Vert_2
=
\Vert
\hat{\bf c}^{k, \ITERR}
-\mathcal C^{\ITERR}({\bf z}^{k},\hat{\bf c}^{k, \ITERR-1})
\Vert_2
=
\Vert
\hat{\bf c}^{k, \ITERR}
-{\bf c}^{k, \ITERR}
\Vert_2$$ 
$a.s.,$
 and, for $\ell>\ITERR$,
\begin{align*}&
\Vert\tilde{\bf c}_{\ITERR}^{\ell}
-\tilde{\bf c}_{\ITERR-1}^{\ell}\Vert_2
=
\Vert
\mathcal C^{\ell}({\bf z}^{k}, \tilde{\bf c}_{\ITERR}^{\ell-1})
-\mathcal C^{\ell}({\bf z}^{k}, \tilde{\bf c}_{\ITERR-1}^{\ell-1})
\Vert_2\\&
\leq
L_C\Vert
\tilde{\bf c}_{\ITERR}^{\ell-1}
-\tilde{\bf c}_{\ITERR-1}^{\ell-1}
\Vert_2
\leq\dots\leq
(L_C)^{\ell-\ITERR}\Vert
\hat{\bf c}^{k, \ITERR}
-{\bf c}^{k, \ITERR}
\Vert_2,\ a.s.,
\end{align*}
where the last step follows from induction over $\ell$. Replacing these bounds in \eqref{eq:agg_Lipt_1}, we finally obtain
\begin{align*}
&\big\Vert{\bf z}^{k+1} - {\bf z}^\infty\big\Vert 
\\&\leq 
\big\Vert\tilde{\mathcal A}({\bf z}^{k})-{\bf z}^\infty\big\Vert  
+ L_A
 \sum_{\ITERR=1}^{\q }\sum_{\ell=0}^{\q-\ITERR}
(L_C)^{\ell}
\Vert
\hat{\bf c}^{k, \ITERR}
-{\bf c}^{k, \ITERR}
\Vert_2
\\&
\leq
\lambda\big\Vert{\bf z}^{k}-{\bf z}^\infty\big\Vert 
+\tilde L_A
 \sum_{\ITERR=1}^{\q }
\Vert
\hat{\bf c}^{k, \ITERR}
-{\bf c}^{k, \ITERR}
\Vert_2,\quad a.s.,
\end{align*}
where $\tilde L_A$ is defined in \eqref{eq:cstar} and we used Assumption \ref{assump:R_conv_z}.
Taking the expectation conditional on the filtration $\mathcal F^{k,s}$ while applying Lemma \ref{lemma:rvs} and \eqref{Qerror}, 
starting from $s=R,R-1,\dots, 1$,
it follows that
\begin{align*}
&\sqrt{\mathbb E\big[\big\Vert{\bf z}^{k+1} - {\bf z}^\infty\big\Vert^2 | \mathcal F^{k,1}\big]}{\leq}\RTUQ \big\Vert {\bf z}^{k}-{\bf z}^\infty\big\Vert
{+}\tilde L_A\sqrt{\m d}\q\QNSB^{0} \cdot (\RTQ)^k
 \\
&
\quad +\tilde L_A\QNSC\sum_{\ITERR=1}^{\q }
\sqrt{\mathbb E\big[\Vert{\bf c}^{k, \ITERR} - \hat{\bf c}^{k-1, \ITERR}\Vert_2^2| \mathcal F^{k, 1} \big]} ,\quad a.s..
\end{align*}
Finally, taking unconditional expectation and using Lemma \ref{lemma:rvs} concludes the proof.

\section{Deterministic and Random quantizer's design}  \label{pf:lemmata:quant}
\subsection{Proof of Lemma \ref{lemma:quant}}
\label{pf:lemma:quant}
Let $\mathcal Q(\bullet): [-\RANGE,\RANGE]^d\to
\mathbb Q^d$ be a component-wise quantizer, with the $n$th component quantizer
$\mathcal Q_n(\bullet)$ mapping points in the interval
$[-\RANGE,\RANGE]$
to discrete points in the set $\mathbb Q$. We assume that the same quantizer is applied across all $n$, since each component is optimized with the same range and number of quantization points. The goal is to
 define a quantizer 
$\mathcal Q$ which satisfies the BC-rule
within $\mathbf x\in[-\delta,\delta]^d$
with maximal range $\RANGE$.
To this end, a necessary and sufficient condition is
\begin{align}
\label{cond_prob}
|\mathcal Q_n(x) - x| \leq
\QNSB+\QNSC|x|,\ \forall x\in[-\delta,\delta],\ \forall n \in [d].
\end{align}
The sufficiency can be proved using Cauchy–Schwarz inequality. To prove the necessity,
assume that \eqref{cond_prob} is violated for some $x\in[-\RANGE, \RANGE]$, i.e.,
$|\mathcal Q_{n}(x) - x| >\QNSB+\QNSC|x|$, and let ${\bf x} = x\mathbf 1$. It follows that
\begin{align*}
\Vert \mathcal Q({\bf x}) - {\bf x}\Vert_2 
&= \sqrt{d}|\mathcal Q_{n}(x) - x| 
> \sqrt{d}\QNSB+\QNSC\sqrt{d}|x|
\\&
= \sqrt{d}\QNSB+\QNSC\Vert {\bf x}\Vert_2, 
\end{align*}
implying the BC-rule is not satisfied at ${\bf x}$.

Hence, we now focus on the design of a component-wise quantizer $\mathcal Q_n$  satisfying \eqref{cond_prob} with maximal range $\delta$.
In the following, we omit the dependence on $n$ for convenience.

Assume that $N=|\mathbb Q|$ is odd (the case $N$ even can be studied in a similar fashion, and is provided at the end of this proof for completeness),
and let
$\mathbb Q\triangleq
\cup_{\ell=0}^{(N-1)/2}\{\tilde q_{\ell},-\tilde q_{\ell}\}$ be the set of quantization points, with $0=\tilde q_0<\tilde q_1<\dots, \tilde q_{\ell}<\tilde q_{\ell+1}<\dots$.
 Note that we restrict to a symmetric quantizer since
the error metric is symmetric around 0 (the detailed proof on the optimality of symmetric quantizers is omitted due to space constraints). 
 We then aim to solve
\begin{align}
\begin{array}{cc}
\underset{\delta\geq 0,\mathcal Q}{\max} & \delta \\
\text{s.t. } & 
|\mathcal Q(x)- x|
\leq \QNSB + \QNSC x
,\ \forall
x\in [0,\delta],
\end{array} \label{eq:lemma_quant}
\end{align}
where  the constraint
\eqref{cond_prob} is imposed only to $x\in[0,\delta]$ since
the quantizer is symmetric around 0.
Since
the quantization error in \eqref{eq:lemma_quant}
 is measured in Euclidean distance, it is optimal to restrict the quantization points to $\mathbb Q\subset[-\delta,\delta]$ and to map the input
to the nearest quantization point (ties may be resolved arbitrarily).
Then, letting $\mathcal X_{\ell}=((\tilde q_{\ell-1}+\tilde q_{\ell})/2,(\tilde q_{\ell}+\tilde q_{\ell+1})/2]$, with $\tilde q_{-1}=0$ and
$\tilde q_{(N+1)/2}=2\delta-\tilde q_{(N-1)/2}$, it follows that
$[0,\delta]\equiv\cup_{\ell =0}^{(N-1)/2}\mathcal X_\ell$ and
$\mathcal Q(x)=\tilde q_{\ell},\forall x\in\mathcal X_\ell$.
Therefore, the optimization problem \eqref{eq:lemma_quant} can be expressed equivalently as
\begin{align*}
\begin{array}{cc}
\underset{\delta\geq 0,\tilde{\mathbf q}}{\max} & \delta \\
\text{s.t.} &
(\tilde q_{\ell}- x)^2
{\leq} (\QNSB + \QNSC x)^2,\ \forall x \in \mathcal X_\ell,  \forall \ell{=}0,1,\dots,\frac{N{-}1}{2}, \\
 & 0= \tilde q_0 \leq \dots \leq \tilde q_{(N+1)/2}=2\delta-\tilde q_{(N-1)/2}.
\end{array}
\end{align*}
Equivalently,
\begin{align*}
\begin{array}{cc}
\underset{\delta\geq 0,\tilde{\mathbf q}}{\max}&\delta \\
\text{s.t.}&
\underset{x \in \mathcal X_\ell}\max\,
(\tilde q_{\ell}- x)^2
-(\QNSB + \QNSC x)^2
\leq 0, \forall \ell{=}0,1,\dots,\frac{N{-}1}{2}, \\
& 0= \tilde q_0 \leq \dots \leq \tilde q_{(N+1)/2}=2\delta-\tilde q_{(N-1)/2},
\end{array}
\end{align*}
and solving the maximization with respect to $x\in\mathcal X_\ell$ (note that the quadratic function is convex in $x$, hence it is maximized at the boundaries of $\mathcal X_\ell$), we obtain
\begin{align*}
\begin{array}{cc}
\underset{\delta\geq 0,\tilde{\mathbf q}}{\max} & \frac{\tilde q_{(N-1)/2}+\tilde q_{(N+1)/2}}{2} \\
\text{s.t.} &
\tilde q_\ell\leq
\tilde q_{\ell-1} \Big(\frac{1+\QNSC}{1-\QNSC}\Big) + \frac{2\QNSB}{1-\QNSC}
,\ \forall \ell\in[(N+1)/2], \\
 & 0 = \tilde q_0 < \tilde q_1<\dots<\tilde q_{(N+1)/2}.
\end{array}
\end{align*}
Solving this problem with respect to $\tilde{\mathbf q}$ yields $q_0 = 0$ and
\begin{align*}
q_{\ell} = q_{\ell-1} \Big(\frac{1+\QNSC}{1-\QNSC}\Big) + \frac{2\QNSB}{1-\QNSC},\quad \forall\ell\geq 1.
\end{align*}
Solving by induction, we obtain
$q_\ell$ as in \eqref{qell}, $\delta(\QNSB,\QNSC,N)$ as in \eqref{eq:delta_det}, and
\begin{align*}
    \ell(x) = \mathrm{sign}(x) \cdot \min \Big\{\ell\geq 0: \frac{q_{\ell} + q_{\ell+1}}{2} \geq |x|\Big\},
\end{align*}
yielding \eqref{eq:lstar} after solving with the expression of $q_\ell$.
A similar technique can be proved for the case when $N$ is even, yielding quantization points
\begin{align*}
&q_{\ell}=-q_{-\ell}=
\dfrac{\QNSB}{\QNSC}\Big[
\frac{(1+\QNSC)^{\ell}}{(1-\QNSC)^{\ell-1}}-1\Big],
\ \quad \forall \ell\geq 1
\end{align*}
and $\delta(\QNSB,\QNSC,N){=}\frac{q_{N/2}{+}q_{N/2+1}}{2}$, which concludes the proof.

\subsection{Proof of Lemma \ref{lemma:quant_prob}} \label{pf:lemma_quant_prob}
Using a similar technique as in Appendix \ref{pf:lemma:quant} when $N$ is odd,
using the fact that $[0,\delta]=\cup_{\ell\in[(N-1)/2]}[\tilde q_{\ell-1},\tilde q_{\ell}]$ and $\delta=\tilde q_{N/2}$
it suffices to solve
\begin{align*}
\begin{array}{cc}
\underset{\delta\geq 0,\tilde{\mathbf q}}{\max} & q_{(N-1)/2} \\
\text{s.t.}& \!\!\mathbb E[|{\mathcal Q}(x) {-} x|^2]
{\leq}(\QNSB{+}\QNSC x)^2,\forall x \in [\tilde q_{\ell-1}, \tilde q_{\ell}], \forall \ell{\in}\left[\frac{N-1}{2}\right]\!{,} \\
&\mathbb E[\mathcal Q(x)] = x,\ \ \  
 0= \tilde q_0 \leq \dots \leq \tilde q_{(N-1)/2}= \delta.
\end{array}
\end{align*}
Furthermore, 
since $x \in [\tilde q_{\ell-1}, \tilde q_{\ell}]$
is mapped to $\tilde q_{\ell-1}$ w.p. $(\tilde q_{\ell}-x)/(\tilde q_{\ell}-\tilde q_{\ell-1})$ and to $\tilde q_{\ell}$ w.p. $(x-\tilde q_{\ell-1})/(\tilde q_{\ell}-\tilde q_{\ell-1})$
to satisfy $\mathbb E[\mathcal Q(x)] = x$, the problem can be expressed equivalently as
\begin{align*}
\begin{array}{cc}
\underset{\delta\geq 0,\tilde{\mathbf q}}{\max} &q_{(N-1)/2} \\
\text{s.t.}
&(x{-}\tilde q_\ell)(x{-}\tilde q_{\ell-1}){+}(\QNSB{+}\QNSC x)^2{\geq}0,\forall x{\in}[\tilde q_{\ell-1},  \tilde q_{\ell}],\\&
\forall \ell{\in}\left[\frac{N{-}1}{2}\right],\ \ \ 
 \qquad 0= \tilde q_0 \leq \dots \leq \tilde q_{(N-1)/2}=\delta,
\end{array}
\end{align*}
or equivalently
\begin{align*}
\begin{array}{cc}
\underset{\delta\geq 0,\tilde{\mathbf q}}{\max} &q_{(N-1)/2} \\
\text{s.t.} &
\!\!\!\!\!\underset{x{\in}[\tilde q_{\ell-1}, \tilde q_{\ell}]}{\min}
(x{-}\tilde q_\ell)(x{-}\tilde q_{\ell-1}){+}(\QNSB{+}\QNSC x)^2{\geq}0
,\forall \ell{\in}\left[\frac{N{-}1}{2}\right], \\&
 \qquad 0= \tilde q_0 \leq \dots \leq \tilde q_{(N-1)/2}=\delta.
\end{array}
\end{align*}
Solving the minimization over $x\in [\tilde q_{\ell-1}, \tilde q_{\ell}]$ and solving with respect to $\tilde{\mathbf q}$ yields the following optimal quantization points:
 $q_0=0$ and
$$
q_{\ell}
=
q_{\ell-1}(\sqrt{1+(\omega)^2}+\omega)^2+2\eta(\sqrt{1+(\omega)^2}+\omega),\quad \forall \ell\geq 1.
$$
Solving by induction, we obtain
$q_\ell$ as in \eqref{qellprob}, $\delta(\QNSB,\QNSC,N)$ as in
\eqref{eq:delta_prob}, and the probabilistic quantization rule as in \eqref{eq:lstar_prob}, with $\ell$ given by
\begin{align*}
    \ell= \mathrm{sign}(x) \cdot \min \Big\{\ell\geq 0: q_{\ell} \geq |x|\Big\},
\end{align*}
yielding \eqref{eq:lstar_prob} after solving with the expression of $q_\ell$.

A similar technique can be proved for the case when $N$ is even, yielding the quantization points
$$
q_{\ell}
=
\frac{\eta}{\omega}
\left[
\frac{(\sqrt{1+(\omega)^2}+\omega)^{2\ell-1}}{\sqrt{1+(\omega)^2}}
-1
\right]
,\quad \forall \ell\geq 1.
$$
and $\delta(\QNSB,\QNSC,N)=q_{N/2}$, which concludes the proof.

\subsection{Proof of Corollary \ref{coro:converse}} \label{app:pf_converse}
Let $\mathcal Q({\bf x})$ be a generic deterministic or probabilistic quantizer with domain $[-\RANGE, \RANGE]^d$ and codomain $\mathbb Q \in \mathbb R^d$ with $|\mathbb Q| < \infty$, that satisfies the BC-rule with $\QNSB = 0$. It follows that
\begin{align}
&\nonumber
\omega\Vert\mathbf x\Vert_2\geq\sqrt{\mathbb E[\Vert\mathcal Q(\mathbf x)-\mathbf x\Vert_2^2]}
\geq
\min_{\mathbf q\in\mathbb Q}\Vert\mathbf q-\mathbf x\Vert_2
\\&
=\Vert\mathcal Q_{\mathrm{det}}(\mathbf x)-\mathbf x\Vert_2
, \forall {\bf x} \in [-\RANGE, \RANGE]^d,
\label{xcg}
\end{align}
where the lower bound is achievable by a deterministic quantizer that maps $\mathbf x$ to the nearest quantization point, denoted as $\mathcal Q_{\mathrm{det}}(\mathbf x)$.
Let ${\mathcal Q}_n(x) = {\bf e}_n^\top \mathcal Q_{\mathrm{det}}(x{\bf e}_n)$ be the projection of $\mathcal Q_{\mathrm{det}}$ on its $n$th element, where ${\bf e}_n$ is the $n$th canonical vector.
Since $\Vert \mathcal Q_{\mathrm{det}}(x{\bf e}_n) - x{\bf e}_n \Vert_2 \geq \vert {\mathcal Q}_n(x) - x\vert$, from \eqref{xcg} it follows that
\begin{align*}
\QNSC \Vert x{\bf e}_n\Vert_2
=
\QNSC |x|
\geq
\Vert \mathcal Q_{\mathrm{det}}(x{\bf e}_n) - x{\bf e}_n \Vert_2 \geq \vert {\mathcal Q}_n(x) - x\vert,
\end{align*}
$ \forall x \in [-\RANGE,\RANGE],$
hence ${\mathcal Q}_n$ satisfies the BC-rule with $\QNSB=0$ as well.
Note that ${\mathcal Q}_n$ is a scalar quantizer with $N_n\leq |\mathbb Q|$ quantization points.
However, Lemma \ref{lemma:quant} dictates that
$\RANGE=0$ for this quantizer, hence the contradiction. We have thus proved the statement of Corollary \ref{coro:converse} for both deterministic and probabilistic compression rules.

\section{Communication cost analysis} \label{pf:comm_cost}
\subsection{Proof of Theorem \ref{thm:comm_cost_linear}}
\label{pf:thm:comm_cost_linear}
We first present some preliminary results instrumental in proving Theorem \ref{thm:comm_cost_linear}, whose proofs are deferred to Appendix~\ref{app:pf_commcost_aux}.
The idea of the proof is to study the asymptotic behavior of an upper bound on the number of bits required per iteration, provided in the following lemma.
\begin{lemma} \label{lemma:num_bits}
Under the same setting as Theorem \ref{thm:conv},
and the proposed ANQ satisfying the BC rule, the average number of bits required per agent at the $k$th iteration, $B^k$, is upper bounded as
\begin{align*}
\mathbb E[B^k] \leq \log_2(S+1)\bigg[ 3d\q+d\q
\log_S
\bigg(3+\frac{\bar F^{\q}(\RTQ, \QNSC, \QNSB^0)}{\sqrt{\m d}\QNSB^0}
\bigg)\bigg]
\end{align*}
$\mathrm{bits},\ \forall k\geq 0,$
where $\bar{F}^{\q}(\RTQ, \QNSC, \QNSB^0)$ is defined in
\eqref{eq:FR_bar}.
\end{lemma}
In addition, we need the following lemma to connect the asymptotic results of the logarithmic function and its argument.
\begin{lemma} \label{lemma:log_bigO}
For positive functions $f, g$, it holds: $\ln f(x) = \mathcal O(\ln g(x))$ as $x\to x_0$ if $\liminf_{x\to x_0} g(x) > 1$, and $f(x) = \mathcal O(g(x))$ as $x \to x_0$.
\end{lemma}

We are now ready to prove the main theorem. From Lemma~\ref{lemma:num_bits}, the average number of bits per agent per iteration is upper bounded by
\begin{align*}
\mathbb E[B_k] \leq  \log_2(S+1)\bigg[ 3d\q+d\q
\log_S
\bigg(3+\frac{\bar F^{\q}(\RTQ, \QNSC, \QNSB^0)}{\sqrt{\m d}\QNSB^0}
\bigg)\bigg],
\end{align*}
$\forall k= 0,1,\ldots,$
where 
$$\bar{F}^{\q}(\RTQ, \QNSC, \QNSB^0)=a_1+a_2{\LIAP}_0,$$
with   $a_1$, $a_2$  and ${\LIAP}_0$ defined  in \eqref{a1_ori}, \eqref{a2_ori} and \eqref{eq:V0}, respectively.
We want to prove that this is 
$\mathbb E[B_k]=\mathcal O (d\ln (1+\frac{1}{\RTQ(\RTQ-\RTUQ)}))$ under Assumption \ref{assump:commcost} and conditions 
$$\QNSB^0 = \Theta (L_Z(\RTQ-\RTUQ)), 1 -\QNSC/\bar\QNSC(\RTQ)=\Omega(1).$$
Using the fact that $\bar{F}^{\q}(\RTQ, \QNSC, \QNSB^0)/\QNSB^0=a_1/\QNSB^0+a_2{\LIAP}_0/\QNSB^0$,
it is sufficient to show that 
$a_1/(\sqrt{md}\QNSB^0)=\mathcal O(1+\frac{1}{\RTQ})$ and
$a_2{\LIAP}_0/(\sqrt{md}\QNSB^0)=\mathcal O(1+\frac{1}{\RTQ(\RTQ-\RTUQ)})$.
In fact, using $\q/\RTQ\leq1/\bar\QNSC(\RTQ)$, we can bound
\begin{align*}
\frac{a_1}{\sqrt{md}\QNSB^0} &\leq
\frac{1}{\RTQ}(1+L_C \RTQ )
\max\{1,(2L_C)^{\q-1}\}\frac{\q}{1 -\QNSC/\bar\QNSC(\RTQ)}
,\\
a_2 &\leq
\frac{2 L_{Z}}{\RTQ}\max\{1,(2L_C)^{\q-1}\}\frac{\q}{1 -\QNSC/\bar\QNSC(\RTQ)}
.
\end{align*}
Clearly, $a_1/(\sqrt{md}\QNSB^0)=\mathcal O(1+\frac{1}{\RTQ})$
and $a_2=\mathcal O(\frac{L_{Z}}{\RTQ})$
since
$L_C,\q=\mathcal O(1)$ and $1 -\QNSC/\bar\QNSC(\RTQ)=\Omega(1)$.
We next study $V_0$. 
From its expression in \eqref{eq:V0}, we notice that
$V_0= \mathcal O(\sqrt{md})$ since $\QNSC/\RTQ\leq 1$,
$\max\{c^*,\Vert{\bf z}^0 - {\bf z}^\infty\Vert_2\} = \mathcal O(\sqrt{\m d})$,
$\QNSB^0 = \Theta(L_Z(\RTQ-\RTUQ))$,
$L_A\cdot L_Z, L_C,\q = \mathcal O(1)$, 
and
$1-\QNSC/\bar{\QNSC}=\Omega(1)$. Therefore, it follows that
$a_2V_0/(\sqrt{md}\QNSB^0)
=\mathcal O(\frac{1}{\RTQ(\RTQ-\RTUQ)})
=\mathcal O(1+\frac{1}{\RTQ(\RTQ-\RTUQ)})$, and the proof is completed by invoking Lemma~\ref{lemma:log_bigO}.

\subsection{Auxiliary results for Theorem \ref{thm:comm_cost_linear}} \label{app:pf_commcost_aux}
\subsubsection{Proof of Lemma \ref{lemma:num_bits}}
Let $\Delta {\bf c}_{i}^{k,\ITERR} \triangleq  {\bf c}_{i}^{k,\ITERR} - \hat{\bf c}_{i}^{k-1, \ITERR} $ be the input to the quantizer for agent $i$, at iteration $k$ and communication round $s$. We now study the average number of bits required for i) the deterministic quantizer, and ii) the probabilistic quantizer with $\mathbb E[\mathcal Q(x)] = x$.

\noindent{\bf i) Deterministic quantizer:} The average number of bits required is bounded as (see \eqref{eq:C_ell_ub}, one can also verify that the following also holds for even $N$)
\begin{align*}
b_i^{k,\ITERR}{\leq}\log_2(S{+}1)\bigg[3d{+}d\log_S\bigg(
2{+}\frac{\ln\big(1{+}\frac{\QNSC\Vert \Delta {\bf c}_{i}^{k,\ITERR}\Vert_2}{\sqrt{d}\QNSB^0\cdot (\RTQ)^k}\big)}{\ln(1+\QNSC) - \ln(1-\QNSC)}\bigg)\bigg]
\end{align*}
$\mathrm{bits}.$
We now upper bound the argument inside the second logarithm.
Since it is a decreasing function of $\omega$, it is maximized in the limit $\omega\to 0$, yielding
\begin{align*}
\frac{\ln\big(1+\frac{\QNSC\Vert \Delta {\bf c}_{i}^{k,\ITERR}\Vert_2}{\sqrt{d}\QNSB^0\cdot (\RTQ)^k}\big)}{\ln(1+\QNSC) - \ln(1-\QNSC)}
\leq
\frac{\Vert \Delta {\bf c}_{i}^{k,\ITERR}\Vert_2}{2\sqrt{d}\QNSB^0\cdot (\RTQ)^k}.
\end{align*}
With this upper bound, we can then upper bound
the average number of bits per agent at communication round $\ITERR$, iteration $k$,
 as
\begin{align*}
&\mathbb E[b^{k, \ITERR}]
\triangleq\frac{1}{\m}\sum_{i=1}^\m
\mathbb E[b_{i}^{k,\ITERR}]
\\&\stackrel{(a)}{\leq}
\log_2(S+1)\bigg[3d+
d
\log_S
\bigg(2+\frac{\sqrt{\mathbb E[\Vert
{\bf c}^{k, \ITERR} - \hat{\bf c}^{k-1, \ITERR}
\Vert_2^2]}}{2\sqrt{\m d}\QNSB^0 \cdot (\RTQ)^k}
\bigg)\bigg] \\
&\stackrel{(b)}{\leq} 
\log_2(S+1)\bigg[3d+d
\log_S
\bigg(2+\frac{F^{\ITERR}(\RTQ, \QNSC, \QNSB^0)}{2\sqrt{\m d}\QNSB^0}
\bigg)\bigg]
\\
&\stackrel{(c)}{\leq} 
\log_2(S+1)\bigg[3d+d
\log_S
\bigg(2+\frac{\bar F^{\q}(\RTQ, \QNSC, \QNSB^0)}{2\sqrt{\m d}\QNSB^0}
\bigg)\bigg],
\end{align*}
where $(a)$ follows from Cauchy–Schwarz
inequality, Jensen's inequality, and the definition of $\Delta {\bf c}^{k, \ITERR}$;  $(b)$ follows from \eqref{qin_dist_induction};
and $(c)$ follows from $F^{\ITERR}\leq\bar F^{\q}$ (see \eqref{eq:FR_bar}).

\noindent{\bf ii) Probabilistic quantizer with $\mathbb E[\mathcal Q(x)] = x$:} Using the same technique as in i), along with the inequality $1-1/x\leq \ln(x)\leq x-1$ for $x>1$ to bound the argument inside the second logarithm of \eqref{eq:C_ell_ub_prob}, we find the bound
\begin{align*}
\mathbb E[b^{k, \ITERR}]
\leq
\log_2(S+1)\bigg[3d+d
\log_S
\bigg(3+\frac{\bar F^{\q}(\RTQ, \QNSC, \QNSB^0)}{\sqrt{\m d}\QNSB^0}
\bigg)\bigg].
\end{align*}
One can also verify that it also holds for even $N$.

Finally, for both the deterministic and probabilistic cases, the proof is completed by
summing over $\ITERR \in [\q]$ to get the average communication cost per agent at iteration $k$.

\subsubsection{Proof of Lemma \ref{lemma:log_bigO}}
If $f(x) = \mathcal O(g(x))$ as $x \to x_0$ and $\liminf_{x\to x_0} g(x) > 1$, then
\begin{align*}
&\limsup_{x\to x_0}\Big\vert\frac{\ln f(x)}{\ln g(x)}\Big\vert
\leq
1+\limsup_{x\to x_0}\Big\vert\frac{\ln(f(x)/g(x))}{\ln g(x)}  \Big\vert
\\&\stackrel{(a)}{\leq} 1+\frac{\limsup_{x\to x_0}|\ln(f(x)/g(x))|}{|\liminf_{x\to x_0}\ln g(x)|} 
<\infty,
\end{align*}
  where $(a)$ follows from $\liminf_{x\to x_0} g(x) > 1$.  This  completes the proof.

\subsection{Proof of Theorem \ref{thm:comm_cost_mesh}} \label{pf:thm:comm_cost_mesh}
Since $\sqrt{\mathbb E[\Vert {\bf z}^{k} - {\bf z}^\infty\Vert^2]} \leq{\LIAP}_0\cdot (\RTQ)^k$ (cf. Theorem \ref{thm:conv}), the $\varepsilon$-accuracy is achieved if $
k[-\ln(\RTQ)]\geq\ln({\LIAP}_0/\sqrt{\m\varepsilon})$, which yields
$k(1-\RTQ)\geq\ln({\LIAP}_0/\sqrt{\m\varepsilon})$ since $-\ln(\RTQ) \geq 1-\RTQ$. Hence, $\varepsilon$-accuracy is achieved if all conditions in Theorem \ref{thm:conv} hold and $k \geq k_{\varepsilon} \triangleq \big\lceil\frac{1}{1-\RTQ}\ln\frac{{\LIAP}_0}{\sqrt{\m\varepsilon}} \big\rceil$. Hence, to compute the upper bound of the communication cost $\sum_{k=0}^{k_{\varepsilon}-1}\mathbb E[B^k]$, we need the upper bounds for $\mathbb E[B^k],\frac{1}{1-\RTQ}$ and ${\LIAP}_0$.
In the proof of Theorem~\ref{thm:comm_cost_linear}, we found that, under Assumption \ref{assump:commcost} and the conditions $1{-}\QNSC/\bar\QNSC(\RTQ)=\Omega(1), \QNSB^0 = \Theta(L_Z(\RTQ-\RTUQ))$,
$$
\mathbb E[B^k] = \mathcal O\Big(d\log_2\Big(1+\frac{1}{\RTQ(\RTQ-\RTUQ)}\Big)\Big),\
{\LIAP}_0=\mathcal O(\sqrt{md}).
$$
Moreover, it can be shown that
$1/\RTQ\leq\frac{(1-\RTUQ)^2}{(1-\RTQ)(\RTQ-\RTUQ)}$
for 
$\RTQ\in(\RTUQ,1)$, and therefore
$$
\frac{1}{\RTQ(\RTQ-\RTUQ)}
\leq
\frac{1}{(1-\RTUQ)}
\Big[
\frac{(1-\RTUQ)^2}{(1-\RTQ)(\RTQ-\RTUQ)}
\Big]^2
\frac{1-\RTQ}{1-\RTUQ}
=\mathcal O\Big(\frac{1}{1-\RTUQ}\Big),
$$
where we used the assumption in Theorem \ref{thm:comm_cost_mesh} that  $\frac{(1-\RTUQ)^2}{(1-\RTQ)(\RTQ-\RTUQ)}=\mathcal O(1)$.
It then follows from Lemma \ref{lemma:log_bigO} that
$\mathbb E[B^k] = \mathcal O(d\log_2(1+\frac{1}{1-\RTUQ}))$.
On the other hand,  we can bound $k_{\varepsilon}$ as
$$
k_{\varepsilon}\leq
\frac{1}{1-\RTUQ}
\Big(1-\RTUQ+\frac{1-\RTUQ}{1-\RTQ}\log_2\frac{{\LIAP}_0}{\sqrt{\m\varepsilon} }\Big)
=\mathcal O\Big(\frac{1}{1-\RTUQ}\log_2\Big(\frac{d}{\varepsilon}\Big)\Big),
$$
since $\frac{1-\RTUQ}{1-\RTQ}=\mathcal O(1)$ and ${\LIAP}_0=\mathcal O(\sqrt{md})$,
Therefore, the communication cost satisfies
\begin{align*}
&\sum_{k=0}^{k_{\varepsilon}-1}\mathbb E[B^k]
=\mathcal O\Big(dk_{\varepsilon}\log_2\Big(1+\frac{1}{1-\RTUQ}\Big)\Big)
\\&
=\mathcal O\Big(
\frac{d}{1-\RTUQ}\log_2\Big(\frac{d}{\varepsilon}\Big)
\log_2\Big(1+\frac{1}{1-\RTUQ}\Big)\Big),
\end{align*}
which completes the proof.

\subsection{Proof of Lemma \ref{lemma:C_ell_ub}} \label{app:pf_C_ell_ub}
Consider $\ell\geq 0$. Using \eqref{Lb},
the number of information symbols required to encode $\ell$ is
$$
b_\ell^*=\min\bigg\{b\geq 0:\ell\leq\bigg\lfloor\frac{(S)^{b+1}-1}{2(S-1)}\bigg\rfloor\bigg\}.
$$
Similarly, for $\ell<0$,
$$
b_\ell^*=\min\bigg\{b\geq 0:-\ell\leq
\bigg\lceil\frac{(S)^{b+1}-1}{2(S-1)}\bigg\rceil-1
\bigg\}.
$$
Since $\lfloor\frac{(S)^{b+1}-1}{2(S-1)}\rfloor
\geq \lceil\frac{(S)^{b+1}-1}{2(S-1)}\rceil-1
\geq \frac{(S)^{b+1}-1}{2(S-1)}-1
$, we can then upper bound $b_\ell, \ell\in\mathbb Z$, as
\begin{align*}
b_\ell^*
&\leq \min\{b\geq 1:1+2(S-1)(1+|\ell|)\leq (S)^{b+1}\} \\
&=
\min\{b\geq 1:
b\geq\log_S(2-1/S+2(1-1/S)|\ell|)
\}\\&
=
\lceil\log_S(2-1/S+2(1-1/S)|\ell|)\rceil.    
\end{align*}
Using $\lceil x\rceil\leq x+1$
We can then further upper bound
\begin{align*}
b_\ell^*&\leq
1+\log_S(2-1/S+2(1-1/S)|\ell|)
\leq
\log_S(2S+2S|\ell|)\\&
\leq
2+\log_S(1+|\ell|).
\end{align*}
This yields the upper bound to the communication cost (including the termination symbol)
\begin{align}
\bar C_{\mathrm{comm}}(\ell)
\leq
3+\log_S(1+|\ell|) \quad \mathrm{symbols}.  \label{eq:C_ell_ub_ell}
\end{align}
Let ${\bf x} = (x_n)_{n=1}^d$.
Note that for both the deterministic and probabilistic quantizers, we can express $\ell(x)$ as
$$
|\ell(x)|\leq
\Big\lceil
c_1+c_2\ln\Big(1+\frac{\omega}{\eta}|x|\Big)
\Big\rceil,
$$
for some $c_1\leq 0$, $c_2>0$
(see \eqref{eq:lstar} and \eqref{eq:lstar_prob} for a closed-form expression of $c_1$ and $c_2$). Invoking \eqref{eq:C_ell_ub_ell} and $\bar C_{\mathrm{comm}}({\bf x}) = \sum_{n=1}^d\bar C_{\mathrm{comm}}(\ell_n)$ yields
\begin{align*}
&C({\bf x})
\leq 3d+\sum_{n=1}^d\log_S\bigg(1 + \Big\lceil
c_1+c_2\ln\Big(1+\frac{\omega}{\eta}|x_n|\Big)
\Big\rceil\bigg) \nonumber\\
&\stackrel{(a)}{\leq}  3d+d\log_S\bigg(
2+ 
c_1+c_2\ln\Big(1+\frac{\omega\Vert \mathbf x\Vert_2}{\sqrt{d}\eta}\Big)\bigg) \quad \mathrm{symbols}  \nonumber \\
&= \log_2(S{+}1)\bigg[3d+d\log_S\bigg(
2+ 
c_1+c_2\ln\Big(1{+}\frac{\omega\Vert \mathbf x\Vert_2}{\sqrt{d}\eta}\Big)\bigg)\bigg]
\end{align*}
$\mathrm{bits},$
where $(a)$ follows from $\lceil x\rceil \leq x+1$,
Jensen's inequality and Cauchy–Schwarz inequality, in order. Invoking the expressions of $c_1$ and $c_2$ from \eqref{eq:lstar} and \eqref{eq:lstar_prob}, respectively, yield the result for the deterministic and probabilistic quantizers with odd $N$. Similar techniques can be used to find $\ell(x)$ and thus the result for quantizers with even $N$.

\section{Examples of \eqref{eq:M}}
\label{app:examples}
In this section, we will show that \eqref{eq:M} contains a gamut of distributed algorithms, corresponding to different choices of $\q, \mathcal C_{i}^{\ITERR}$, and $\mathcal A_i$. Given \eqref{eq:P}, we will assume  that each $f_i$ is $L$-smooth and $\mu$-strongly convex, and define $\tilde{\bf x}_i^* = \arg\min_{{\bf x}_i}f_i({\bf x}_i)$.

Every distributed algorithm on mesh networks we will describe below  alternates one step of optimization with possibly multiple rounds of communications. In each communication round, every agent $i$  combines linearly the signals received by its neighbors using weights $(w_{i j})_{j\in \mathcal{N}_i}$; let ${\bf W} = (w_{i j})_{i,j=1}^\m$. 
Consistently with the undirected graph $\mathcal G$, we will tacitly  assume that ${\bf W}$ is symmetric and doubly stochastic, i.e., ${\bf W}={\bf W}^\top$ and ${\bf W}\cdot\mathbf 1=\mathbf 1$,
with  $w_{i j} > 0$ if $(j,i) \in \mathcal{E}$, and $w_{i j} = 0$ otherwise. We assume that the eigenvalues of $\bf W$ are in $[\nu, 1]$, with $\nu>0$.\footnote{This assumption is also required in \cite{Xu2020} for prox-EXTRA, prox-NEXT, prox-DIGing, and prox-NIDS to achieve $\Vert {\bf z}^k - {\bf z}^\infty\Vert = \mathcal O(\sqrt{\m d }(\RTUQ)^k)$.}  
Note that this condition can be achieved by design: in fact, given a doubly stochastic weight matrix 
$\tilde{\bf W}$, one can choose  
  \begin{equation}\label{eq:W_from_W_tilde}
      \mathbf{W}=\frac{(1+\nu)}{2}\mathbf{I}+\frac{(1-\nu)}{2}\tilde{\mathbf W}, \end{equation} 
 for any given  $\nu \in (0,1]$. Note that, for any given ${\bf z}^0$ and ${\bf z}^\infty$ with bounded entries, it holds $\Vert {\bf z}^0 - {\bf z}^\infty\Vert_2  = \mathcal O(\sqrt{\m d})$.

Finally, in the rest of this section, we will adopt the following notations: ${\bf x}^{k} = ({\bf x}_i^{k})_{i=1}^\m, {\bf y}^{k} = ({\bf y}_i^{k})_{i=1}^\m, {\bf w}^{k} = ({\bf w}_i^{k})_{i=1}^\m, {\bf x} = ({\bf x}_i)_{i=1}^\m, {\bf y} = ({\bf y}_i)_{i=1}^\m$, ${\bf w} = ({\bf w}_i)_{i=1}^\m$
and $\tilde{\bf x}^* = (\tilde{\bf x}_i^*)_{i=1}^{\m}$;  $\hat{\bf W} = {\bf W} \otimes {\bf I}_d$, and ${\bf G}^\dagger$ is the pseudo-inverse of matrix $\bf G$. Given ${\bf x}^{k} = ({\bf x}_i^{k})_{i=1}^\m$, we also define  $\nabla f({\bf x}^k) \triangleq (\nabla f_i({\bf x}_i^k))_{i=1}^{\m}$. {For any function $g:\mathbb{R}^d\to \mathbb{R}$ and positive semi-definite matrix $\bf G$, define $\Vert {\bf x} \Vert_{\bf G} \triangleq \sqrt{{\bf x}^\top {\bf G}{\bf x}}$ and }
\begin{align*}
    {\rm prox}_{{\bf G}, g}({\bf x}) \triangleq \underset{\bf z\in \mathbb{R}^d}{\arg\min}\quad g({\bf z})+\frac{1}{2}\left\Vert  {\bf z}-{\bf x}\right\Vert_{{\bf G}^{-1}}^2. 
\end{align*}

\subsection{GD over star networks \cite{Nesterov2014}}
\label{AppD1}
Consider Problem \eqref{eq:P} with $r\equiv0$ over a   master/workers system. The GD update 
\begin{align}\label{eq:GD-master}
{\bf x}^{k+1} &= {\bf x}^{k} - \frac{\STEP}{\m}\sum_{i=1}^\m \nabla f_i\left({\bf x}^{k}\right),
\end{align}
with ${\bf x}^0 \in \mathbb R^{d}$, is implemented at the master node as follows:   at iteration $k$, the server broadcasts ${\bf x}^{k}$ to the  $\m$ agents;  each agent $i$ then  computes its own gradient $\nabla f_i({\bf x}^{k})$ and sends it back to the master; upon collecting all local gradients, the server updates the variable ${\bf x}^{k+1}$ according to (\ref{eq:GD-master}). 

The GD (\ref{eq:GD-master}) can be cast as \eqref{eq:M} with $\q=1$ round of communications, using  the  following:
\begin{align}
&{\bf z} = {\bf 1}_{\m} \otimes {\bf x}, \nonumber \\
&\hat{\bf c}_{i}^{k,1} = 
\mathcal C_{i}^1\big({\bf z}_i^{k}, {\bf 0} \big) 
= \nabla f_i\left({\bf x}^{k}\right), \label{eq:GD_star_c1}\\
&{\bf z}^{k+1} 
= \mathcal A\big(
{\bf z}^{k}, \hat{\bf c}_{[\m]}^{k,1} \big)
= {\bf 1}_{\m} \otimes{\bf x}^{k} - \frac{\STEP}{\m} \sum_{i=1}^\m {\bf 1}_{\m} \otimes\hat{\bf c}_{i}^{k,1}. \label{eq:GD_star_A}
\end{align}
We now show that the above instance of \eqref{eq:M} satisfies Assumptions \ref{assump:R_conv_z}, \ref{assump:Lipt_A}, \ref{assump:Lipt_C}, and \ref{assump:commcost}.

$\bullet$ \textbf{On Assumption 1:} Using \cite{Nesterov2014} it is not difficult to check that, if $\STEP = 2/(\mu + L)$, then GD over star networks satisfies Assumption \ref{assump:R_conv_z} with 
\begin{align*}
\RTUQ = \frac{\kappa-1}{\kappa+1} < 1,
\end{align*}
and  the norm $\|\bullet\|$ defined as 
\begin{align*}
\Vert{\bf z} \Vert =  \sqrt{m}\Vert  {\bf x} \Vert_2.
\end{align*}
Note that $\Vert{\bf z} \Vert=\Vert{\bf z} \Vert_2$.

$\bullet$ \textbf{On Assumptions \ref{assump:Lipt_A}, \ref{assump:Lipt_C}, and \ref{assump:commcost}:} Based on \eqref{eq:GD_star_c1} and \eqref{eq:GD_star_A}, the mappings $\mathcal A$ and $\mathcal C$ read
\begin{align*}
    \mathcal A({\bf z}, {\bf c}^1) = {\bf z} - \frac{\STEP}{\m} \sum_{i=1}^\m {\bf 1}_{\m} \otimes {\bf c}_{i}, \quad
    \mathcal C^1({\bf z}, {\bf 0}) =  \nabla f({\bf z}),
\end{align*}
respectively; and $\mathcal Z = \{{\bf 1}_{\m}\otimes {\bf x}:{\bf x}\in \mathbb R^d\}$. 
Since
\begin{align*}
\mathcal A\left({\bf z}, {\bf c}\right) - \mathcal A\left({\bf z}, {\bf c}'\right)
= - \frac{\STEP}{\m} \sum_{i=1}^m {\bf 1}_{\m} \otimes \left({\bf c}_{i} - {\bf c}_{i}'\right),
\end{align*}
we have
\begin{align*}
\big\Vert\mathcal A\left({\bf z}, {\bf c}\right) - \mathcal A\left({\bf z}, {\bf c}'\right)\big\Vert
= \frac{\STEP}{\sqrt{\m}}\left\Vert \sum_{i=1}^m ({\bf c}_{i} - {\bf c}_{i}')\right\Vert_2
{\leq} \STEP\left\Vert {\bf c}-{\bf c}'\right\Vert_2.
\end{align*}
Hence,  Assumption \ref{assump:Lipt_A} holds with $L_A =\STEP$. 

We now derive $L_C$ and $L_Z$. Note that 
\begin{align*}
\Vert \mathcal C^1({\bf z}, {\bf 0})- \mathcal C^1({\bf z}', {\bf 0}) \Vert_2
\leq  L\left\Vert{\bf z} - {\bf z}'\right\Vert_2,
\end{align*}
which implies that Assumption \ref{assump:Lipt_C} holds with  $L_C = 0$ and $L_Z =  L$. Since $\STEP\leq 2/L$, it follows that $L_A\cdot L_Z=\mathcal O(1)$ and $\Vert{\mathcal C}^{1}({\bf z}^{0}, {\bf 0})\Vert_2 \leq  L \Vert{\bf z}^{0} - \tilde{\bf x}^* \Vert_2 = \mathcal O(L_Z\sqrt{\m d})$.
Hence, Assumption \ref{assump:commcost} holds.

\subsection{(Prox-)EXTRA \cite{Xu2020}}
\label{subsec:ABC_EXTRA}
The update of prox-EXTRA solving \eqref{eq:P} reads 
\begin{align*}
{\bf x}^{k} &= {\rm prox}_{\STEP{\bf I}, r}({\bf w}^{k}), \\
{\bf w}^{k+1} &= \hat{\bf W}{\bf x}^{k}-\STEP \nabla f({\bf x}^{k})-{\bf y}^{k},\\
{\bf y}^{k+1} &= {\bf y}^{k}+\big({\bf I} - \hat{\bf W}\big){\bf w}^{k+1},
\end{align*}
with ${\bf y}^0 = {\bf 0}$ and ${\bf w}^0 \in \mathbb R^{\m d}$.

Prox-EXTRA can be cast as \eqref{eq:M} with $\q=2$ rounds of communications with 
\begin{align}
& {\bf z}^\top=[{\bf y}^\top,{\bf w}^\top], \nonumber\\
& \hat{\bf c}_{i}^{k,1}=\mathcal C_{i}^1\big({\bf z}_i^{k},{\bf 0}\big) 
= {\rm prox}_{\STEP{\bf I}, r}({\bf w}_i^{k}), \label{eq:p_EXTRA_c1} \\
&
\hat{\bf c}_{i}^{k,2}{=}\mathcal C_{i}^2\big({\bf z}_i^{k}, \hat{\bf c}_{\mathcal N_i}^{k,1}\big) 
{=}\sum_{j \in \mathcal N_i}w_{ij}\hat{\bf c}_{j}^{k,1} - \STEP \nabla f_i(\hat{\bf c}_i^{k,1}) - {\bf y}_i^{k}, \label{eq:p_EXTRA_c2} \\
&{\bf z}_i^{k+1} 
= \mathcal A_i\big(
{\bf z}_i^{k}, \hat{\bf c}_{{\mathcal N}_i}^{k,1},\hat{\bf c}_{{\mathcal N}_i}^{k,2} \big)\nonumber\\&
\qquad= \left[
\begin{array}{*{20}c}
{\bf y}_i^{k}
+\sum_{j \in \mathcal N_i} w_{ij}(\hat{\bf c}_{i}^{k,2}-\hat{\bf c}_{j}^{k,2}) \\
\hat{\bf c}_{i}^{k,2}
\end{array}
\right]. \label{eq:p_EXTRA_A}
\end{align}

We now show that the above instance of \eqref{eq:M} satisfies  Assumptions \ref{assump:R_conv_z}, \ref{assump:Lipt_A}, \ref{assump:Lipt_C}, and \ref{assump:commcost}.

$\bullet$ \textbf{On Assumption 1:} Using \cite[Theorem 18]{Xu2020} it is not difficult to check that, if $\STEP = \frac{2\rho_{m}({\bf W})}{L + \mu \rho_{m}({\bf W})}$
and $\nu=\frac{\kappa}{\kappa+1}$,
then prox-EXTRA satisfies Assumption \ref{assump:R_conv_z} with 
\begin{align*}
  \RTUQ &= \max\Big\{\frac{\kappa-\rho_m({\bf W})}{\kappa+\rho_m({\bf W})}\frac{1}{\sqrt{\rho_m({\bf W})}}, \sqrt{\rho_{2}\big({\bf W}\big)}\Big\}
\\&
 \leq
\max\Big\{
\frac{\kappa}{\kappa+1}
\sqrt{\frac{(\kappa+1)^3}{\kappa(\kappa+2)^2}}
, \sqrt{\rho_{2}\big({\bf W}\big)}\Big\} \\&
 \leq
\max\Big\{
\frac{\kappa}{\kappa+1}
, \sqrt{\rho_{2}\big({\bf W}\big)}\Big\} 
  < 1,
\end{align*}
(note that $\rho_m({\bf W})\geq\nu$)
and the norm $\|\bullet\|$ defined as  
\begin{align*}
&\Vert{\bf z}\Vert^2 = {\bf y}^\top({\bf I} - \hat{\bf W})^\dagger{\bf y}+{\bf w}^\top\hat{\bf W}^{-1}{\bf w}.
\end{align*}
Note that  $\Vert{\bf z}\Vert^2\geq \Vert{\bf z}\Vert^2_2$, due to
  $\rho_{i}({\bf W}) \in [\nu,1]$, $i \in [m]$.

$\bullet$ \textbf{On Assumptions {\ref{assump:Lipt_A}, \ref{assump:Lipt_C}, and} \ref{assump:commcost}:} 
Based on \eqref{eq:p_EXTRA_c1}, \eqref{eq:p_EXTRA_c2} and \eqref{eq:p_EXTRA_A}, the mappings $\mathcal A$ and $\mathcal C$ read
\begin{align*}
    &\mathcal A({\bf z}, {\bf c}^1, {\bf c}^2) = \left[
\begin{array}{*{20}c}
{\bf y} + ({\bf I} - \hat{\bf W}) {\bf c}^{2}  \\
{\bf c}^{2}
\end{array}
\right], \\
&\mathcal C^1({\bf z}, {\bf 0}) = {\rm prox}_{\STEP{\bf I}, r}({\bf w}),\ \text{and} \ 
\mathcal C^2({\bf z}, {\bf c}) = \hat{\bf W}{\bf c} - \STEP \nabla f({\bf c}) - {\bf y},
\end{align*}
respectively; and $\mathcal Z = \mathrm{span}({\bf I}-\hat{\bf W}) \times \mathbb R^{md}$, where we defined (with a slight abuse of notation) ${\rm prox}_{\STEP{\bf I}, r}({\bf w}) = ({\rm prox}_{\STEP{\bf I}, r}({\bf w}_i))_{i=1}^{\m}$.
Note that
\begin{align*}
& \big\Vert{\mathcal A}\left({\bf z}, {\bf c}^{1}, {\bf c} \right) - {\mathcal A}\left({\bf z}, {\bf c}^{1}, {\bf c}'\right) \big\Vert^2 
= \Vert \sqrt{\hat{\bf W}^{-1}}({\bf c} - {\bf c}')\Vert_2^2
\\&
 +\Big\Vert \sqrt{{\bf I}-\hat{\bf W}}({\bf c} - {\bf c}')\Big\Vert_2^2 \leq (1+\nu^{-1}) \Vert {\bf c} - {\bf c}'\Vert_2^2
 \leq 3\Vert {\bf c} - {\bf c}'\Vert_2^2,
\end{align*}
(note that $\nu^{-1}=1+\kappa^{-1}\leq 2$)
and ${\mathcal A}\left({\bf z}, {\bf c}^1, {\bf c}^2\right)$ is constant with respect to ${\bf c}^1$, hence Assumption \ref{assump:Lipt_A} holds with $L_A = \sqrt{3}$.

We next derive $L_C$ and $L_Z$. Since the proximal mapping is non-expansive \cite{Moreau1965}, it follows that
\begin{align*}
&\big\Vert \mathcal C^1({\bf z}, {\bf 0})- \mathcal C^1({\bf z}', {\bf 0})\big\Vert_2 
= \Vert {\rm prox}_{\STEP{\bf I}, r}({\bf w}) - {\rm prox}_{\STEP{\bf I}, r}({\bf w}')\Vert_2 \\&\qquad
\leq \Vert {\bf w}-{\bf w}'\Vert_2
\leq \Vert {\bf z} - {\bf z}'\Vert_2, \\
&\big\Vert \mathcal C^2({\bf z}, {\bf c})- \mathcal C^2({\bf z}', {\bf c}')\big\Vert_2
\\&\qquad
{=}\big\Vert \hat{\bf W} ({\bf c}{-}{\bf c}'){-}\STEP (\nabla f({\bf c}){-}\nabla f({\bf c}'))
{-}({\bf y}{-}{\bf y}')\big\Vert_2 \\
&\qquad\leq   \Vert {\bf c}-{\bf c}'\Vert_2 + \STEP L \Vert {\bf c}-{\bf c}'\Vert_2 + \Vert {\bf z}-{\bf z}'\Vert_2.
\end{align*}
Hence,
 Assumption \ref{assump:Lipt_C} holds with $L_C = 1 + \STEP L$ and $L_Z = 1$. Since $\STEP\leq 2/L$, it follows that $L_C = \mathcal O(1)$. For the initial conditions, we have
\begin{align*}
&\Vert{\mathcal C}^{1}({\bf z}^{0}, {\bf 0})\Vert_2 = \Vert{\rm prox}_{\STEP{\bf I}, r}({\bf w}^0)\Vert_2 
\\&
\stackrel{(a)}{\leq} \Vert{\rm prox}_{\STEP{\bf I}, r}({\bf w}^0) - {\rm prox}_{\STEP{\bf I}, r}({\bf w}^\infty)\Vert_2 + \Vert{\rm prox}_{\STEP{\bf I}, r}({\bf w}^\infty)\Vert_2 \\&
\stackrel{(b)}{\leq} \Vert{\bf w}^0 - {\bf w}^\infty\Vert_2 + \Vert {\bf x}^\infty \Vert_2 = \mathcal O(\sqrt{\m d}), \\&
\Vert{\mathcal C}^{2}({\bf z}^{0}, {\bf 0})\Vert_2
= \Vert {\bf y}^0+\STEP \nabla f({\bf 0})\Vert_2
\\&
\leq \Vert {\bf y}^0 - {\bf y}^\infty\Vert_2 + \Vert {\bf y}^\infty\Vert_2
+\STEP \Vert\nabla f({\bf 0})\Vert_2
\\&
\leq \Vert {\bf y}^0 - {\bf y}^\infty\Vert_2 + \Vert {\bf y}^\infty\Vert_2
+\STEP L\Vert\tilde{\bf x}^*\Vert_2
= \mathcal O(\sqrt{\m d}),
\end{align*}
where $(a)$ follows from the triangle inequality; and $(b)$ follows from the non-expansive property of the proximal mapping, and ${\bf x}^\infty={\rm prox}_{\STEP{\bf I}, r}({\bf w}^\infty)$ at the fixed point.

Therefore, $L_A\cdot L_Z=\mathcal O(1),L_C = \mathcal O(1)$, $\Vert{\mathcal C}^{1}({\bf z}^{0}, {\bf 0})\Vert_2=\mathcal O(L_Z\sqrt{\m d})$ and $ \Vert{\mathcal C}^{2}({\bf z}^{0}, {\bf 0})\Vert_2=\mathcal O(L_Z\sqrt{\m d})$; hence Assumption \ref{assump:commcost} holds.

\subsection{(Prox-)NIDS \cite{Xu2020}}
\label{subsec:ABC_NIDS}
The update of prox-NIDS solving \eqref{eq:P}, reads
\begin{align*}
{\bf x}^{k} &= {\rm prox}_{\STEP{\bf I}, r}({\bf w}^{k}), \\
{\bf w}^{k+1} &= \hat{\bf W}\big({\bf x}^{k}-\STEP \nabla f({\bf x}^{k})\big)-{\bf y}^{k},\\
{\bf y}^{k+1} &= {\bf y}^{k}+({\bf I} - \hat{\bf W}){\bf w}^{k+1},  
\end{align*}
with ${\bf y}^0 = {\bf 0}$ and ${\bf w}^0 \in \mathbb R^{\m d}$.

Prox-NIDS can be cast as \eqref{eq:M} with $\q=2$ rounds of communications, using the following:
\begin{align}
&{\bf z}^\top=[{\bf y}^\top,{\bf w}^\top], \nonumber \\
&\mathcal C_{i}^1\big({\bf z}_i^{k}, {\bf 0}\big) 
= {\rm prox}_{\STEP{\bf I}, r}({\bf w}_i^{k}) - \STEP \nabla f_i({\rm prox}_{\STEP{\bf I}, r}({\bf w}_i^{k})), \label{eq:p_NIDS_c1} \\
&\mathcal C_i^{2}\big({\bf z}_i^{k}, \hat{\bf c}_{\mathcal N_i}^{k,1}\big) 
= \sum_{j \in \mathcal N_i}w_{ij}\hat{\bf c}_{j}^{k,1}  - {\bf y}_i^{k}, \label{eq:p_NIDS_c2} \\
&{\bf z}_i^{k+1} 
= \mathcal A_i\big(
{\bf z}_i^{k}, \hat{\bf c}_{{\mathcal N}_i}^{k,1},\hat{\bf c}_{{\mathcal N}_i}^{k,2} \big)
\nonumber\\&
\qquad= \left[
\begin{array}{*{20}c}
{\bf y}_i^{k}+\sum_{j \in \mathcal N_i} w_{ij} \left(\hat{\bf c}_{i}^{k,2} - \hat{\bf c}_{j}^{k,2}\right)  \\
\hat{\bf c}_{i}^{k,2}
\end{array}
\right].  \label{eq:p_NIDS_A}
\end{align}
We now show that the above instance of \eqref{eq:M} satisfies Assumptions \ref{assump:R_conv_z}, \ref{assump:Lipt_A}, \ref{assump:Lipt_C}, and \ref{assump:commcost}.

$\bullet$ \textbf{On Assumption 1:} Using \cite[Theorem 18]{Xu2020} it is not difficult to check that, if $\STEP = 2/(\mu + L)$, then prox-NIDS satisfies Assumption~\ref{assump:R_conv_z} with  
\begin{align*}
    \RTUQ = \max\Big\{\frac{\kappa-1}{\kappa+1}, \sqrt{\rho_{2}\big({\bf W}\big)}\Big\} < 1,
\end{align*}
and the norm $\|\bullet\|$ defined as 
\begin{align*}
\Vert{\bf z}\Vert^2 ={\bf y}^\top\hat{\bf W}^{-1}({\bf I}-\hat{\bf W})^\dagger \hat{\bf W}^{-1}{\bf y}+ {\bf w}^\top \hat{\bf W}^{-1}{\bf w}.
\end{align*}
Note that $\Vert{\bf z}\Vert^2 \geq \Vert{\bf z}\Vert_2^2$.

$\bullet$ \textbf{On Assumptions \ref{assump:Lipt_A}, \ref{assump:Lipt_C}, and \ref{assump:commcost}:} Based on \eqref{eq:p_NIDS_c1}, \eqref{eq:p_NIDS_c2}, and \eqref{eq:p_NIDS_A}, the mappings $\mathcal A$ and $\mathcal C$ read
\begin{align*}
    &\mathcal A({\bf z}, {\bf c}^1, {\bf c}^2) = \left[
\begin{array}{*{20}c}
{\bf y} + ({\bf I} - \hat{\bf W}) {\bf c}^{2}  \\
{\bf c}^{2}
\end{array}
\right], \\
&\mathcal C^1({\bf z}, {\bf 0}) = {\rm prox}_{\STEP{\bf I}, r}({\bf w}) - \STEP \nabla f({\rm prox}_{\STEP{\bf I}, r}({\bf w})),\  \text{and} \\&
\mathcal C^2({\bf z}, {\bf c}) = \hat{\bf W}{\bf c} - {\bf y},
\end{align*}
respectively; and $\mathcal Z = \mathrm{span}({\bf I}-\hat{\bf W}) \times \mathbb R^{md}$. 
 Note that
\begin{align*}
& \big\Vert{\mathcal A}\left({\bf z}, {\bf c}_{1}, {\bf c} \right) - {\mathcal A}\left({\bf z}, {\bf c}_{1}, {\bf c}'\right)  \big\Vert 
= \Vert\hat{\bf W}^{-1}({\bf c} - {\bf c}')\Vert_2
\\&\quad\leq  \nu^{-1}\Vert {\bf c} - {\bf c}'\Vert_2,
\end{align*}
and ${\mathcal A}({\bf z}, {\bf c}^1, {\bf c}^2)$ is constant with respect to ${\bf c}^1$.
It follows that Assumption \ref{assump:Lipt_A} holds with $L_A =\nu^{-1}$.

We next derive $L_C$ and $L_Z$. Using the non-expansive property of the proximal mapping, it holds
\begin{align*}
&\Vert \mathcal C^1({\bf z}, {\bf 0})- \mathcal C^1({\bf z}', {\bf 0}) \Vert_2 = 
\Vert ({\rm prox}_{\STEP{\bf I}, r}({\bf w}) - {\rm prox}_{\STEP{\bf I}, r}({\bf w}')) 
\\&
\quad- \STEP (\nabla f({\rm prox}_{\STEP{\bf I}, r}({\bf w}))-\nabla f({\rm prox}_{\STEP{\bf I}, r}({\bf w}')))\Vert_2 \\
&\leq (1+\STEP L) \Vert {\rm prox}_{\STEP{\bf I}, r}({\bf w})-{\rm prox}_{\STEP{\bf I}, r}({\bf w}')\Vert_2
\\&
\leq (1+\STEP L) \Vert {\bf w}-{\bf w}'\Vert_2
\leq (1+\STEP L) \Vert {\bf z}-{\bf z}'\Vert_2, \\
&\Vert \mathcal C^2({\bf z}, {\bf c} )- \mathcal C^2({\bf z}', {\bf c}' )\Vert_2
=\Vert
\hat{\bf W}({\bf c}-{\bf c}') - ({\bf y}-{\bf y}')\Vert_2
\\&
\leq \Vert \hat{\bf W} ({\bf c}-{\bf c}')\Vert_2 + \Vert {\bf y}-{\bf y}'\Vert_2 
\leq  \Vert {\bf c}-{\bf c}'\Vert_2 + \Vert {\bf z}-{\bf z}'\Vert_2,
\end{align*}
which implies that Assumption \ref{assump:Lipt_C} holds with $L_C = 1$ and $L_Z = 1+\STEP L$. Since $\STEP\leq 2/L$, it follows that $L_Z = \mathcal O(1)$. For the initial conditions, 
using ${\bf x}^{\infty}={\rm prox}_{\STEP{\bf I}, r}({\bf w}^{\infty})$ for the fixed point,
we have
\begin{align*}
&\Vert\mathcal C^1\big({\bf z}^{0}, {\bf 0}\big) \Vert_2
\leq
\Vert {\rm prox}_{\STEP{\bf I}, r}({\bf w}^{0})-{\rm prox}_{\STEP{\bf I}, r}({\bf w}^{\infty})+{\bf x}^{\infty}\Vert_2
\\&
+\STEP\Vert\nabla f({\rm prox}_{\STEP{\bf I}, r}({\bf w}^{0}))
-\nabla f({\rm prox}_{\STEP{\bf I}, r}({\bf w}^{\infty}))
+\nabla f({\bf x}^{\infty})\Vert_2
\\
&\leq
\Vert {\bf w}^{0}-{\bf w}^{\infty}\Vert_2+\Vert {\bf x}^{\infty}\Vert_2
+\STEP L\Vert{\bf w}^{0}-{\bf w}^{\infty}\Vert_2\\&
\quad+\STEP L\Vert{\bf x}^{0}-\tilde{\bf x}^{*}\Vert_2
= \mathcal O(\sqrt{\m d})
\end{align*}
 and $\Vert{\mathcal C}^{2}({\bf z}^{0}, {\bf 0})\Vert_2 \leq \Vert {\bf y}^0 - {\bf y}^\infty\Vert_2 +  \Vert {\bf y}^\infty\Vert_2 = \mathcal O(\sqrt{\m d})$.

Therefore, $L_A\cdot L_Z=\mathcal O(1),L_C = \mathcal O(1),\Vert{\mathcal C}^{1}({\bf z}^{0}, {\bf 0})\Vert_2=\mathcal O(L_Z\sqrt{\m d})$, and $\Vert{\mathcal C}^{2}({\bf z}^{0}, {\bf 0})\Vert_2=\mathcal O(L_Z\sqrt{\m d})$; hence Assumption \ref{assump:commcost} holds.

\subsection{NIDS \cite{Li2019}}
\label{subsec:NIDS}
The update of NIDS solving \eqref{eq:P} with $r\equiv0$, reads
\begin{align*}
    {\bf x}^{k+1} &= {\bf x}^k - \STEP\nabla f({\bf x}^k)\\& \quad- \STEP\Big[  {\bf y}^{k}+c({\bf I}-\check{\bf W})\big({\bf x}^k-\STEP\nabla f({\bf x}^k) - \STEP{\bf y}^k\big)\Big], \\
    {\bf y}^{k+1} &= {\bf y}^{k}+c({\bf I}-\check{\bf W})\big({\bf x}^k-\STEP\nabla f({\bf x}^k) - \STEP{\bf y}^k\big),
\end{align*}
with ${\bf x}^0 \in \mathbb R^{md}$ and ${\bf y}^0 = {\bf 0}$, where the step-size satisfies $\STEP < 2/L$, 
$\check{\bf W} = \tilde{\bf W} \otimes {\bf I}$, $\tilde{\bf W}$ is a doubly stochastic weight matrix satisfying the communication topology, and $c >0$ satisfies $c\STEP({\bf I} - \tilde{\bf W}) \preceq {\bf I}$.

NIDS \cite{Li2019} can be cast as \eqref{eq:M} with $\q=1$ round of communications, using the following:
\begin{align}
&{\bf z}^\top = [{\bf x}^\top, \STEP{\bf y}^\top], \nonumber \\
&\mathcal C_i^1({\bf z}_i^k, {\bf 0}) = {\bf x}_i^k - \STEP\nabla f_i({\bf x}_i^k) - \STEP {\bf y}_i^k, \label{eq:NIDS_c1} \\
&{\bf z}_i^{k+1} = \mathcal A_i({\bf z}_i^k, \hat{\bf c}_{\mathcal N_i}^{k, 1}) 
 \label{eq:NIDS_A}\\&
= \left[ 
\begin{array}{*{20}c}
{\bf x}_i^k - \STEP\nabla f_i({\bf x}_i^k)
 - \STEP{\bf y}_i^{k}
  - \STEP c\sum_{j\in\mathcal N_i}\tilde w_{ij}(\hat{\bf c}_i^{k, 1} - \hat{\bf c}_j^{k, 1})
  \\
\STEP{\bf y}_i^{k}+ \STEP c\sum_{j\in\mathcal N_i}\tilde w_{ij}(\hat{\bf c}_i^{k, 1} - \hat{\bf c}_j^{k, 1})
\end{array}
\right]. \nonumber
\end{align}

We now show that the above instance of \eqref{eq:M} satisfies Assumptions \ref{assump:R_conv_z}, \ref{assump:Lipt_A}, \ref{assump:Lipt_C}, and \ref{assump:commcost}.

$\bullet$ \textbf{On Assumption 1:} Using \cite{Li2019} it is not difficult to check that, if $\STEP = 1/L$ and $c =\frac{1}{2\STEP}$, then NIDS satisfies Assumption~\ref{assump:R_conv_z} with
\begin{align*}
    \RTUQ = \max\Big\{\sqrt{1- {\kappa}^{-1}}, 
    \sqrt{\frac{1+\rho_{2}(\tilde{\bf W})}{2}} \Big\},
\end{align*}
and the norm $\|\bullet\|$ defined as 
\begin{align*}
\Vert{\bf z}\Vert^2 = \Vert {\bf x}\Vert_2^2 + \frac{1}{\STEP}\Vert\STEP{\bf y}\Vert_{c^{-1}({\bf I}-\check{\bf W})^\dagger}^2,
\end{align*}
which satisfies $\Vert{\bf z}\Vert^2 \geq \Vert{\bf z}\Vert_2^2$.
 Note that, with  $\mathbf{W}$ defined as in \eqref{eq:W_from_W_tilde}, we can further bound
\begin{align*}
    \RTUQ \leq\max\Big\{\sqrt{1- {\kappa}^{-1}}, 
    \sqrt{\rho_{2}({\bf W})} \Big\}<1.
\end{align*}

$\bullet$ \textbf{On Assumptions \ref{assump:Lipt_A}, \ref{assump:Lipt_C}, and \ref{assump:commcost}:} Based on \eqref{eq:NIDS_c1} and \eqref{eq:NIDS_A}, the mappings $\mathcal A$ and $\mathcal C$ read
\begin{align*}
    &\mathcal A({\bf z}, {\bf c}^1) = \left[
\begin{array}{*{20}c}
{\bf x} - \STEP \nabla f({\bf x}) - \STEP{\bf y} -\STEP c({\bf I} - \check{\bf W}){\bf c}^1   \\
\STEP{\bf y} +\STEP c({\bf I} - \check{\bf W}){\bf c}^1
\end{array}
\right] \quad \text{and} \\
&\mathcal C^1({\bf z}, {\bf 0}) = {\bf x} - \STEP \nabla f({\bf x}) - \STEP {\bf y},
\end{align*}
respectively; and $\mathcal Z = \mathbb R^{md} \times \mathrm{span}({\bf I}-\check{\bf W})$. 
Note that
\begin{align*}
& \big\Vert{\mathcal A}\left({\bf z}, {\bf c} \right) - {\mathcal A}\left({\bf z}, {\bf c}'\right)  \big\Vert^2 
= (c\STEP)^2 ({\bf c} - {\bf c}')^\top({\bf I} - \check{\bf W})^2({\bf c} - {\bf c}') 
\\&
 + c\STEP ({\bf c} - {\bf c}')^\top({\bf I} - \check{\bf W})({\bf c} - {\bf c}') \leq 2\Vert {\bf c} - {\bf c}'\Vert_2^2.
\end{align*}
It follows that Assumption \ref{assump:Lipt_A} holds with $L_A = \sqrt{2}$.

We next derive $L_C$ and $L_Z$. Using the non-expansive property of the proximal mapping, it holds
\begin{align*}
&\Vert \mathcal C^1({\bf z}, {\bf 0})- \mathcal C^1({\bf z}', {\bf 0}) \Vert_2 \\
&= \Vert ({\bf x} - {\bf x}') - \STEP[\nabla f({\bf x})-\nabla f({\bf x}')] - \STEP({\bf y} - {\bf y}')\Vert_2 \\
&\leq 
\Vert {\bf x} - {\bf x}'\Vert_2
+\Vert \STEP{\bf y} - \STEP{\bf y}'\Vert_2
+\STEP L\Vert {\bf x}-{\bf x}'\Vert_2
\\&
\leq 
(\sqrt{2}+\STEP L)\Vert {\bf z} - {\bf z}'\Vert_2,
\end{align*}
which implies that Assumption \ref{assump:Lipt_C} holds with $L_C = 1$ and $L_Z = \sqrt{2}+\STEP L$. Since $\STEP =1/L$, it follows that $L_Z = \mathcal O(1)$. For the initial conditions, we have 
\begin{align*}
&\Vert{\mathcal C}^{1}({\bf z}^{0}, {\bf 0})\Vert_2
\leq
\Vert{\bf x}^0-{\bf x}^\infty+{\bf x}^\infty\Vert_2
+\STEP\Vert\nabla f({\bf x}^0)-\nabla f(\tilde{\bf x}^*)\Vert_2
\\&
\leq
\Vert{\bf x}^0-{\bf x}^\infty\Vert_2
+\Vert{\bf x}^\infty\Vert_2
+\STEP L\Vert{\bf x}^0-\tilde{\bf x}^*\Vert_2= \mathcal O(\sqrt{\m d})
.\end{align*}
Therefore, $L_A\cdot L_Z=\mathcal O(1),L_C = \mathcal O(1),\Vert{\mathcal C}^{1}({\bf z}^{0}, {\bf 0})\Vert_2=\mathcal O(L_Z\sqrt{\m d})$; hence Assumption \ref{assump:commcost} holds.

\subsection{(Prox-)NEXT \cite{Xu2020}}
\label{subsec:ABC_NEXT}
The update of prox-NEXT solving \eqref{eq:P} reads
\begin{align*}
{\bf x}^{k} &= {\rm prox}_{\STEP{\bf I}, r}({\bf w}^{k}), \\
{\bf w}^{k+1} &= \hat{\bf W}^2\big({\bf x}^{k}-\STEP \nabla f({\bf x}^{k})\big)-{\bf y}^{k},\\
{\bf y}^{k+1} &= {\bf y}^{k}+({\bf I} - \hat{\bf W})^2{\bf w}^{k+1}, 
\end{align*}
with ${\bf y}^0 = {\bf 0}$ and ${\bf w}^0 \in \mathbb R^{\m d}$.

Prox-NEXT can be cast as \eqref{eq:M} with $\q=4$ rounds of communications, using the following definitions:
\begin{align}
&{\bf z}^\top=[{\bf y}^\top,{\bf w}^\top], \nonumber \\
&\hat{\bf c}_{i}^{k,1}{=}
\mathcal C_{i}^1\big({\bf z}_i^{k}, {\bf 0}\big) 
{=}{\rm prox}_{\STEP{\bf I}, r}({\bf w}_i^{k}){-}\STEP \nabla f_i({\rm prox}_{\STEP{\bf I}, r}({\bf w}_i^{k})), \label{eq:p_NEXT_c1} \\
&\hat{\bf c}_{i}^{k,2} = 
\mathcal C_{i}^2\big({\bf z}_i^{k}, \hat{\bf c}_{\mathcal N_i}^{k,1} \big) 
= \sum_{j \in \mathcal N_i}w_{ij}\hat{\bf c}_{j}^{k,1}, \label{eq:p_NEXT_c2} \\
&\hat{\bf c}_{i}^{k,3} = 
\mathcal C_{i}^3\big({\bf z}_i^{k}, \hat{\bf c}_{\mathcal N_i}^{k,2}\big) 
= \sum_{j \in \mathcal N_i}w_{ij}\hat{\bf c}_{j}^{k,2} - {\bf y}_i^{k}, \label{eq:p_NEXT_c3} \\
&\hat{\bf c}_{i}^{k,4} = 
\mathcal C_{i}^4\big({\bf z}_i^{k}, \hat{\bf c}_{\mathcal N_i}^{k,3}\big) 
= \sum_{j \in \mathcal N_i}w_{ij}\big(\hat{\bf c}_{i}^{k,3} - \hat{\bf c}_{j}^{k,3}\big), \label{eq:p_NEXT_c4} \\
&{\bf z}_i^{k+1} 
= \mathcal A_i\big(
{\bf z}_i^{k}, \hat{\bf c}_{{\mathcal N}_i}^{k,1},\hat{\bf c}_{{\mathcal N}_i}^{k,2}, \hat{\bf c}_{{\mathcal N}_i}^{k,3},\hat{\bf c}_{{\mathcal N}_i}^{k,4} \big) 
\nonumber\\&
\qquad= \left[
\begin{array}{*{20}c}
{\bf y}_i^{k} + \sum_{j \in \mathcal N_i} w_{ij} \big(\hat{\bf c}_{i}^{k,4} - \hat{\bf c}_{j}^{k,4}\big)  \\
\hat{\bf c}_{i}^{k,3}
\end{array}
\right]. \label{eq:p_NEXT_A} 
\end{align}
We now show that the above instance of \eqref{eq:M} satisfies Assumptions \ref{assump:R_conv_z}, \ref{assump:Lipt_A}, \ref{assump:Lipt_C}, and \ref{assump:commcost}.

$\bullet$ \textbf{On Assumption 1:} Using \cite[Theorem 18]{Xu2020} it is not difficult to check that, if $\STEP= 2/(\mu+L)$, then prox-NEXT satisfies Assumption~\ref{assump:R_conv_z} with 
\begin{align*}
    \RTUQ = \max\Big\{\frac{\kappa-1}{\kappa+1}, \sqrt{1-(1-\rho_{2}\big({\bf W}\big))^2}\Big\} < 1,
\end{align*}
and  the norm $\|\bullet\|$ defined as 
\begin{align*}
\Vert{\bf z}\Vert^2 =
{\bf y}\hat{\bf W}^{-2}\big(({\bf I}-\hat{\bf W})^2\big)^\dagger\hat{\bf W}^{-2}{\bf y}
+ \big\Vert \hat{\bf W}^{-2}{\bf w}\big\Vert_{{\bf I} - ({\bf I} - \hat{\bf W})^2}^2.
\end{align*}
Note that $\Vert{\bf z}\Vert^2 \geq \Vert{\bf z}\Vert^2_2$.

$\bullet$ \textbf{On Assumptions \ref{assump:Lipt_A}, \ref{assump:Lipt_C}, and \ref{assump:commcost}:} Based on \eqref{eq:p_NEXT_c1}-\eqref{eq:p_NEXT_A}, the mappings $\mathcal A$ and $\mathcal C$ read
\begin{align*}
    &\mathcal A({\bf z}, {\bf c}^1, {\bf c}^2, {\bf c}^3, {\bf c}^4) = \left[
\begin{array}{*{20}c}
{\bf y} + ({\bf I} - \hat{\bf W}) {\bf c}^{4}  \\
{\bf c}^{3}
\end{array}
\right], \\
&\mathcal C^1({\bf z}, {\bf 0}){=}{\rm prox}_{\STEP{\bf I}, r}({\bf w}) {-}\STEP \nabla f({\rm prox}_{\STEP{\bf I}, r}({\bf w})),\  
\mathcal C^2({\bf z}, {\bf c}){=}\hat{\bf W}{\bf c}, \\
&\mathcal C^3({\bf z}, {\bf c}) = \hat{\bf W}{\bf c} - {\bf y}, \quad \text{and} \quad
\mathcal C^4({\bf z}, {\bf c}) = ({\bf I} - \hat{\bf W}){\bf c},
\end{align*}
respectively; and $\mathcal Z = \mathrm{span}({\bf I}-\hat{\bf W}) \times \mathbb R^{\m d}$.
Note that
\begin{align*}
& \Vert{\mathcal A}({\bf z}, {\bf c}^1, {\bf c}^2,{\bf c}, {\bf c}^4 )
- {\mathcal A}({\bf z}, {\bf c}^1, {\bf c}^2,{\bf c}', {\bf c}^4)\Vert^2 \\ 
&= ({\bf c}-{\bf c}')^\top\hat{\bf W}^{-2}[2\hat{\bf W}^{-1}-{\bf I}] ({\bf c}-{\bf c}')\\&
\leq \nu^{-2}(2\nu^{-1}-1) \Vert {\bf c}-{\bf c}'\Vert_2^2
\leq
\nu^{-4}\Vert {\bf c}-{\bf c}'\Vert_2^2
\end{align*}
and
\begin{align*}
& \big\Vert{\mathcal A}\left({\bf z}, {\bf c}_{1}, {\bf c}_{2}, {\bf c}_{3}, {\bf c}\right) 
- {\mathcal A}\left({\bf z}, {\bf c}_{1}, {\bf c}_{2}, {\bf c}_{3}, {\bf c}'\right)  \big\Vert 
\\&
=  \Vert \hat{\bf W}^{-2}({\bf c}-{\bf c}')\Vert_2 
\leq \nu^{-2}\Vert {\bf c}-{\bf c}'\Vert_2.
\end{align*}
Moreover, ${\mathcal A}({\bf z}, {\bf c}^1, {\bf c}^2, {\bf c}^3, {\bf c}^4 )$ is constant with respect to ${\bf c}^1, {\bf c}^2$.
Therefore,  Assumption \ref{assump:Lipt_A} holds with $L_A =\nu^{-2}$.

We next derive $L_C$ and $L_Z$. Using the non-expansive property of the proximal operator, it follows that
\begin{align*}
&\Vert \mathcal C^1({\bf z},{\bf 0})- \mathcal C^1({\bf z}',{\bf 0}) \Vert_2 = \Vert ({\rm prox}_{\STEP{\bf I}, r}({\bf w}) - {\rm prox}_{\STEP{\bf I}, r}({\bf w}'))\\& \quad- \STEP (\nabla f({\rm prox}_{\STEP{\bf I}, r}({\bf w}))-\nabla f({\rm prox}_{\STEP{\bf I}, r}({\bf w}')))\Vert_2 \\
&\leq (1+\STEP L) \Vert {\rm prox}_{\STEP{\bf I}, r}({\bf w})-{\rm prox}_{\STEP{\bf I}, r}({\bf w}')\Vert_2
\\&
\leq (1+\STEP L) \Vert {\bf w}-{\bf w}'\Vert_2
\leq (1+\STEP L) \Vert {\bf z}-{\bf z}'\Vert_2, \\
&\Vert \mathcal C^2({\bf c},{\bf z})- \mathcal C^2({\bf c}',{\bf z}')\Vert_2
= \Vert \hat{\bf W} ({\bf c}-{\bf c}')\Vert_2
\leq \Vert {\bf c}-{\bf c}'\Vert_2, \\
&\Vert \mathcal C^3({\bf c},{\bf z})- \mathcal C^3({\bf c}',{\bf z}')\Vert_2
\leq \Vert \hat{\bf W} ({\bf c}-{\bf c}')\Vert_2 + \Vert {\bf y}-{\bf y}'\Vert_2 
\\&
\qquad\leq  \Vert {\bf c}-{\bf c}'\Vert_2 + \Vert {\bf z}-{\bf z}'\Vert_2, \\
&\Vert \mathcal C^4({\bf c},{\bf z})- \mathcal C^4({\bf c}',{\bf z}')\Vert_2
= \Vert ({\bf I}-\hat{\bf W}) ({\bf c}-{\bf c}')\Vert_2
\leq \Vert {\bf c}-{\bf c}'\Vert_2,
\end{align*}
which implies that Assumption \ref{assump:Lipt_C} holds with $L_C = 1$ and $L_Z = 1+\STEP L$. Since $\STEP\leq 2/L$, it follows that $L_Z = \mathcal O (1)$. For the initial conditions,
using the fixed point ${\bf x}^\infty={\rm prox}_{\STEP{\bf I}, r}({\bf w}^\infty)$,
 we have 
\begin{align*}
&\Vert{\mathcal C}^{1}({\bf z}^{0}, {\bf 0})\Vert_2\leq
\Vert{\rm prox}_{\STEP{\bf I}, r}({\bf w}^0)-{\rm prox}_{\STEP{\bf I}, r}({\bf w}^\infty)+{\bf x}^\infty\Vert_2
\\&
{+}\STEP \Vert\nabla f({\rm prox}_{\STEP{\bf I}, r}({\bf w}^0))
{-}\nabla f({\rm prox}_{\STEP{\bf I}, r}({\bf w}^\infty))
{+}\nabla f({\bf x}^\infty)
\Vert_2
\\&
{\leq}
(1{+}\STEP L)\Vert{\bf w}^0{-}{\bf w}^\infty\Vert_2
{+}\Vert{\bf x}^\infty\Vert_2
{+}\STEP L\Vert{\bf x}^\infty{-}\tilde{\bf x}^*\Vert_2
{=}\mathcal O(\sqrt{\m d})
\end{align*}
 and 
 $\Vert{\mathcal C}^{2}({\bf z}^{0}, {\bf 0})\Vert_2=\Vert{\mathcal C}^{3}({\bf z}^{0}, {\bf 0})\Vert_2=\Vert{\mathcal C}^{4}({\bf z}^{0}, {\bf 0})\Vert_2 =0$.

Therefore, $L_A\cdot L_Z=\mathcal O(1),L_C = \mathcal O(1)$, $\Vert{\mathcal C}^{i}({\bf z}^{0}, {\bf 0})\Vert_2=\mathcal O(L_Z\sqrt{\m d}),\forall i=1,\dots, 4$; hence Assumption \ref{assump:commcost} holds.

\subsection{(Prox-)DIGing \cite{Xu2020}}
\label{subsec:ABC_DIGing}
The update of prox-DIGing solving \eqref{eq:P}, reads
\begin{align*}
{\bf x}^{k} &= {\rm prox}_{\STEP{\bf I}, r}({\bf w}^{k}), \\
{\bf w}^{k+1} &= \hat{\bf W}^2{\bf x}^{k}-\STEP \nabla f({\bf x}^{k})-{\bf y}^{k},\\
{\bf y}^{k+1} &= {\bf y}^{k}+({\bf I} - \hat{\bf W})^2{\bf w}^{k+1}, 
\end{align*}
with ${\bf y}^0 = {\bf 0}$ and ${\bf w}^0 \in \mathbb R^{\m d}$.

Prox-DIGing can be cast as \eqref{eq:M} with $\q=4$ rounds of communications, using the following definitions:
\begin{align}
&{\bf z}^\top=[{\bf y}^\top,{\bf w}^\top], \nonumber\\
&\hat{\bf c}_{i}^{k,1} = 
\mathcal C_{i}^1\big({\bf z}_i^k, {\bf 0}\big) 
= {\rm prox}_{\STEP{\bf I}, r}({\bf w}_i^{k}), \label{eq:p_DIGing_c1} \\
&\hat{\bf c}_{i}^{k,2} = 
\mathcal C_{i}^2\big({\bf z}_i^k, \hat{\bf c}_{\mathcal N_i}^{k,1}\big)
= \sum_{j \in \mathcal N_i}w_{ij}\hat{\bf c}_{j}^{k,1}, \label{eq:p_DIGing_c2} \\
&\hat{\bf c}_{i}^{k,3} = 
\mathcal C_{i}^3\big({\bf z}_i^k, \hat{\bf c}_{\mathcal N_i}^{k,2}\big) \nonumber\\&\qquad
= \sum_{j \in \mathcal N_i}w_{ij}\hat{\bf c}_{j}^{k,2} - \STEP \nabla f_i({\rm prox}_{\STEP{\bf I}, r}({\bf w}_i^{k})) - {\bf y}_i^{k}, \label{eq:p_DIGing_c3} \\
&\hat{\bf c}_{i}^{k,4} = 
\mathcal C_{i}^4\big({\bf z}_i^k, \hat{\bf c}_{\mathcal N_i}^{k,3}\big)
=\sum_{j \in \mathcal N_i}w_{ij}\big(\hat{\bf c}_{i}^{k,3} - \hat{\bf c}_{j}^{k,3}\big),\label{eq:p_DIGing_c4} \\
&{\bf z}_i^{k+1} 
= \mathcal A_i\big(
{\bf z}_i^{k}, \hat{\bf c}_{\mathcal N_i}^{k,1},\hat{\bf c}_{\mathcal N_i}^{k,2}, \hat{\bf c}_{\mathcal N_i}^{k,3},\hat{\bf c}_{\mathcal N_i}^{k,4} \big) 
\nonumber\\&\qquad
= \left[
\begin{array}{*{20}c}
{\bf y}_i^{k} + \sum_{j \in \mathcal N_i} w_{ij} \big(\hat{\bf c}_{i}^{k,4} - \hat{\bf c}_{j}^{k,4}\big)  \\
\hat{\bf c}_{i}^{k,3}
\end{array}
\right]. \label{eq:p_DIGing_A}
\end{align}
We now show that the above instance of \eqref{eq:M} satisfies Assumptions \ref{assump:R_conv_z}, \ref{assump:Lipt_A}, \ref{assump:Lipt_C}, and \ref{assump:commcost}.

$\bullet$ \textbf{On Assumption 1:} Using \cite[Theorem 18]{Xu2020}
and following similar derivations as for (Prox-)EXTRA,
 it is not difficult to check that, if {$\STEP = \frac{2(\rho_{m}({\bf W}))^2}{L + \mu(\rho_{m}({\bf W}))^2}$ and
$\nu=\sqrt{\frac{\kappa}{\kappa+1}}$}, then prox-DIGing satisfies Assumption~\ref{assump:R_conv_z} with 
\begin{align*}
    &\RTUQ = \max\Big\{\frac{\kappa-(\rho_{m}({\bf W}))^2}{\kappa+(\rho_{m}({\bf W}))^2}\frac{1}{\sqrt{\rho_{m}({\bf W})(2-\rho_{m}({\bf W}))}}, \\&
    \qquad\qquad\sqrt{1-(1-\rho_{2}\big({\bf W}\big))^2}\Big\}
    \\&\leq
    \max\Big\{
\frac{\kappa}{\kappa+1}, \sqrt{1-(1-\rho_{2}\big({\bf W}\big))^2}\Big\}
     < 1,
\end{align*}
(note that $\rho_m({\bf W})\geq\nu$)
and the norm $\|\bullet\|$ defined as 
\begin{align*}
\Vert{\bf z}\Vert^2 = \frac{1}{2\nu - \nu^2}\Big({\bf y}^\top\big(({\bf I}-\hat{\bf W})^2\big)^\dagger{\bf y}+\left\Vert {\bf w}\right\Vert_{{\bf I} - ({\bf I} - \hat{\bf W})^2}^2\Big).
\end{align*}
Note that $\Vert{\bf z}\Vert^2 \geq \Vert{\bf z}\Vert^2_2$.

$\bullet$ \textbf{On Assumptions \ref{assump:Lipt_A}, \ref{assump:Lipt_C}, and \ref{assump:commcost}:} Based on \eqref{eq:p_DIGing_c1}-\eqref{eq:p_DIGing_A}, the mappings $\mathcal A$ and $\mathcal C$ read
\begin{align*}
    &\mathcal A({\bf z}, {\bf c}^1, {\bf c}^2, {\bf c}^3, {\bf c}^4) = \left[
\begin{array}{*{20}c}
{\bf y} + ({\bf I} - \hat{\bf W}) {\bf c}^{4}  \\
{\bf c}^{3}
\end{array}
\right], \\
&\mathcal C^1({\bf z}, {\bf 0}) = {\rm prox}_{\STEP{\bf I}, r}({\bf w}),\quad 
\mathcal C^2({\bf z}, {\bf c}) = \hat{\bf W}{\bf c}, \\
&\mathcal C^3({\bf z}, {\bf c}) = \hat{\bf W}{\bf c}  - \STEP \nabla f({\rm prox}_{\STEP{\bf I}, r}({\bf w})) - {\bf y}, \quad \text{and}\\&
\mathcal C^4({\bf z}, {\bf c}) = ({\bf I} - \hat{\bf W}){\bf c},
\end{align*}
respectively; and $\mathcal Z = \mathrm{span}({\bf I}-\hat{\bf W}) \times \mathbb R^{\m d}$.
Note that
\begin{align*}
 & \big\Vert{\mathcal A}\left({\bf z}, {\bf c}_{1}, {\bf c}_{2}, {\bf c}, {\bf c}^4\right) 
- {\mathcal A}\left({\bf z}, {\bf c}_{1}, {\bf c}_{2}, {\bf c}', {\bf c}^4\right) \big\Vert^2 
\\&=  \frac{1}{2\nu - \nu^2}({\bf c}-{\bf c}')^\top [{\bf I}-({\bf I}-\hat{\bf W})^2] ({\bf c}-{\bf c}')
\leq \frac{\Vert {\bf c}-{\bf c}'\Vert_2^2}{2\nu - \nu^2} ,
\end{align*}
and 
\begin{align*}
\big\Vert{\mathcal A}\left({\bf z}, {\bf c}_{1}, {\bf c}_{2}, {\bf c}_{3}, {\bf c}\right) 
- {\mathcal A}\left({\bf z}, {\bf c}_{1}, {\bf c}_{2}, {\bf c}_{3}, {\bf c}'\right)  \big\Vert^2  
=  \frac{\Vert {\bf c}-{\bf c}'\Vert_2^2}{2\nu - \nu^2}.
\end{align*}
Moreover, ${\mathcal A}({\bf z}, {\bf c}^1, {\bf c}^2, {\bf c}^3, {\bf c}^4 )$ is constant with respect to ${\bf c}^1, {\bf c}^2$.
Therefore,  Assumption \ref{assump:Lipt_A} holds with $L_A{=}1/\sqrt{\nu}{\leq}2^{1/4}$.

We next derive $L_C$ and $L_Z$. We have 
\begin{align*}
&\Vert \mathcal C^1({\bf z}, {\bf 0})- \mathcal C^1({\bf z}', {\bf 0}) \Vert_2
\leq \Vert {\bf w}-{\bf w}'\Vert_2\leq \Vert {\bf z}-{\bf z}'\Vert_2, \\
&\Vert \mathcal C^2({\bf z},{\bf c})- \mathcal C^2({\bf z}',{\bf c}')\Vert_2
= \Vert \hat{\bf W} ({\bf c}-{\bf c}')\Vert_2
\leq  \Vert {\bf c}-{\bf c}'\Vert_2, \\
&\Vert \mathcal C^3({\bf z},{\bf c})- \mathcal C^3({\bf z}',{\bf c}')\Vert_2
\leq \Vert {\bf c}-{\bf c}'\Vert_2
+\STEP L\Vert {\bf w}-{\bf w}'\Vert_2
\\&
\quad + \Vert  {\bf y}-{\bf y}'\Vert_2
\leq \Vert {\bf c}-{\bf c}'\Vert_2+\sqrt{1+(\STEP L)^2}\Vert  {\bf z}-{\bf z}'\Vert_2
, \\
&\Vert \mathcal C^4({\bf z},{\bf c})- \mathcal C^4({\bf z}',{\bf c}')\Vert_2
= \Vert ({\bf I}-\hat{\bf W})({\bf c}-{\bf c}')\Vert_2
\leq \Vert {\bf c}-{\bf c}'\Vert_2,
\end{align*}
which implies that Assumption \ref{assump:Lipt_C} holds with $L_C = 1$ and $L_Z =\sqrt{1+(\STEP L)^2}$.  Since $\STEP\leq 2/L$, it follows that $L_Z = \mathcal O(1)$. For the initial conditions,
using ${\rm prox}_{\STEP{\bf I}, r}({\bf w}^\infty)={\bf x}^\infty$ for the fixed point,
 we have 
\begin{align*}
&\Vert{\mathcal C}^{1}({\bf z}^{0}, {\bf 0})\Vert_2 \leq \Vert {\bf w}^0 - {\bf w}^\infty\Vert_2 + \Vert {\bf x}^\infty\Vert_2 = \mathcal O(\sqrt{\m d}),\\
&\Vert{\mathcal C}^{3}({\bf z}^{0}, {\bf 0})\Vert_2 \leq \STEP L (\Vert {\bf w}^0{-}{\bf w}^\infty\Vert_2{+}\Vert {\bf x}^\infty{-}\tilde{\bf x}^*\Vert_2) 
 = \mathcal O(\sqrt{\m d}),
\end{align*}
and $\Vert{\mathcal C}^{2}({\bf z}^{0}, {\bf 0})\Vert_2=\Vert{\mathcal C}^{4}({\bf z}^{0}, {\bf 0})\Vert_2= 0$.
Therefore, $L_A\cdot L_Z=\mathcal O(1),L_C = \mathcal O(1)$, $\Vert{\mathcal C}^{i}({\bf z}^{0}, {\bf 0})\Vert_2= \mathcal O(L_Z\sqrt{\m d}),\forall i=1,\dots, 4$; hence Assumption \ref{assump:commcost} holds.

\subsection{Primal-Dual algorithm \cite{Magnusson2020, Uribe2021}}
\label{subsec:Dual}
Let ${\bf L} = (l_{ij})_{i,j=1}^m$ be the Laplacian matrix associated with the 0-1 adjacency matrix of $\mathcal G$, i.e., $l_{ii} = \vert \mathcal N_i \setminus \{i\}\vert,   i\in [m]$; and $l_{ij} = -\mathbbm{1}\{(i,j) \in \mathcal E\},   i\neq j\in  [m] $; and $\hat{\bf L} = {\bf L} \otimes {\bf I}_d$. 

The Primal-Dual algorithm solving \eqref{eq:P} with $r \equiv 0$ reads \cite{Magnusson2020, Uribe2021}
\begin{align*}
{\bf x}_i^{k} &= \underset{{\bf x}_i}{\arg\min}\,\, f_i({\bf x}_i) + {\bf x}_i^\top{\bf y}_i^{k},\\
{\bf y}_i^{k+1} &= {\bf y}_i^{k} + \STEP \sum_{j\in\mathcal N_i} l_{ij}{\bf x}_j^{k},
\end{align*}
with ${\bf y}_i^{0} = {\bf 0}$.

The Primal-Dual algorithm can be cast in the form  \eqref{eq:M}, with $\q=1$ round of communications, using  the  following:
\begin{align}
&{\bf z}={\bf y}, \nonumber \\
&\mathcal C_{i}^1\big({\bf z}_i^{k}, {\bf 0}\big) 
= \underset{{\bf x}_i}{\arg\min}\,\, f_i({\bf x}_i) + {\bf x}_i^\top{\bf y}_i^{k}, \label{eq:dual_c1} \\
&{\bf z}_i^{k+1} 
= \mathcal A_i\big(
{\bf z}_i^{k}, \hat{\bf c}_{{\mathcal N}_i}^{k,1} \big)
= {\bf y}_i^{k} + \STEP\cdot \sum_{j\in\mathcal N_i} l_{ij}\hat{\bf c}_{j}^{k,1}. \label{eq:dual_A}
\end{align}
We now show that the above instance of \eqref{eq:M} satisfies Assumptions \ref{assump:R_conv_z}, \ref{assump:Lipt_A}, \ref{assump:Lipt_C}, and \ref{assump:commcost}.

$\bullet$ \textbf{On Assumption 1:} Define ${\bf M} = \sqrt{\bf \Sigma}{\bf Q}$, where $\hat{\bf L} = {\bf Q}^\top {\bf \Sigma}{\bf Q}$ is the eigenvalue decomposition of $\hat{\bf L}$, with ${\bf \Sigma}$ being diagonal with elements sorted in descending order; and let  $\bar{\bf M}$ be the matrix containing the non-zero rows of ${\bf M}$. Using \cite{Magnusson2020} it is not difficult to check that, if $\STEP = \frac{2L\mu}{\mu\rho_{m-1}({\bf L}) + L\rho_{1}({\bf L})}$, the Primal-Dual algorithm satisfies Assumption \ref{assump:R_conv_z} with 
\begin{align*}
\RTUQ = \dfrac{\frac{\rho_1({\bf L})}{\rho_{m-1}({\bf L})}-\frac{1}{\kappa}}{\frac{\rho_1({\bf L})}{\rho_{m-1}({\bf L})}+\frac{1}{\kappa}} < 1,
\end{align*}
and  the norm $\|\bullet\|$ defined as 
\begin{align*}
\Vert{\bf z}\Vert = \sqrt{\rho_1({\bf L})}\left\Vert \left(\bar{\bf M}\bar{\bf M}^\top\right)^{-1}\bar{\bf M} {\bf z}\right\Vert_{2}.
\end{align*}
Note that $\Vert{\bf z}\Vert\geq \Vert{\bf z}\Vert_2$.

$\bullet$ \textbf{On Assumptions \ref{assump:Lipt_A}, \ref{assump:Lipt_C}, and \ref{assump:commcost}:} Based on \eqref{eq:dual_c1} and \eqref{eq:dual_A}, the mappings $\mathcal A$ and $\mathcal C$ read
\begin{align*}
    \mathcal A({\bf z}, {\bf c}) {=} {\bf z}{+} \STEP\cdot\hat{\bf L}{\bf c}, \ 
\mathcal C^1({\bf z}, {\bf 0}) {=} \left[ 
\begin{array}{c}
     \!\!\underset{{\bf x}}{\arg\min}f_1({\bf x}){+}{\bf x}^\top{\bf z}_1\!\!  \\
     \vdots \\
     \!\!\underset{{\bf x}}{\arg\min}f_{\m}({\bf x}){+}{\bf x}^\top{\bf z}_{\m}\!\!
\end{array}
\right]\!\!,
\end{align*}
respectively; and $\mathcal Z = \mathrm{span}({\bf L})$.  Note that
\begin{align*}
\mathcal A\left({\bf z}, {\bf c}\right) - \mathcal A\left({\bf z}, {\bf c}'\right)
= \STEP \hat{\bf L}\left({\bf c} - {\bf c}'\right).
\end{align*}
It follows that
\begin{align*}
&\big\Vert\mathcal A\left({\bf z}, {\bf c}\right){-}\mathcal A\left({\bf z}, {\bf c}'\right)\big\Vert
{=}\STEP\sqrt{ \rho_1({\bf L}) }\left\Vert\bar{\bf M}\left({\bf c}{-}{\bf c}'\right)\right\Vert_2 \\
&\leq  \STEP\rho_1\left({\bf L}\right)\left\Vert{\bf c} - {\bf c}'\right\Vert_2,
\end{align*}
which implies that Assumption \ref{assump:Lipt_A} holds with $L_A =\STEP\rho_1({\bf L})$. 
We now derive $L_C$ and $L_Z$. Since
$\mathcal C_{i}^1\big({\bf z}_i, {\bf 0}\big) = \underset{{\bf x}_i}{\arg\min}\,\, f_i({\bf x}_i) + {\bf x}_i^\top{\bf z}_i$,
it follows that
${\bf z}_i=-\nabla f_i(\mathcal C_{i}^1\big({\bf z}_i, {\bf 0}\big))$,
and strong convexity of $f_i$ implies
$$\Vert{\bf z}_i'-{\bf z}_i\Vert_2\geq \mu
\Vert\mathcal C_{i}^1\big({\bf z}_i', {\bf 0}\big)-\mathcal C_{i}^1\big({\bf z}_i, {\bf 0}\big)\Vert_2.$$
It readily follows that
 Assumption \ref{assump:Lipt_C} holds with $L_C = 0$ and $L_Z =1/\mu$. For the initial conditions, since ${\bf z}^0 = {\bf 0}$ we have 
 $ \mathcal C^1({\bf z}^0, {\bf 0})=\tilde{\bf x}^*$ and 
$\Vert \mathcal C^1({\bf z}^0, {\bf 0})\Vert_2 = \Vert \tilde{\bf x}^*\Vert_2 = \mathcal O(\sqrt{m d})$.

From the expression of $\gamma$, it is straightforward to see that
$\STEP \leq \frac{2\mu}{\rho_{1}({\bf L})}$, so that $L_A\cdot L_Z=\mathcal O(1),L_C = \mathcal O(1)$.
Furthermore, for all the objective functions of \eqref{eq:P} such that  $\mu = \mathcal O(1)$, we also have 
 $\Vert{\mathcal C}^{1}({\bf z}^{0}, {\bf 0})\Vert_2=\mathcal O(L_Z\sqrt{\m d})$,
 hence Assumption~\ref{assump:commcost} holds.
 For instance, this is the typical case in  machine learning problems where a regularization $\mu/2 \|\mathbf{x}\|^2$ is added to the objective function to enforce strong convexity, with $\mu = \mathcal O(1)$.

\begin{IEEEbiography}[{\includegraphics[width=1in,height=1.25in,clip,keepaspectratio]{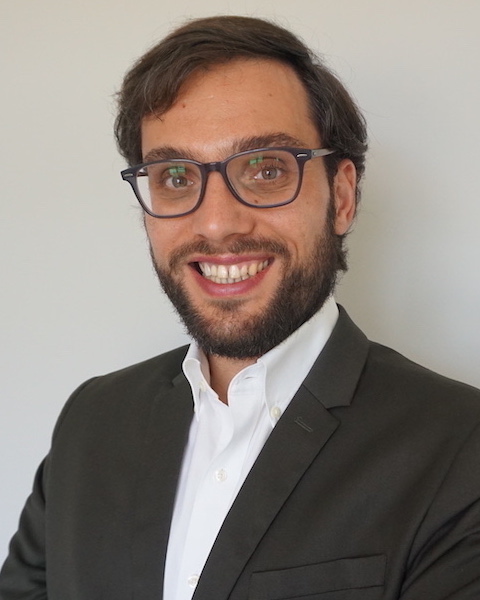}}]{Nicol\`{o} Michelusi}
(Senior Member, IEEE) received the B.Sc. (with honors), M.Sc. (with honors), and Ph.D. degrees from the University of Padova, Italy, in 2006, 2009, and 2013, respectively, and the M.Sc. degree in telecommunications engineering from the Technical University of Denmark, Denmark, in 2009, as part of the T.I.M.E. double degree program. From 2013 to 2015, he was a Postdoctoral Research Fellow with the Ming-Hsieh Department of Electrical Engineering, University of Southern California, Los Angeles, CA, USA, and from 2016 to 2020, he was an Assistant Professor with the School of Electrical and Computer Engineering, Purdue University, West Lafayette, IN, USA. He is currently an Assistant Professor with the School of Electrical, Computer and Energy Engineering, Arizona State University, Tempe, AZ, USA. His research interests include 5G wireless networks, millimeter-wave communications, stochastic optimization, distributed optimization, and federated learning over wireless. He served as Associate Editor for the IEEE TRANSACTIONS ON WIRELESS COMMUNICATIONS from 2016 to 2021, and a Reviewer for several IEEE journals. He was the Co-Chair for the Distributed Machine Learning and Fog Network workshop at the IEEE INFOCOM 2021, the Wireless Communications Symposium at the IEEE Globecom 2020, the IoT, M2M, Sensor Networks, and Ad-Hoc Networking track at the IEEE VTC 2020, and the Cognitive Computing and Networking symposium at the ICNC 2018. He received the NSF CAREER award in 2021.
\end{IEEEbiography}

\begin{IEEEbiography}[{\includegraphics[width=1.1in,keepaspectratio]{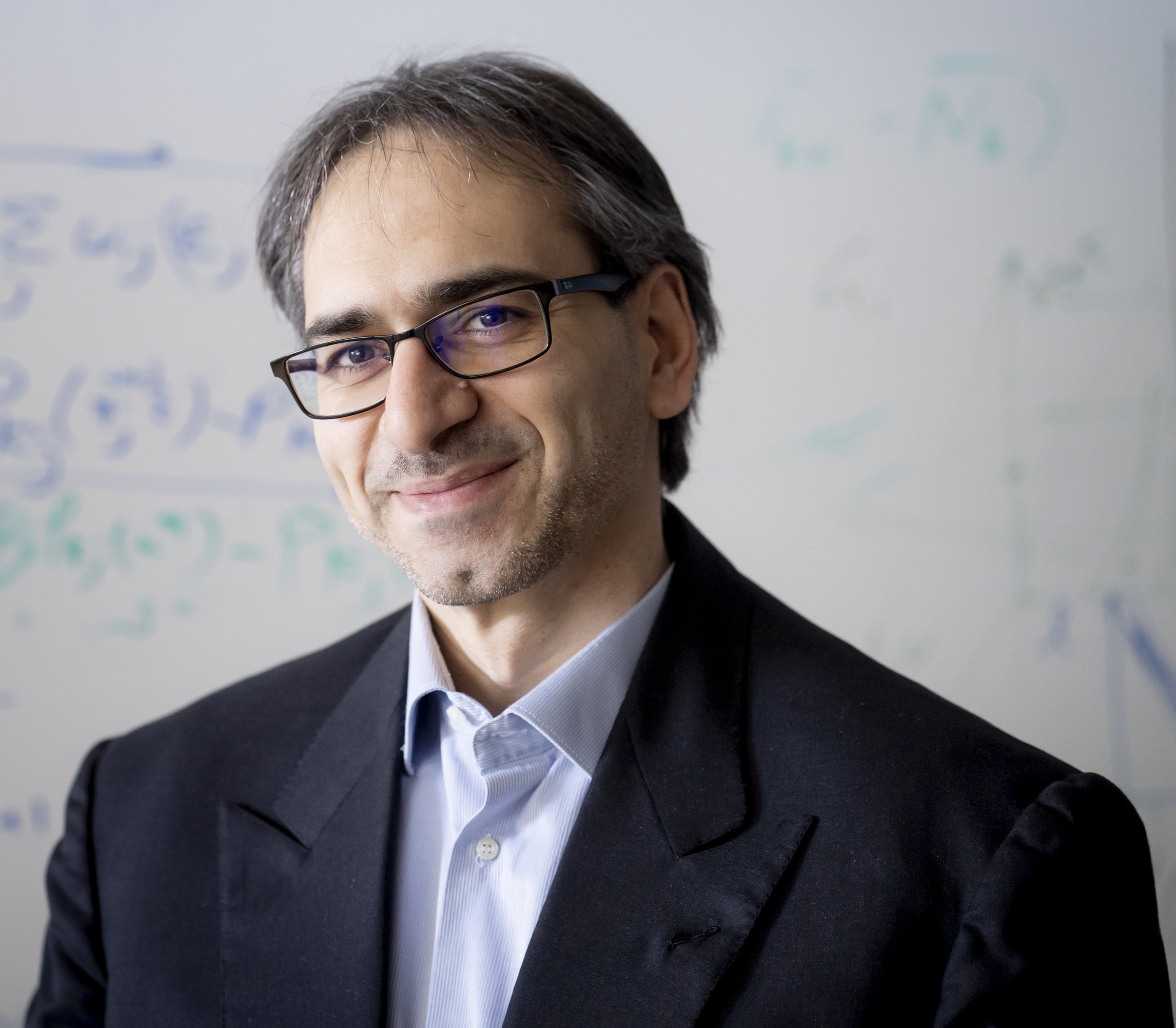}}]{Gesualdo Scutari} (Fellow, IEEE) received the Elec- trical Engineering and Ph.D. degrees (both with Hons.) from the University of Rome ``La Sapienza'' Rome, Italy, in 2001 and 2005, respectively. He is a Professor with the School of Industrial Engineering, Purdue University, West Lafayette, IN, USA. His research interests include continuous and distributed optimization, equilibrium programming, and their applications to signal processing and machine learning. He is a Senior Area Editor of the IEEE TRANSACTIONS ON SIGNAL PROCESSING and an Associate Editor of SIAM Journal on Optimization, Among others, he was the recipient of the 2013 NSF CAREER Award, the 2015 IEEE Signal Processing Society Young Author Best Paper Award, and the 2020 IEEE Signal Processing Society Best Paper Award.\end{IEEEbiography}

\begin{IEEEbiography}[{\includegraphics[width=1in,keepaspectratio]{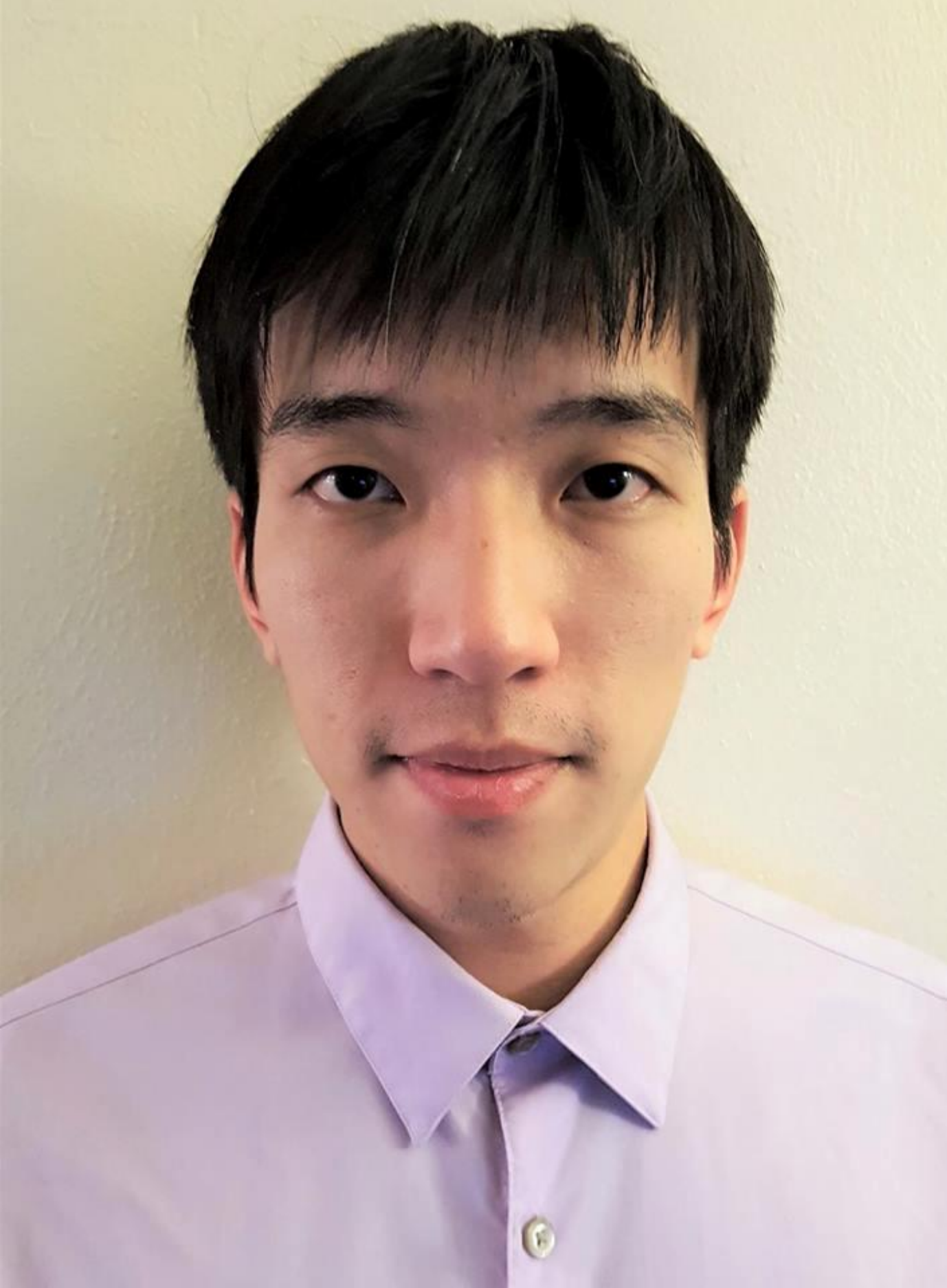}}]{Chang-Shen Lee}  Chang-Shen Lee received the B.Sc. degree in electrical engineering and computer science and the M.Sc. degree in communications engineering from National Chiao Tung University, Hsinchu, Taiwan, in 2012 and 2014, respectively. He received the Ph.D. degree in electrical and computer engineering from Purdue University, West Lafayette, IN, USA, in 2021. He is currently a senior software engineer at Bloomberg, New York, USA.\end{IEEEbiography}

  \end{document}